\documentclass[11pt,reqno]{amsart}

\usepackage{amsthm}
\usepackage{amssymb}
\usepackage{latexsym}
\usepackage{multicol}
\usepackage{verbatim,enumerate}
\usepackage{accents}
\usepackage[usenames]{color}
\usepackage{dirtytalk}

\usepackage{hyperref}
\usepackage{hyperref}
\usepackage{amsmath, amscd}
\usepackage{soul}
\usepackage{tikz}
\usetikzlibrary{matrix,arrows,decorations.pathmorphing}
\usetikzlibrary{positioning}

\advance\textwidth by 1.2in \advance\oddsidemargin by -.6in \advance\evensidemargin by -.6in

\theoremstyle{plain}
\newtheorem{thm}{Theorem}[subsection]
\newtheorem{prop}[thm]{Proposition}
\newtheorem{lem}[thm]{Lemma}
\newtheorem{cor}[thm]{Corollary}

\theoremstyle{definition}
\newtheorem{example}[thm]{Example}
\newtheorem{defn}[thm]{Definition}

\theoremstyle{remark}
\newtheorem{rem}[thm]{Remark}

\newcommand{\lie}[1]{\mathfrak{#1}}
\newcommand{\wh}[1]{\widehat{#1}}

\newcommand\bc{\mathbb C}
\newcommand\bn{\mathbb N}
\newcommand\bz{\mathbb Z}

\newcommand{\Gr}{\mathrm{Gr}}

\def\a{\alpha}

\newenvironment{pf}{\proof}{\endproof}
\newcounter{cnt}
 
\makeatletter
\def\mydggeometry{\makeatletter\dg@YGRID=1\dg@XGRID=20\unitlength=0.003pt\makeatother}
\makeatother 

\numberwithin{equation}{section}
\makeatletter
\def\section{\def\@secnumfont{\mdseries}\@startsection{section}{1}%
  \z@{.7\linespacing\@plus\linespacing}{.5\linespacing}%
  {\normalfont\scshape\centering}}
\def\subsection{\def\@secnumfont{\bfseries}\@startsection{subsection}{2}%
  {\parindent}{.5\linespacing\@plus.7\linespacing}{-.5em}%
  {\normalfont\bfseries}}
\makeatother

\begin{document}


\title{Maximal closed subroot systems of real affine root systems}

\author{Krishanu Roy}
\thanks{}
\address{HBNI, The Institute of Mathematical Sciences, Chennai, India}
\email{krishanur@imsc.res.in}
\author{R. Venkatesh}
\thanks{RV is partially supported by the ``MAT/16-17/047/DSTX/RVEN''}
\address{Indian Institute of Science, Bangalore, India}
\email{rvenkat@iisc.ac.in}

\begin{abstract}
We completely classify and give explicit descriptions of all maximal closed subroot systems of real affine root systems. 
As an application we describe a procedure to get the classification of all
regular subalgebras of affine Kac--Moody algebras in
terms of their root systems.
\end{abstract}

\maketitle
\section{Introduction}

Given a finite crystallographic root system one can naturally ask for the list of all its subroot systems.
In \cite{B&S}, A. Borel and J. De Siebenthal gave a partial answer to this question;
they classified all the maximal closed subroot systems of finite crystallographic root systems.
A subroot system of a finite root system is said to be closed
if it is closed under addition with respect to the ambient finite root system. Naturally one can extend this definition to any real affine root systems 
(i.e., the real roots of affine Kac--Moody algebras). We simply call the real affine root systems as affine root systems throughout this paper. 
More precisely, let $\lie g$ be an affine Kac--Moody algebra and let $\Phi$ be the set of its real roots. 
A subroot system $\Psi$ of $\Phi$ is called closed if for each $\alpha$, $\beta \in \Psi$, such that $\alpha+\beta \in \Phi$, we have $\alpha+\beta \in \Psi$.
The main goal of this paper is to give a complete classification of all maximal closed subroot systems of affine root system $\Phi$ in both untwisted and twisted cases. 

One of our main motivations comes from the work of M. J. Dyer and G. I. Lehrer \cite{Dyer} (see also \cite{DyerTAMS}), where they developed some new ideas to classify all the
subroot systems of untwisted affine root systems, or more generally the subroot systems of real root systems of loop algebras of Kac--Moody algebras. The subroot systems of affine root systems are characterized in terms of certain explicit
compatible cosets of $\mathbb{Z}$ and the subroot systems of the 
underlying (gradient) finite  root systems, see Section 
\ref{characterization}
for more details. Using this characterization, we give very explicit descriptions of all maximal closed subroot systems of affine root systems. 
Our main theorem is indeed a corollary of the results of \cite{Dyer} for the untwisted case, but it is not stated as a corollary of their result anywhere  as we know. Some difficulties naturally arise
when we consider the twisted case. For example, the gradient root system of a proper closed subroot system of a twisted affine root system does not need to be a closed subroot system,
see Proposition \ref{twistedgradient}. 
Indeed this is the precise fact that makes it harder to deal with the twisted case, see Section \ref{twistedgradient} for more details. We need a case-by-case analysis
when the gradient root system of a maximal closed affine subroot system is not closed. Table \ref{mrs} in Section \ref{tabeluntwisted}
and Table \ref{mrs2} in Section \ref{tabletwisted} contain the complete list of all maximal closed subroot systems of untwisted and twisted affine algebras respectively.

The classification of closed subroot systems is very essential in the classification of semi-simple subalgebras of semi-simple Lie algebras. 
In \cite{Dynkin}, E. B. Dynkin introduced the notion of regular semi-simple subalgebras 
in order to classify all the semi-simple subalgebras of finite dimensional complex semi-simple Lie algebras.
He classified regular semi-simple subalgebras in terms of their root systems, which are closed subroot systems of the root system of the ambient Lie algebra.
One can define regular subalgebras in the context of affine Kac--Moody algebras by generalizing the definition of regular semi-simple subalgebras. 
A subalgebra of the affine Kac--Moody algebra $\lie g$ is said to be a regular subalgebra if there exists a closed subroot system $\Psi$ of $\Phi$ 
such that it is generated as a Lie subalgebra by the root spaces $\lie g_\a$ for $\a\in \Psi$. See Section \ref{equivdefregular} for some characterizations of regular subalgebras.
The regular subalgebra defined by $\Psi$ is uniquely determined by 
$\Psi$ and conversely $\Psi$ is also uniquely determined the regular subalgebra defined by $\Psi$, see Section \ref{closed} for more details.
Hence, classifying all the regular subalgebras is same as classifying all the closed subroot systems of $\Phi.$  In this paper we describe a procedure to classify
all the closed subroot systems using the information about maximal closed subroot systems. As a by-product we get a complete classification of all regular subalgebras of affine Kac--Moody algebras. 

Our second motivation for this work comes from the work of A. Felikson, A. Retakh and P. Tumarkin \cite{FRT}, where they described a procedure to classify all 
the regular subalgebras of affine Kac--Moody algebras. They determine all possible maximal closed affine type subroot systems in terms of their Weyl group in order to classify all the regular subalgebras. 
It appears that some maximal subroot systems were omitted in their classification list. For example, the root system of type $\tt A_{2}^{(1)} \oplus \tt A_{2}^{(1)}$ appears as a decomposable maximal closed subroot system
of $\tt E_6^{(2)}$ and the root system of type $\tt D_{5}^{(2)}$ appears as an  indecomposable maximal closed subroot system
of $\tt E_6^{(2)}$,  which were omitted in their list, see \cite[Table 2]{FRT} and Table  \ref{mrs7}. 
See the Remark \ref{differences} for the complete list of differences between our
classification list and their classification list.
Moreover our approach is completely different from their approach.

\vskip 2.0mm

The paper is organized as follows. In Sections~\ref{prelim} and \ref{characterization},  we give a short description of affine root systems, 
set up the notations and introduce some basic 
theory which will be used throughout this paper.  
In Section~\ref{untwi}, we give a complete description of all maximal closed subroot systems of untwisted affine root systems.
In Section~\ref{twistedcase}, we give a complete description of the maximal closed subroot systems of twisted affine root systems for which
 the corresponding gradient subroot system is either
 equal to 
the gradient root system of the original twisted affine root system or it is a proper closed subroot system.
Section~\ref{D2n} (resp. Section \ref{2A2n-1}, \ref{3D4}, \ref{2E6}) contains the classification of all maximal closed subroot systems of $\tt D_{n+1}^{(2)}$ 
(resp. $\tt A_{2n-1}^{(2)}$, $\tt D_4^{(3)}$, $\tt E_6^{(2)}$)  whose gradient subroot system is a proper semi-closed subroot system.
The case $\tt A_{2n}^{(2)}$ is treated separately in Section~\ref{2A2n} for $n\ge 2$ and the case 
$\tt A_{2}^{(2)}$ is treated separately in Section~\ref{2A2}. In Section~\ref{closed}, we describe a procedure to classify all the regular subalgebras of affine Kac--Moody
 algebras.

\vskip 2.0mm

\textit{Acknowledgement: The first author is very grateful to K. N. Raghavan and S. Viswanath 
for many useful discussions
and encouragement. Part of this work was done when the first author visited the Indian Institute of Science, Bangalore. 
He thanks G. Misra and Indian Institute of Science for their hospitality.
The second author acknowledges the hospitality and
excellent working conditions at The Institute of Mathematical Sciences, Chennai where part of this
work was done. Both authors thank P. Tumarkin for many helpful discussions.
Both authors are very grateful to the anonymous referees for their generous comments on the previous
manuscript. Their comments helped very much to improve the readability of this manuscript.
}

\section{Preliminaries}\label{prelim}

\subsection{}\label{defaffine} We denote the set of complex numbers by $\bc$ and, respectively, the set of integers, non--negative integers, and positive integers  by $\bz$, $\bz_+$, and $\bn$.
 
We refer to \cite{K90} for the general theory of affine Lie algebras and we refer  to \cite{Bor, MacD} for the general theory of affine root systems.
Throughout, $A$ will denote an indecomposable affine Cartan matrix,  and $S$  will denote the corresponding Dynkin diagram with the labeling of vertices as 
in Table Aff2 from \cite[pg.54--55]{K90}. Let $\mathring{S}$ be the Dynkin diagram obtained from 
$S$ by dropping the zero node and let $\mathring{A}$ be the Cartan matrix, whose Dynkin diagram is $\mathring{S}$.

Let $\lie g$ and $\mathring{\lie g}$  be the  affine Lie algebra and the finite--dimensional simple Lie algebra associated to $A$ and $\mathring{A}$ over $\mathbb{C}$, respectively. 
We shall realize $\mathring{\lie g}$ as a subalgebra of $\lie g$.
We fix $\mathring{\lie h}\subseteq\lie h$ Cartan subalgebras of $\mathring{\lie g}$ and respectively $\lie g$. Then we have 
$$\mathfrak{h}=\mathring{\mathfrak{h}}\oplus \mathbb{C}K\oplus \mathbb{C}d. $$
where $K$ is the canonical central element, and $d$ is the derivation. Consider $\mathring{\mathfrak{h}}^*$ as a subspace of $\mathfrak{h}^*$
by setting $\lambda(K)=\lambda(d)=0$ for all $\lambda\in \mathring{\mathfrak{h}}^*.$ 
Let $\delta\in \mathfrak{h}$ be given by $\delta(d)=a_0$, where $a_0$ is $2$ if $\mathfrak{g}$ is of type $A_{2n}^{(2)}$ and $1$ otherwise, and $\delta(\mathring{\mathfrak{h}}\oplus \mathbb{C}K)=0.$
Let $(\ ,\ )$ be a standard symmetric non-degenerate invariant bilinear form on $\lie h^{*}$.

 \subsection{}\label{definitionofm} We denote by $\Delta(\lie g)$ the set of roots of $\lie g$ with respect to $\lie h$, and the set of real roots of $\lie g $ 
 by $\Delta_{\mathrm{re}}(\lie g) = :\Phi$ and the set of imaginary roots of $\lie g$ by $\Delta_{\mathrm{im}}(\lie g)$. {\em{ We call $\Phi$ as affine root systems in this paper. By abuse of notations, we say that  $\Phi$ is of affine type $X$ (resp. untwisted or twisted) if and only if
 $\Delta(\lie g)$ is of affine type $X$ (resp. untwisted or twisted).}}
The set of roots of $\mathring{\lie g}$ with respect to $\mathring{\lie h}$ is denoted by $\mathring{\Phi}$ and note that $\mathring{\Phi}$ can be identified as a subroot system of $\Phi.$
Let $\Phi_\ell$ and $\Phi_s$ (resp. $\mathring{\Phi}_\ell$ and $\mathring{\Phi}_s$) denote respectively the subsets of $\Phi$ (resp. $\mathring{\Phi}$) consisting of the long and short roots.
We set
$$m=\begin{cases}
1,& \text{if $\Phi$ is of untwisted type}\\
2,& \text{if $\Phi$ is of type $\tt A^{(2)}_{2n}\ (n\geq 1), \tt A^{(2)}_{2n-1}\ (n\geq 3), \tt D^{(2)}_{n+1}\ (n\geq 2)\mbox{ or }\tt E^{(2)}_{6}$}\\
3,& \text{if $\Phi$ is of type $\tt D^{(3)}_{4}$.} 
\end{cases}
$$
We have (see \cite[Page no. 83]{K90}) $\Phi= \{\alpha+r\delta: \alpha\in \mathring{\Phi}, r\in \mathbb{Z}\} \ \text{if $m=1$}$ and 
$$\Phi=  \{\alpha+r\delta: \alpha\in \mathring{\Phi}_s, r\in \mathbb{Z} \}\cup \{\alpha+mr\delta: \alpha\in \mathring{\Phi}_\ell, r\in \mathbb{Z} \}$$
if $m=2$ or $3$, but 
                $\Phi$ is not of type $\tt A^{(2)}_{2n}$ and else
$$\Phi= \{\tfrac{1}{2}(\alpha+(2r-1)\delta: \alpha\in \mathring{\Phi}_\ell, r\in \mathbb{Z} \}\cup \{\alpha+r\delta: \alpha\in \mathring{\Phi}_s, r\in \mathbb{Z} \}\cup \{\alpha+2r\delta: \alpha\in \mathring{\Phi}_\ell, r\in \mathbb{Z} \}.$$

\subsection{}\label{Weylgroup} Given $\alpha\in \Phi$, we denote by $\alpha^\vee\in\mathfrak{h}$ the coroot associated to $\alpha$.
Then we set $\langle \beta, \alpha^\vee\rangle:= \beta(\alpha^\vee)= \frac{ 2(\beta, \alpha)}{(\alpha, \alpha)}.$
Define reflections 
 $\textbf{s}_{\alpha}\colon \mathfrak{h}^{*}\rightarrow \mathfrak{h}^{*}$ for $\alpha\in \Phi$ as follows:
  $$\textbf{s}_{\alpha}(\beta)= \beta -\langle \beta, \alpha^\vee\rangle\alpha$$
where $\beta\in \mathfrak{h}^{*}$.  For $\alpha\in \mathring{\Phi}$, $\textbf{s}_{\alpha}$ restricts to the reflection in $\alpha$ on $\mathring{\mathfrak{h}}^{*}$.
  We let $W:=\{\textbf{s}_{\alpha}: \alpha\in \Phi \}$ denote the Weyl group of $\mathfrak{g}$
  and denote by $\mathring{W}:=\{\textbf{s}_{\alpha}: \alpha\in \mathring{\Phi} \}$ the Weyl group of $\mathring{\mathfrak{g}}$.

\subsection{} In this section, we recall some general definitions and facts about finite and affine root systems. 
\begin{defn}\label{subrootdefinition}
 A proper non--empty subset $\Psi$ of $\Phi$ (resp., $\mathring{\Phi}$) is called
 \begin{enumerate}
  \item a subroot system of $\Phi$
  (resp., $\mathring{\Phi}$), if $\textbf{s}_{\alpha}(\beta)\in \Psi$ for all $\alpha,\beta\in \Psi$;
  \item  closed in $\Phi$
  (resp., $\mathring{\Phi}$), $\text{if} \ \alpha, \beta\in \Psi \ \text{and} \ \alpha+\beta\in \Phi \ \text{(resp.,} \ \mathring{\Phi})\ \text{implies} \ \alpha+\beta\in \Psi$;
  \item closed subroot system of $\Phi$
  (resp., $\mathring{\Phi}$), if it is both subroot system and closed.
 \end{enumerate}

\end{defn}

 \begin{defn} 
A proper closed subroot system $\Psi$ of $\Phi$ (resp., $\mathring{\Phi}$) is said to be a maximal closed subroot system of $\Phi$ (resp., $\mathring{\Phi}$) if
$\Psi\subseteq \Delta \subsetneq \Phi$ (resp., $\mathring{\Phi}$) implies $\Delta=\Psi$ for all closed subroot system of $\Phi$ (resp., $\mathring{\Phi}$).
\end{defn}

\begin{defn}
 
 Let $\Psi\le \Phi$ be a subroot system. The gradient root system  associated with $\Psi$ is defined to be 
 $$\Gr(\Psi):=\left\{(\alpha+r\delta)|_{\mathring{\lie h}} : \alpha+r\delta\in \Psi \right\},$$
where recall that $\mathring{\lie h}$ is the Cartan subalgebra of $\mathring{\lie g}$ defined in Section \ref{defaffine}. Since $\delta|_{\mathring{\lie h}}=0$, we have $(\alpha+r\delta)|_{\mathring{\lie h}}=\alpha|_{\mathring{\lie h}}=\alpha$ for $\alpha+r\delta\in \Psi.$
 In particular we have
$$\mathrm{Gr}(\Phi)=\begin{cases}
        \mathring{\Phi}\cup \frac{1}{2}\mathring{\Phi}_\ell& \text{if $\wh{\lie g}$ is of type $\tt A^{(2)}_{2n}$}\\
        \mathring{\Phi}& \text{otherwise.}
         \end{cases}$$
         \end{defn}
         The definition of $\Gr(\Psi)$ is dependent on the ambient root system $\Phi$. But we do not want to put $\Phi$ as an additional parameter in the notation.
Note that $\Gr(\Psi)$ does not need be a reduced root system in general. For example,  $\mathrm{Gr}(\Phi)$ is non-reduced finite root system of type $\tt BC_n$ when $\lie g$ is of type $\tt A^{(2)}_{2n}$.
It is easy to see that the gradient root system associated with $\Psi$ is a subroot system of $\Gr(\Phi)$ in the sense of Definition \ref{subrootdefinition}(1). We say $\Gr(\Psi)$
is reduced if $\Gr(\Psi)$ does not contain a subroot system of type $\tt BC_r$ for any $r\ge 1.$ 
The Weyl group of $\Gr(\Psi)$ generated by $\{\textbf{s}_\alpha:\alpha\in \Gr(\Psi)\}$ is denoted by $W_{\Gr(\Psi)}.$

\begin{defn}
  Let $\Psi\le \Gr(\Phi)$ be a subroot system. The lift of $\Psi$ in $\Phi$ is defined to be 
  $$\widehat \Psi:=\bigcup\limits_{\a\in \Psi} \left\{\alpha+r\delta : \ \text{for all $r$ such that}\ \alpha+r\delta\in \Phi\right\}$$ 
  It is easy to see that the lift $\widehat \Psi$ of $\Psi$ is a subroot system of $\Phi.$
\end{defn}

\begin{defn}\label{type}
Let $\Psi$ be an irreducible subroot system of $\Phi$.
We say that $\Psi$ is of type $\tt X_n^{(r)}$ if there exists a vector space isomorphism $\varphi: \mathbb{R}\Psi \to \mathbb{R}\tt X_n^{(r)}$ such that
$$\varphi(\Psi)=\text{$\tt X_n^{(r)}$} \ \ \text{and}\ \ \langle \beta, \alpha^\vee\rangle=\langle \varphi(\beta), \varphi(\alpha^\vee)\rangle\;\; \text{for all} \ \ \alpha,\beta\in\Psi,$$
where $\mathbb{R}\Psi$ (resp., $\mathbb{R}\tt X_n^{(r)}$) denotes the vector space spanned by $\Psi$ (resp., $\tt X_n^{(r)}$) over $\mathbb{R}$.
Let $\Psi$ be a reducible subroot system of $\Phi$. We say that $\Psi$ is of type 
$\tt X_{n_1}^{(r_1)}\oplus X_{n_2}^{(r_2)}\oplus\cdots \oplus X_{n_k}^{(r_k)}$ if 
$\Psi=\Psi_1\oplus \Psi_2\oplus \cdots \oplus\Psi_k$ such that $\Psi_i$ is irreducible for all $1\le i\le k$, $\Psi_i$ are mutually orthogonal and $\Psi_i$ is of type $\tt X_{n_i}^{(r_i)}$ 
for all $1\le i\le k$.
\end{defn}

\begin{rem}
Notice that the vector space generated by the irreducible components of a reducible root system need not be direct. For example, 
consider the affine root system $\Delta$ of type $\tt G_2^{(1)}$ and its real roots $\Phi=\{\alpha+n \delta: \alpha \in \mathring{\Phi}, n \in\mathbb{Z}\}$ where
$\mathring{\Phi}$ is of type $\text{$\tt G_2$}$.
Let $\{\a_1, \a_2\}$ be the simple system of $\mathring{\Phi}$, such that $\a_2$ is a short root. 
Then define $$\Psi=\{\pm\a_2+n\delta : n\in\mathbb{Z}\}\cup\{\pm\theta+n\delta: n\in\mathbb{Z}\},$$ where $\theta$ is the long root of $\mathring{\Phi}$.
 Clearly,  $\Psi$ is a closed subroot system of type $\tt A_1^{(1)}\oplus A_1^{(1)}$ but the sum of vector spaces spanned by each component is not direct.
\end{rem}

\vskip 2mm
\noindent
The following Lemma is immediate from the above definitions.
\begin{lem}\label{closedlemma}
 Let $\Phi$ be an irreducible affine root system and let $\Gr(\Phi)$ be its corresponding gradient root system. 
 If $\Psi$ is a closed subroot system of $\Gr(\Phi)$ then the lift $\wh{\Psi}$ is also a closed subroot system of $\Phi$. 
\end{lem}

\subsection{}

We make the following conventions throughout this paper $\tt B_1=C_1=D_1=A_1$, $\tt B_2=C_2$, $\tt D_2=A_1\oplus A_1$, $\tt D_3=A_3$,
$\tt A_1^{(1)}=B_1^{(1)}=C_1^{(1)}$, $\tt B_2^{(1)}=C_2^{(1)}$, 
$\tt D_2^{(1)}=A_1^{(1)}\oplus A_1^{(1)}$, $\tt A_3^{(1)}=D_3^{(1)}$, $\tt A_1^{(2)}=A_1^{(1)}$ and $\tt A_3^{(2)}=D_3^{(2)}$.
We end this section by recalling the list of all maximal closed subroot systems of an irreducible finite crystallographic root system of rank $n$
from \cite[Page 136]{kane}. 

\begin{table}[ht]
\caption{Types of maximal closed subroot systems of irreducible finite root systems}
\centering 
\begin{tabular}{|c|c|c|}
\hline
Type & Reducible & Irreducible \\
\hline 
$\tt A_n$ & $\tt A_r \oplus A_{n-r-1}\; (0 \leq r \leq n-2)$ & $\tt A_{n-1}$  \\[1ex]
$\tt B_n$ & $\tt B_r \oplus D_{n-r}\;\;\;\; (1 \leq r \leq n-2)$ & $\tt B_{n-1}$, $\tt D_n$ \\[1ex]
$\tt C_n$ & $\tt C_r \oplus C_{n-r}\;\;\;\; (1 \leq r \leq n-1)$ & $\tt A_{n-1}$ \\[1ex]
$\tt D_n$ & $\tt D_r \oplus D_{n-r}\;\;\; (2 \leq r \leq n-2)$ & $\tt A_{n-1}$, $\tt D_{n-1}$\\[1ex]
$\tt E_6$ & $\tt A_5 \oplus A_{1}$, $\tt A_2 \oplus A_{2} \oplus A_2$ & $\tt D_{5}$ \\[1ex]
$\tt E_7$ & $\tt A_5 \oplus A_{2}$, $\tt A_1 \oplus D_{6} $ & $\tt E_{6}$, $\tt A_7$ \\[1ex]
$\tt E_8$ & $\tt A_1 \oplus E_{7}$, $\tt E_6 \oplus A_{2}, A_4 \oplus A_{4} $ & $\tt D_{8}$, $\tt A_8$\\ [1ex]
$\tt F_4$ & $\tt A_2 \oplus A_{2}$, $\tt C_3 \oplus A_1 $ & $\tt B_{4}$ \\[1ex]
$\tt G_2$ & $\tt A_1 \oplus A_{1}$ & $\tt A_{2}$ \\[1ex]

\hline
\end{tabular}
\label{mrs3}
\end{table}


\section{Characterization of closed subroot systems}\label{characterization}
We will closely follow the arguments in \cite{Dyer} (see also \cite{DyerTAMS}) to complete the classification of maximal closed subroot systems of affine root systems. 
The authors of \cite{Dyer} considered only the untwisted affine root systems or
more generally considered real root systems of loop algebras of Kac--Moody algebras  in \cite{Dyer}. Here in this paper we will deal with both untwisted and twisted affine root systems.
We leave out the proofs of most of the results presented in this section as it closely follows the arguments of \cite{Dyer}.

\subsection{}\label{dyerbasics}
Recall that $\Phi$ is the set of real roots of the indecomposable affine Kac--Moody Lie algebra $\lie g$ defined in Section \ref{defaffine}. 
Let $\Psi$ be a subroot system of $\Phi.$
Define $$Z_\alpha(\Psi)=\left\{r : \alpha+r\delta\in \Psi \right\}, \ \text{for} \ \alpha\in \mathrm{Gr}(\Psi).$$
It is easy to see that
$\Psi=\left\{\alpha+r\delta: \alpha\in \Gr(\Psi), r\in Z_\alpha(\Psi)\right\}.$ 
We immediately have (see, Lemma 8 in \cite{Dyer}) 
\begin{equation}\label{zsets}
 Z_{\beta}(\Psi)-\langle\beta,\alpha^\vee \rangle Z_{\alpha}(\Psi)\subseteq Z_{\textbf{s}_{\alpha}(\beta)}(\Psi),\ \text{for all $\alpha, \beta\in\mathrm{Gr}(\Psi).$}
\end{equation}

\begin{lem}[\cite{Dyer}, Lemma 13]\label{extn}
Let $\Phi$ be an irreducible affine root system and
let $\Psi$ be a subroot system of $\Phi$ and assume that $\Gr(\Psi)$ is reduced.
 Let $\Gamma$ be a simple system of $\Gr(\Psi)$ and let $p:\Gamma\to \mathbb{Z}$ be an arbitrary function. Then there exists a unique 
 $\mathbb{Z}$--linear extension $p$ to $\Gr(\Psi)$, 
 which we denote again by $p$ for simplicity, 
 $p:\Gr(\Psi)\to \mathbb{Z}$ given by $\alpha\mapsto p_\alpha$ satisfying
\begin{equation}\label{palpha}
 p_{\beta}-\langle\beta,\alpha^\vee \rangle p_{\alpha}=p_{\textbf{s}_{\alpha}(\beta)}
\end{equation}
for all $\alpha, \beta\in \Gr(\Psi).$                                                                  
\end{lem}

\subsection{}\label{pexis}
The following proposition is very crucial.
\begin{prop}\label{keypropositionnon2A2n}
Let $\Phi$ be an irreducible affine root system and
let $\Psi$ be a subroot system of $\Phi$. Then there exists a function 
$p^\Psi:\Gr(\Psi)\to \mathbb{Z}, \a\mapsto p_\a^{\Psi},$ and non-negative integers $n_\a^{\Psi}$ for each $\a\in\Gr(\Psi)$
such that $Z_\a(\Psi)=p_\a^{\Psi}+n_\a^{\Psi}\mathbb{Z}$. Moreover
the function $p^\Psi$ is $\mathbb{Z}-$linear if $\Gr(\Psi)$ is reduced.
\end{prop}

\begin{pf}
We will first assume that $\Gr(\Psi)$ is reduced.
 Let $\Gamma$ be a simple system of $\Gr(\Psi)$ and choose arbitrary elements $p_\alpha^\Psi\in Z_{\alpha}(\Psi)$ for each $\alpha\in \Gamma.$ 
 Define a function $p^\Psi:\Gamma\to \mathbb{Z}$ given by $\alpha\mapsto p_\alpha^\Psi$. Now,  fix the unique $\mathbb{Z}$--linear extension of $p^\Psi$ to $\Gr(\Psi)$ as in Lemma \ref{extn}.
 Define
 $$\text{$Z_\alpha'(\Psi)=Z_\alpha(\Psi)-p_\alpha^\Psi=\{r-p_\a^\Psi: r\in Z_\a(\Psi)\}$ for $\alpha\in \Gr(\Psi)$.}$$
 Since each root of $\Gr(\Psi)$ is conjugate to some simple root by an element in  $W_{\Gr(\Psi)}$, we get
 $p_\alpha^\Psi\in Z_\alpha(\Psi)$, for all $\alpha\in \Gr(\Psi)$ and $$ Z_{\beta}'(\Psi)-\langle\beta,\alpha^\vee \rangle Z_{\alpha}'(\Psi)\subseteq Z_{\textbf{s}_{\alpha}(\beta)}'(\Psi),\ \text{for all $\alpha, \beta\in\Gr(\Psi)$}, $$
 using the equation (\ref{zsets}) and (\ref{palpha}).
 One can easily see that $Z_\alpha'(\Psi)$ are subgroups for all $\alpha\in\Gr(\Psi)$, since  $0\in Z_\alpha'(\Psi)$,
 $Z_\alpha'(\Psi)=Z_{-\alpha}'(\Psi)$ and $Z_\alpha'(\Psi)+2Z_\alpha'(\Psi)=Z_\alpha'(\Psi)$
 for all $\alpha\in\Gr(\Psi)$ (proof of this fact is
 same as the proof of Lemma 22 in \cite{Dyer}). Hence there exists $n_\alpha^\Psi\in \mathbb{Z}_+$ for each $\alpha\in\Gr(\Psi)$ such that
 $Z_\alpha'(\Psi)=n_\alpha^\Psi \mathbb{Z}$. This completes the proof in this case. 
 
 \vskip 2mm
We are now left with the case $\Gr(\Psi)$ is non-reduced. Since the sets $Z_\a(\Psi)$ depends only on the individual irreducible components of $\Psi$,
we can assume that $\Psi$ is irreducible. In particular, $\Gr(\Psi)$ is of type $\text{$\tt BC_r$}$ for some $r\ge 1.$ So, we have
$$\mathrm{Gr}(\Psi)=\big\{\pm\epsilon_i, \pm2\epsilon_i, \pm\epsilon_i\pm\epsilon_j: 1\leq i\neq j\leq r\big\}$$ if $r\ge 2$  or 
$\mathrm{Gr}(\Psi)=\{\pm\epsilon_1, \pm2\epsilon_1\}$ if $r=1$ (see \cite[Page no. 547]{carter}).
Write $\Gr(\Psi)_s=\{\pm\epsilon_i : 1\le i\le r\}$, $\Gr(\Psi)_{\mathrm{im}}=\{\pm\epsilon_i\pm \epsilon_j : 1\le i\neq j\le r\}$ and 
 $\Gr(\Psi)_\ell=\{\pm2\epsilon_i : 1\le i\le r\} $. By convention, we have $\Gr(\Psi)_{\mathrm{im}}=\emptyset$ if $r=1.$
Let $\Gamma=\{\a_1=\epsilon_1-\epsilon_2,\cdots, \a_{r-1}=\epsilon_{r-1}-\epsilon_{r}, \a_r=\epsilon_{r}\}$ be the simple system of $\Gr(\Psi)$ and here by convention we have
$\Gamma=\{\epsilon_1\}$ when $r=1$.
Choose arbitrary elements $p_\alpha^\Psi\in Z_{\alpha}(\Psi)$ for each $\alpha\in \Gamma$ and 
define the function $p^\Psi:\Gamma\to \tfrac{1}{2}\mathbb{Z}$, $\a\mapsto p_\a^\Psi$ as before. 
Fix the unique $\mathbb{Z}-$linear extension of $\overline{p^\Psi}$ to $\Gr(\Psi)$ as in Lemma \ref{extn}.
Since the long roots of $\Gr(\Psi)$ are not Weyl group conjugate to simple roots, we will not have $\overline{p_\a^\Psi}\in Z_\a(\Psi)$ for
all long roots $\a\in\Gr(\Psi)_\ell$ as before in reduced case. 
But this is the only obstruction that we have in this case. 
To overcome this issue, first fix a $\mathbb{Z}-$linear extension of $p^\Psi:\Gamma \to \frac{1}{2}\mathbb{Z}$ to 
$p^\Psi:\mathrm{Gr}(\Psi)_s \cup \mathrm{Gr}(\Psi)_{\mathrm{im}}\to \frac{1}{2}\mathbb{Z}$ and choose $p^\Psi_{\a}\in Z_{\a}(\Psi)$ arbitrarily for the positive roots of $\Gr(\Psi)_\ell$. 
Then we see that $-p^\Psi_{\a}\in Z_{-\a}(\Psi)$ for $\a\in \Gr(\Psi)_\ell$. 
So, we take
$p^\Psi_{-\a}:=-p^\Psi_{\a}$
for the negative roots of $\Gr(\Psi)_\ell$ and define a natural extension $$\text{$p^\Psi:\Gr(\Psi)\to \frac{1}{2}\mathbb{Z}$ of
$p^\Psi:\Gr(\Psi)_s \cup \Gr(\Psi)_{\mathrm{im}}\to \frac{1}{2}\mathbb{Z}$}$$
by assigning these arbitrarily chosen $p^\Psi_\a$ to $\a$ for each long root $\a$. 
Now,  note that this new extension $p^\Psi:\Gr(\Psi)\to \frac{1}{2}\mathbb{Z}$
is no longer $\mathbb{Z}-$linear map. 
As before, we define $Z_\a'(\Psi)=Z_\a(\Psi)-p^\Psi_\a$ for all $\a\in \Gr(\Psi)$. 
Then by definition of $Z_\a'(\Psi)$, we have $0\in Z_\a'(\Psi)$ for all $\a\in \Gr(\Psi)$.
Note that $Z_\a(\Psi)$ satisfies the equation (\ref{zsets}), which implies that  
$$ Z_{\beta}'(\Psi)-\langle\beta,\alpha^\vee \rangle Z_{\alpha}'(\Psi)\subseteq Z_{\textbf{s}_{\alpha}(\beta)}'(\Psi)+(p^\Psi_{\bold s_\a(\beta)}-(p^\Psi_\beta-\langle\beta,\alpha^\vee \rangle p^\Psi_\a)),\ \text{for all $\alpha, \beta\in\mathrm{Gr}(\Psi).$} $$
Since $p^\Psi_\a=-p^\Psi_\a$ for all $\a\in \Gr(\Psi)$, we get $Z_{\alpha}'(\Psi)-2Z_{\alpha}'(\Psi)\subseteq Z_{-\alpha}'(\Psi)$ for all $\a\in \Gr(\Psi).$ 
This implies $Z_{-\alpha}'(\Psi)=Z_{\alpha}'(\Psi)$ and $Z_{\alpha}'(\Psi)+2Z_{\alpha}'(\Psi)=Z_{\alpha}'(\Psi)$ for all $\a\in \Gr(\Psi).$
Precisely this fact and 
$0\in Z_\a'(\Psi), \a\in \Gr(\Psi)$ used in the proof of \cite[Lemma 22]{Dyer} to prove that $Z_\a'(\Psi)$ is a subgroup of $\mathbb{Z}$ for all $\a\in \Gr(\Psi)$.
Note that for $\a,\beta\in \Gr(\Psi)$ we have
$\textbf{s}_{\a+p^\Psi_\a\delta}(\beta+p^\Psi_\beta\delta)=\bold s_\a(\beta)+(p^\Psi_\beta-\langle\beta,\alpha^\vee \rangle p^\Psi_\a)\delta,\
\text{which implies that  $p^\Psi_\beta-\langle\beta,\alpha^\vee \rangle p^\Psi_\a\in Z_{\textbf{s}_{\alpha}(\beta)}(\Psi$).}$
This implies that  $(p^\Psi_{\bold s_\a(\beta)}-(p^\Psi_\beta-\langle\beta,\alpha^\vee \rangle p^\Psi_\a))$ must be in $Z_{\bold s_{\a}(\beta)}'(\Psi)$ 
for all $\a,\beta\in \Gr(\Psi)$. Hence, we have 
$$ Z_{\beta}'(\Psi)-\langle\beta,\alpha^\vee \rangle Z_{\alpha}'(\Psi)\subseteq Z_{\textbf{s}_{\alpha}(\beta)}'(\Psi)\ \text{for all $\alpha, \beta\in\mathrm{Gr}(\Psi).$}$$
as before.
Since the sets $Z_\a'(\Psi)$ are subgroups of $\mathbb{Z}$, there exists 
$n^\Psi_\a\in \mathbb{Z}_+$ such that $Z_\a(\Psi)=p^\Psi_\a+n^\Psi_\a\mathbb{Z}$ for all $\a\in \Gr(\Psi)$. This completes the proof in this case.
\end{pf}

\subsection{}
From the Proposition \ref{keypropositionnon2A2n}, it is clear that a subroot system $\Psi$ of $\Phi$ is completely determined by 
the gradient subroot system $\Gr(\Psi)$ and the cosets
$Z_\a(\Psi)=p_\a^{\Psi}+n_\a^{\Psi}\mathbb{Z}$, $\alpha\in \Gr(\Psi)$. 
Naturally if $\Psi$ is closed in $\Phi$, then the ``closedness property of $\Psi$ in $\Phi$" will give us some more restrictions on the gradient subroot systems and the cosets
$Z_\a(\Psi)$. We will completely characterize these restrictions on the gradient subroot systems $\Gr(\Psi)$ and the cosets
$Z_\a(\Psi)$ corresponding to ``closedness property of $\Psi$ in $\Phi$" in Proposition \ref{untwisted01}, \ref{twisted01}, \ref{2A2n01}, \ref{untwisted02}, \ref{twistedgradient} 
and use this information to determine all possible maximal closed subroot systems $\Psi$ of $\Phi$. 
The following lemma tells us about the relationships between  the integers $n_\a^\Psi$. 
Proof of this lemma closely follows the arguments of \cite[Lemma 14]{Dyer} and only uses the fact that 
$$ Z_{\beta}'(\Psi)-\langle\beta,\alpha^\vee \rangle Z_{\alpha}'(\Psi)\subseteq Z_{\textbf{s}_{\alpha}(\beta)}'(\Psi)\ \text{for all $\alpha, \beta\in\mathrm{Gr}(\Phi),$}$$
so we will omit the proof.
  \begin{lem}\label{nalpha}[Lemma 14, \cite{Dyer}]
  Let $\Psi$ be a subroot system of $\Phi$ and let $n_\a^\Psi$ be defined as above.
We have $\langle\beta, \alpha^\vee \rangle n_{\alpha}^\Psi\mathbb{Z}\subseteq n_{\beta}^\Psi\mathbb{Z}$ for all $\alpha,\beta\in \Gr(\Psi)$,
and $n_{\alpha}^\Psi=n_{\beta}^\Psi$ for all $\alpha,\beta\in \Gr(\Psi)$ with $\beta\in W_{\Gr(\Psi)}\alpha$. In particular if $n_{\alpha}^\Psi=0$ for some $\a\in \Gr(\Psi)$
then $n_{\beta}^\Psi=0$ for all $\beta\in W_{\Gr(\Psi)}\alpha$. 
 \end{lem}
Note that when $n_{\beta}^\Psi\neq 0$, we have $\langle\beta, \alpha^\vee \rangle n_{\alpha}^\Psi\mathbb{Z}\subseteq n_{\beta}^\Psi\mathbb{Z}$ if and only if 
$n_{\beta}^\Psi$ divides $\langle\beta, \alpha^\vee \rangle n_{\alpha}^\Psi$.

\subsection{}\label{defns}
Suppose $\Gr(\Psi)$ is reducible say $\Gr(\Psi)=\Psi_1\oplus \cdots \oplus\Psi_k$, then by Lemma \ref{nalpha}  for each $1\le i\le k$ we have
$n_\alpha^\Psi=n_\beta^\Psi$ for all $\alpha, \beta \in (\Psi_i)_\ell$ (resp. for all $\alpha, \beta \in (\Psi_i)_s$ and for all $\alpha, \beta \in (\Psi_i)_{\mathrm{im}}$), 
denote this unique number by $n_\ell^{\Psi_i}(\Psi)$ (resp. $n_s^{\Psi_i}(\Psi)$ and $n_{\mathrm{im}}^{\Psi_i}(\Psi)$). 
We drop $\Psi$ in $n_\ell^{\Psi_i}(\Psi)$ (resp. in $n_s^{\Psi_i}(\Psi)$ and in $n_{\mathrm{im}}^{\Psi_i}(\Psi)$)
and simply denote it by $n_\ell^{\Psi_i}$ (resp. $n_s^{\Psi_i}$ and $n_{\mathrm{im}}^{\Psi_i}$)
if the underlying subroot system $\Psi$ is understood.
Note that long roots (or short roots or intermediate roots) of $\Gr(\Psi)$ from the different components are not conjugate under the action of $W_{\Gr(\Psi)}$.
In particular $n_\ell^{\Psi_1}, \cdots, n_\ell^{\Psi_k}$ 
(resp. $n_s^{\Psi_1}, \cdots, n_s^{\Psi_k}$ or $n_{\mathrm{im}}^{\Psi_1}, \cdots, n_{\mathrm{im}}^{\Psi_k}$) may not be equal.
If $\Gr(\Psi)$ is irreducible, we denote $n_\ell^{\Gr(\Psi)}$ and (resp. $n_{\mathrm{im}}^{\Gr(\Psi)}$)
 by $n_\ell^\Psi$ and $n_s^\Psi$  (resp. $n_{\mathrm{im}}^\Psi$) or simply by $n_\ell$ and $n_s$ (resp. $n_{\mathrm{im}}$) if the underlying subroot system $\Psi$ is understood.
By convention, we have $n_s=n_{\mathrm{im}}$ in case $\Gr(\Psi)$ is of type $\tt BC_1.$
Sometimes we will denote $n_s$ as $n_\Psi$ to emphasize its importance.
We also simply denote $Z_\alpha(\Psi)$, $p_\a^\Psi$ and $n_\a^\Psi$ by $Z_\alpha$, $p_\a$, $n_\a$ if the underlying subroot system $\Psi$ is understood.

\subsection{}  The following Lemma
 compares the cosets $Z_\a$ of two subroot systems of $\Phi.$ 
\begin{lem}\label{keylemma}
 Let $\Psi\subseteq \Delta \subseteq \Phi$ be two  subroot systems of  $\Phi.$  
 \begin{enumerate}
  \item Then we have $\Gr(\Psi)\subseteq \Gr(\Delta)$.
  \item  The cosets satisfy $Z_\a(\Psi)\subseteq Z_\a(\Delta)$ for all $\a\in \Gr(\Psi)$ and in particular $n_\a^\Psi\mathbb{Z}\subseteq n_\a^\Delta\mathbb{Z}$  for all $\a\in \Gr(\Psi)$.
  \item If $\Gr(\Psi)= \Gr(\Delta)$ and  $n_\a^\Delta=n_\a^\Psi$ for all $\a\in \Gr(\Delta)$, then we have $\Psi=\Delta$.
 \end{enumerate}
\end{lem}
\begin{pf}
 By the definition of gradient, we have $\Gr(\Psi)\subseteq \Gr(\Delta)$ and by the definition of $Z_\a(\Psi)$, we have $Z_\a(\Psi)\subseteq Z_\a(\Delta)$ for all $\a\in \Gr(\Psi)$.
 In particular, we have $$p_\a^\Psi+n_\a^\Psi\mathbb{Z}\subseteq p_\a^\Delta+n_\a^\Delta\mathbb{Z}, \ \text{for $\a\in \Gr(\Psi)$}.$$ This implies
 $(p_\a^\Delta-p_\a^\Psi)\in n_\a^\Delta\mathbb{Z}$ and
 $n_\a^\Psi\mathbb{Z}\subseteq (p_\a^\Delta-p_\a^\Psi)+n_\a^\Delta\mathbb{Z}=n_\a^\Delta\mathbb{Z}.$ This proves the Statement (2).
 Finally for the last part, assume that $\Gr(\Psi)= \Gr(\Delta)$ and  $n_\a^\Delta=n_\a^\Psi$ for all $\a\in \Gr(\Delta)$.
 For $\a\in \Gr(\Delta)$, we have  $(p_\a^\Delta-p_\a^\Psi)\in n_\a^\Delta\mathbb{Z}$, and hence 
 $$p_\a^\Psi+n_\a^\Delta\mathbb{Z}= p_\a^\Delta+n_\a^\Delta\mathbb{Z}.$$ This implies that  $Z_\a(\Psi)=Z_\a(\Delta)$ for all $\a\in \Gr(\Delta)$ since $n_\a^\Delta=n_\a^\Psi$.
 Thus,  we have $\Psi=\Delta$ since $\Psi=\{\a+r\delta : \a\in \Gr(\Psi), r\in Z_\a(\Psi)\}$ and $\Delta=\{\a+r\delta : \a\in \Gr(\Delta), r\in Z_\a(\Delta)\}$.
 This completes the proof.
\end{pf}

We record the following lemma for the future use. 
\begin{lem}\label{nalphaneqzero}
Let $\Phi$ be an irreducible affine root system and
 let $\Psi\le \Phi$ be a closed subroot system with an irreducible gradient subroot system $\Gr(\Psi)$. Then
 $n_\a^\Psi = 0$ for some $\a\in \Gr(\Psi)$ implies that  $n_\beta^\Psi = 0$ for all $\beta\in \Gr(\Psi)$ 
\end{lem}
\begin{pf}
Suppose $n_\a^\Psi=0$ for some $\a\in \Gr(\Phi)$. Then, 
since $\Gr(\Psi)$ is irreducible, given any $\beta\in \Gr(\Psi)$ there exists a finite sequence of 
 roots $\beta_1=\a,\cdots,\beta_r=\beta$ such that $(\beta_i, \beta_{i+1})\neq 0$ for all $1\le i\le r-1$.
  Then by Lemma \ref{nalpha}, we have $\langle\beta_i, \beta_{i+1}^\vee \rangle n_{\beta_{i+1}}^\Psi\mathbb{Z}\subseteq n_{\beta_i}^\Psi\mathbb{Z}$ for all $1\le i\le r-1.$
 From this it is clear that $n_{\beta_{1}}^\Psi=0 \implies  n_{\beta_{2}}^\Psi=0 \implies \cdots \implies n_{\beta_{r}}^\Psi=0 $. Thus,  we have $n_\beta^\Psi=0$ for all $\beta\in \Gr(\Phi)$. 
 This completes the proof.

\end{pf}

\subsection{} The following proposition determines the integers $n_\a$ for the closed subroot systems of untwisted affine root systems. 
 \begin{prop}\label{untwisted01} Let $\Phi$ be an irreducible untwisted affine root system.
 \begin{enumerate}
  \item  Suppose $\Psi$ is a closed subroot system of $\Phi$ with 
 an irreducible gradient subroot system $\Gr(\Psi)$, then $n_\alpha=n_\beta$ for all $\alpha, \beta\in \Gr(\Psi).$ Denote this unique number by $n_{\Psi}$.
  \item  Suppose $\Psi$ is a maximal closed subroot system of $\Phi$ with $\Gr(\Psi)=\mathring{\Phi}$, then $n_\Psi$ must be a prime number.
 \end{enumerate}
 \end{prop}
\begin{pf}
Suppose $n_\a=0$ for some $\a\in \Gr(\Psi)$, then by Lemma \ref{nalphaneqzero}, we have $n_\beta=0$ for all $\beta\in \Gr(\Psi)$.
Hence, the Statement $(1)$ is clear in this case.
So, assume that $n_\a\neq 0$ for all $\a\in \Gr(\Psi)$.
Suppose $\Gr(\Psi)$ is simply laced, then we have $n_\a=n_\beta$ for all $\a, \beta\in \Gr(\Psi)$ by Lemma \ref{nalpha}. Hence, the Statement $(1)$ is immediate in this case.
So, we assume that $\Gr(\Psi)$ is non simply-laced irreducible 
root system. We can choose two short roots $\alpha_1$ and $\alpha_2$ in $\Gr(\Psi)$ such that their sum 
$\alpha_1+\alpha_2$ is a long root in $\Gr(\Psi)$. Then from Lemma \ref{nalpha}, we have $n_{\a_1}=n_{\a_2}=n_s$ and $n_\ell=n_{\a_1+\a_2}$. As $\Psi$ is closed we have 
$$Z_{\alpha_1}+Z_{\alpha_2} =Z_{\alpha_1+\alpha_2}.$$ Since $p_{\a_1+\a_2}=p_{\a_1}+p_{\a_2}$, 
we get
$Z_{\alpha_1}^{'}+Z_{\alpha_2}^{'} =Z_{\alpha_1+\alpha_2}^{'}$, which implies that  $n_\ell \mid n_s$. 
On the other hand $\langle\beta, \alpha^\vee \rangle=\pm1$ for short root $\beta$ and long root $\a$, see \cite[Page no. 45]{Hump}. 
Using this and by Lemma \ref{nalpha}, we get $n_s\mid n_\ell$ and hence $n_\ell=n_s.$ This completes the proof of Statement $(1).$

For the second part,  it follows from the discussion in Section \ref{pexis} and Statement $(1)$  that there exists $p_\a$ such that 
$Z_\alpha=p_\alpha+n_\Psi\mathbb{Z}$ for all $\a\in \mathring{\Phi}$.
Suppose $n_\Psi=0$, then we have  $$\Psi=\{\a+p^\Psi_\a\delta : \a \in \Gr(\Phi)\}\subsetneq \Delta,$$ where $\Delta$ is a proper closed subroot system of $\Phi$ given by
 $\Delta= \{\a+(p^\Psi_\a+ 2r)\delta : \a \in \Gr(\Phi), r\in \mathbb{Z}\}$. This is a contradiction to our assumption that $\Psi$ is maximal closed subroot system in $\Phi$, so we must have
$n_\Psi\neq 0.$
Suppose $n_\Psi=1$, then it is immediate that $Z_\a=\mathbb{Z}$ for all $\a\in \mathring{\Phi}$. Hence, $\Psi=\Phi$ which is again a contradiction. So, we must have $n_\Psi\neq 1$.
Suppose $n_\Psi$ is not a prime number and let $n_\Psi=uv$ be a nontrivial factorization of $n_\Psi$, then we have
$$\Omega:=\left\{\alpha+(p_\alpha+ur)\delta : \alpha \in \mathring{\Phi}, r \in \mathbb{Z}\right\}$$
 is a closed subroot system of $\Phi$ since the function $\a \mapsto p_\a$ is $\mathbb{Z}-$linear and satisfies the Equation  (\ref{palpha}) and
 $$\bold s_{\a + (p_\a+ur)\delta}(\beta + (p_\beta+ur')\delta)=\bold s_\a(\beta)+ (p_{\bold s_\a(\beta)}+ur'-ur\langle \beta, \a^\vee \rangle)\delta$$
 for $\a, \beta\in \mathring{\Phi}$ and $r, r'\in \mathbb{Z}$.
But $\Psi\subsetneqq\Omega\subsetneqq\Phi$, which contradicts the fact that $\Psi$ is maximal closed subroot system in $\Phi$. This completes the proof of Statement $(2)$.
\end{pf}

\subsection{} We have the following proposition which is similar to Proposition \ref{untwisted01} for twisted affine root systems not of type $\tt A^{(2)}_{2n}$.
Recall the definition of $m$ from Section \ref{definitionofm}.
\begin{prop}
\label{twisted01} Let $\Phi$ be an irreducible twisted affine root system not of type $\tt A^{(2)}_{2n}$ and let $\Psi\le \Phi$ be a subroot system with an
irreducible gradient subroot system  $\Gr(\Psi)$. Let $n_\ell$ and $n_s$ be defined as in Section \ref{defns}.
 \begin{enumerate}
 \item  Suppose $\Psi$ is a closed subroot system of $\Phi$ such that $\Gr(\Psi)$ is simply laced, then we get $n_\a=n_\beta$ for all $\a, \beta\in \Gr(\Psi)$.   
 Denote this unique number by $n_\Psi$.

  \item  Suppose $\Psi$ is a closed subroot system of $\Phi$ such that $\Gr(\Psi)$ is non simply-laced, then we get  $n_\ell=n_s$ if $m|n_s$ and we get \ $n_\ell=mn_s$ if $m\not|n_s.$ 
  Denote $n_s$ by $n_\Psi$.
  
  \item  Suppose $\Psi$ is a maximal closed subroot system of $\Phi$ with $\Gr(\Psi)=\mathring{\Phi}$, then $n_\Psi$ is a prime number.
 \end{enumerate}
 \end{prop}
 
\begin{pf}
Suppose $n_\a=0$ for some $\a\in \Gr(\Psi)$, then by Lemma \ref{nalphaneqzero}, we have $n_\beta=0$ for all $\beta\in \Gr(\Psi)$. Hence, the Statements $(1)$ and $(2)$ are clear in this case.
So, we assume that $n_\a\neq 0$ for all $\a\in \Gr(\Psi)$.
Suppose $\Gr(\Psi)$ is simply laced, then we have $n_\a=n_\beta$ for all $\a, \beta\in \Gr(\Psi)$ by Lemma \ref{nalpha}. Hence, the Statement $(1)$ follows.
So, we assume that $\Gr(\Phi)$ is non-simply laced and irreducible to prove the Statement $(2)$.
Since $\Gr(\Psi)$ is irreducible and non simply-laced, we can choose two short roots $\alpha_1$ and $\alpha_2$ in $\Gr(\Psi)$ such that their sum 
$\alpha_1+\alpha_2$ is a long root in $\Gr(\Psi)$. Again using Lemma \ref{nalpha}, we have $n_{\a_1}=n_{\a_2}=n_s$ and $n_\ell=n_{\a_1+\a_2}$. As $\Psi$ is closed, we have 
$$\left(Z_{\alpha_1}+Z_{\alpha_2}\right)\cap m\mathbb{Z} =Z_{\alpha_1+\alpha_2}.$$ 
Since $p_{\a_1+\a_2}=p_{\a_1}+p_{\a_2}$ and $p_{\a_1+\a_2}\in m\mathbb{Z}$, 
we get
$\left(Z_{\alpha_1}^{'}+Z_{\alpha_2}^{'}\right)\bigcap m\mathbb{Z} =Z_{\alpha_1+\alpha_2}^{'}$, which implies that  
$$n_s\mathbb{Z}\cap m\mathbb{Z}=n_\ell \mathbb{Z}.$$ Thus,  we get $n_\ell=n_s$ if $m|n_s$ and  $n_\ell=mn_s$ if $m\not|n_s.$ This proves the Statement $(2)$ of the proposition.
\noindent

For the last part, observe that $\Psi<\Phi$ is properly contained in $\Phi$ since $\Psi$ is a maximal closed subroot system of $\Phi$. 
We know that there exists $p_\a$ such that $Z_\alpha=p_\alpha+n_\a\mathbb{Z}$ for all $\a\in \mathring{\Phi}$.
Suppose $n_\Psi=0$, then we have  $$\Psi=\{\a+p^\Psi_\a\delta : \a \in \Gr(\Phi)\}\subsetneq \Delta,$$ where $\Delta$ is a proper closed subroot system of $\Phi$ given by
 $\Delta= \{\a+(p^\Psi_\a+ mr)\delta : \a \in \Gr(\Phi), r\in \mathbb{Z}\}$. This is a contradiction to our assumption that $\Psi$ is maximal closed subroot system in $\Phi$, so we must have
$n_\Psi\neq 0.$
If $n_\Psi=1$, then it is immediate that $n_\ell=m$ and $n_s=1$. This implies that $Z_\alpha=\mathbb{Z}$ for short roots $\a$ and 
$Z_\alpha=m\mathbb{Z}$ for long roots $\a$. Hence, we get $\Psi=\Phi$ since $\Gr(\Psi)=\mathring{\Phi}$, again a contradiction.
Suppose $n_\Psi$ is not a prime number, then let $n_\Psi=uv$ be a nontrivial factorization of $n_\Psi$ such that $m|u$ if $m|n_\Psi$.  Let
$$\Omega = \{\alpha+(p_\alpha+ur)\delta : \alpha \in \mathring{\Phi}, r \in \mathbb{Z} \}$$ if $\Gr(\Psi)$ is simply laced or $m\mid n_\Psi$. Otherwise
let $$\Omega = \{\alpha+(p_\alpha+mur)\delta, \beta+(p_\beta+ur)\delta : \alpha \in \mathring{\Phi}_\ell, \beta\in \mathring{\Phi}_s, r \in \mathbb{Z}\}.$$
We claim that $\Omega$ is a closed subroot system of $\Phi$. Note that the function $\a \mapsto p_\a$ is $\mathbb{Z}-$linear and satisfies the Equation (\ref{palpha}). 
Let $\alpha \in \mathring{\Phi}_\ell, \beta\in \mathring{\Phi}_s$, then for $r, r'\in \mathbb{Z}$ we have
$$\bold s_{\beta+(p_\beta+ur)\delta}(\alpha+(p_\alpha+mur')\delta)=\bold s_\beta(\a) + ((p_\alpha+mur')-(p_\beta+ur)\langle \a, \beta^\vee \rangle)\delta.$$
Since $p_{\bold s_\beta(\a)}= p_\a - \langle\a, \beta^\vee \rangle p_\beta$, we have 
$$\bold s_{\beta+(p_\beta+ur)\delta}(\alpha+(p_\alpha+mur')\delta)=\bold s_\beta(\a) + (p_{\bold s_\beta(\alpha)}+ mur'-ur\langle\a, \beta^\vee \rangle)\delta.$$
Now,  since $\langle\a, \beta^\vee \rangle= \langle\beta, \a^\vee \rangle m$ and $\bold s_\beta(\a)$ is a long root, we have $\bold s_{\beta+(p_\beta+ur)\delta}(\alpha+(p_\alpha+mur')\delta)\in \Omega.$
Similarly,  for $\alpha \in \mathring{\Phi}_\ell, \beta\in \mathring{\Phi}_s$ and $r, r'\in \mathbb{Z}$ we have $$\bold s_{\alpha+(p_\alpha+mur')\delta}(\beta+(p_\beta+ur)\delta)=\bold s_\a(\beta) + (p_{\bold s_\a(\beta)}+ ur-mur'\langle \beta, \a^\vee \rangle)\delta\in \Omega$$
since $\bold s_\a(\beta)$ is a short root. Remaining cases are similarly done, so it proves that $\Omega$ is a subroot system.
Since sum of a short root and long root from $\mathring{\Phi}$ can not be a long root again, we get $\Omega$ is closed subroot system in $\Phi.$
But $\Psi\subsetneqq\Omega\subsetneqq\Phi$, which contradicts the fact that $\Psi$ is a maximal closed subroot system in $\Phi$. This completes the proof of Statement $(3)$.
\end{pf}

\subsection{}
We have the following result which is analogues to the Propositions \ref{untwisted01} and \ref{twisted01} in the  $\tt A^{(2)}_{2n}$ setting.
\begin{prop}\label{2A2n01} Let $\Phi$ be an irreducible twisted affine root system of type $\tt A^{(2)}_{2n}$ and let $\Psi\le \Phi$ be a subroot system with an
irreducible gradient subroot system  $\Gr(\Psi)$.
Let $n_\ell$, $n_{\mathrm{im}}$ and $n_s$ be defined as in Section \ref{defns}.
 \begin{enumerate}
 \item  Suppose $\Psi$ is a closed subroot system of $\Phi$ such that $\Gr(\Psi)$ is simply laced, then we get $n_\a=n_\beta$ for all $\a, \beta\in \Gr(\Psi)$.   

  \item  Suppose $\Psi$ is a closed subroot system of $\Phi$ such that $\Gr(\Psi)$ is non-simply laced and does not contain any short root, then we get  $n_\ell=n_{\mathrm{im}}$ if $2|n_{\mathrm{im}}$ and we get $n_\ell=2n_{\mathrm{im}}$ if $2\not|n_{\mathrm{im}}.$ 
   
  \item  Suppose $\Psi$ is a closed subroot system of $\Phi$ such that $\Gr(\Psi)$ is non-simply laced and does not contain any long root, then we get  $n_s=n_{\mathrm{im}}$.    
    
  \item  Suppose $\Psi$ is a closed subroot system of $\Phi$ with $\Gr(\Psi)$ containing short, intermediate and long roots, then $n_s=n_{\mathrm{im}}$, $n_\ell=2n_s$ and $n_s$ is an odd number.
         Denote $n_s$ by $n_\Psi.$
 \item Suppose $\Psi$ is a maximal closed subroot system of $\Phi$ with $\Gr(\Psi)=\Gr(\Phi)$, then $n_\Psi$ must be a prime number. 
 \end{enumerate}

 \end{prop}
 
\begin{proof}

Suppose $n_\a=0$ for some $\a\in \Gr(\Psi)$, then by Lemma \ref{nalphaneqzero}, we have $n_\beta=0$ for all $\beta\in \Gr(\Psi)$. 
Hence, the Statements $(1)$, $(2)$, $(3)$ and $(4)$ are clear in this case.
So, we assume that $n_\a\neq 0$ for all $\a\in \Gr(\Psi)$.
Suppose $\Gr(\Psi)$ is simply laced, then we have $n_\a=n_\beta$ for all $\a, \beta\in \Gr(\Psi)$ by Lemma \ref{nalpha}. Hence, the Statement $(1)$ follows.
Suppose $\Psi$ is a closed subroot system of $\Phi$ such that $\Gr(\Psi)$ does not contain any short root, then $\Psi$ is a closed subroot system of $\tt A^{(2)}_{2n-1}$. 
Hence, the Statement $(2)$ follows from Proposition \ref{twisted01}.
Suppose $\Psi$ is a closed subroot system of $\Phi$ such that $\Gr(\Psi)$ does not contain any long root. 
By Lemma \ref{nalpha}, $n_s\mid n_{\mathrm{im}}$ and $n_{\mathrm{im}}\mid 2n_s$. Then by Proposition \ref{keypropositionnon2A2n}  and Lemma \ref{nalphaneqzero}, we have
$n_\a\in \mathbb{N}$ and $p_\a\in Z_\a(\Psi)$ such that $Z_\a(\Psi)=p_\a+n_\a\mathbb{Z}$ for all $\a\in \Gr(\Phi)$.
If there is only one short root in $\Gr(\Psi)$, then we have $n_s=n_{\mathrm{im}}$ by convention.
So assume that we can choose two short roots $\a, \beta\in \Gr(\Psi)$ such that $\alpha+\beta$ is an intermediate root. Then
since $\Psi$ is closed, we have
$$(p_{\a}+n_s\mathbb{Z})+(p_{\beta}+n_s\mathbb{Z})=p_{\a+\beta}+n_s\mathbb{Z}\subseteq p_{\a+\beta}+n_{\mathrm{im}}\mathbb{Z},$$
which implies that  $n_s\mathbb{Z}\subseteq n_{\mathrm{im}}\mathbb{Z}$ and $n_{\mathrm{im}}\mid n_s$ and hence
$n_s=n_{\mathrm{im}}$. This completes proof of Statement $(3).$

Suppose $\Psi$ is a closed subroot system of $\Phi$ such that $\Gr(\Psi)$ contains short, intermediate and long roots, then $n_s=n_{\mathrm{im}}$ as before. 
By Lemma \ref{nalpha}, $n_{\mathrm{im}}\mid n_{\ell}$ and $n_{\ell}\mid 2n_{\mathrm{im}}$. Then by Proposition \ref{keypropositionnon2A2n} and Lemma \ref{nalphaneqzero}, we have
$n_\a\in \mathbb{N}$ and $p_\a\in Z_\a(\Psi)$ such that $Z_\a(\Psi)=p_\a+n_\a\mathbb{Z}$ for all $\a\in \Gr(\Phi)$.
Since $\Psi$ is closed, we have $Z_{2\a}(\Psi)-Z_{\a}(\Psi)\subseteq Z_{\a}(\Psi)$ for a short root $\a\in \Gr(\Psi)$. 
This implies that  $$(p_{2\a}-p_{\a})+n_\ell \mathbb{Z}\subseteq p_{\a}+n_s\mathbb{Z}\ \text{and hence} \  
p_{2\a}+n_\ell \mathbb{Z}\subseteq (2p_{\a}+n_s\mathbb{Z})\cap 2\mathbb{Z},$$ since $p_{2\a}+n_\ell \mathbb{Z}\subseteq 2\mathbb{Z}$.
From this, we conclude that $n_s$ must be odd since $2p_{\a}$ is odd.
Since $(2p_{\a}+n_s\mathbb{Z})\cap 2\mathbb{Z}\subseteq Z_{2\a}(\Psi)=p_{2\a}+n_\ell \mathbb{Z}$
we have $$p_{2\a}+n_\ell \mathbb{Z}= (2p_{\a}+n_s\mathbb{Z})\cap 2\mathbb{Z}=(2p_{\a}+n_s)+2n_s\mathbb{Z}.$$
This implies, we must have $n_\ell=2n_s$. This completes proof of Statement $(4).$

Suppose $\Psi$ is a maximal closed subroot system with $\Gr(\Psi)=\Gr(\Phi)$ and $n_\a=0$ for some $\a\in \Gr(\Phi)$, then by Lemma \ref{nalphaneqzero}, we have $n_\beta= 0$ for all $\beta\in \Gr(\Phi)$.
 This implies that
 $\Psi=\{\a+p^\Psi_\a\delta : \a \in \Gr(\Phi)\}\subsetneq \Delta$, where $\Delta$ is a proper closed subroot system of $\Phi$ given by
 $$\Delta=\{\a+(p^\Psi_\a+ 3r)\delta : \a \in \Gr(\Phi)_s\cup \Gr(\Phi)_{\mathrm{im}}, r\in \mathbb{Z}\}\cup \{\a+(p^\Psi_\a+ 6r)\delta : \a \in \Gr(\Phi)_\ell, 
 r\in \mathbb{Z}\}.$$
 Then $\Psi$ can not be maximal closed subroot system in $\Phi$, a contradiction to our assumption. Hence, $n_\a\neq 0$ for all $\a\in \Gr(\Phi)$. 
 Suppose $n_\Psi=1$, then we get $\Psi=\Phi$ from Statement (4), a contradiction. So, $n_{\Psi}\neq 1.$
  Now suppose $n_{\Psi}$ is a composite number and $n_{\Psi}=pq$. Since $n_\Psi$ is an odd integer,  without loss of generality we can assume that $p$ is an odd integer.
  Then $\Psi\subsetneq\Delta$, where $\Delta$ is a proper closed subroot system of $\Phi$ given by
 $$\Delta=\{\a+(p^\Psi_\a+ pr)\delta : \a \in \Gr(\Phi)_s\cup \Gr(\Phi)_{\mathrm{im}}, r\in \mathbb{Z}\}\cup \{\a+(p^\Psi_\a+ 2pr)\delta : \a \in \Gr(\Phi)_\ell, 
 r\in \mathbb{Z}\}.$$
Hence, $n_\Psi$ must be a prime number. This completes the proof.
\end{proof}


\section{Untwisted Case}\label{untwi}
Throughout this section we assume that $\Phi$ is an irreducible untwisted affine root system. Note that $\Gr(\Phi)=\mathring{\Phi}$ and $\wh{\mathring{\Phi}}=\Phi$.

\subsection{}
We need the following simple result to complete the classification of maximal closed subroot systems in this case.
The Statement (2) of the following proposition already appears in the proof of \cite[Lemma 4.1]{FRT}.
\begin{prop}\label{untwisted02}
Let $\Phi$ be an irreducible untwisted affine root system and let $\Psi\le \Phi$ be a subroot system.
\begin{enumerate}
 \item If $\Psi\le \Phi$ is a closed subroot system, then $\Gr(\Psi)\le \mathring{\Phi}$ is a closed subroot system.
 \item If $\Psi\le \Phi$ is a maximal closed subroot system, then either $\Gr(\Psi)=\mathring{\Phi}$  or $\Gr(\Psi)\subsetneq \mathring{\Phi}$ is a maximal 
 closed subroot system. In particular we get $\Psi=\widehat{\mathrm{Gr}(\Psi)}$ when $\Gr(\Psi)\subsetneq \mathring{\Phi}$.
\end{enumerate}
\end{prop}

\begin{pf}
Statement $(1)$ is immediate from the definition. 
Now,  suppose $\Gr(\Psi)\neq\mathring{\Phi}$, then we claim that $\Gr(\Psi) \subsetneq \mathring{\Phi}$ is a maximal closed subroot system.
Otherwise,  there exist a closed subroot system $\Omega$ such that
$\Gr(\Psi)\subsetneqq\Omega\subsetneqq\mathring{\Phi}$
which immediately implies that
$\Psi\subsetneqq\widehat{\Omega}\subsetneqq{\Phi}$.
This leads to a contradiction as $\widehat{\Omega}$ is closed in $\Phi$ by Lemma \ref{closedlemma}. 
Since $\widehat{\mathrm{Gr}(\Psi)}$ is a proper closed subroot system
which contains $\Psi$, we must have $\Psi=\widehat{\mathrm{Gr}(\Psi)}.$  This completes the proof of Statement $(2)$.
\end{pf}

\subsection{}\label{mainthmuntwistedcasedyer}
Now,  we are ready to state our main theorem for untwisted case. 
\begin{thm}\label{mainuntwisted}
Let  $\Psi$ be a maximal closed subroot system of $\Phi$.
\begin{enumerate}
 \item If $\Gr(\Psi)= \mathring{\Phi}$, then there exists a $\mathbb{Z}$--linear function $p:\Gr(\Psi)\to \mathbb{Z}$
satisfying (\ref{palpha}) and a prime number $n_\Psi$ such that
 $$\Psi=\left\{\alpha+(p_\alpha+rn_\Psi)\delta: \alpha\in \Gr(\Psi), r\in\mathbb{Z}\right\}.$$
 Conversely,   given a $\mathbb{Z}$--linear function $p:\mathring{\Phi}\to \mathbb{Z}$ satisfying (\ref{palpha}) and a prime number $n_\Psi$ the subroot system $\Psi$ defined above gives  a maximal subroot system
 of $\Phi$. The affine type of $\Psi$ is same as affine type of $\Phi.$
 
 \item If $\Gr(\Psi)\subsetneq \mathring{\Phi}$ is a maximal closed subroot system, then
$$\Psi=\left\{\alpha+r\delta: \alpha\in \Gr(\Psi), r\in\mathbb{Z}\right\}.$$ Conversely,   if $\mathring{\Psi}$ is a proper maximal subroot system of $\mathring{\Phi}$ then 
the lift $\widehat{\mathring{\Psi}}$ is a maximal subroot system of $\Phi.$ The affine type of $\widehat{\mathring{\Psi}}$ is $\tt X_n^{(1)}$ 
if $\mathring{\Psi}$ is of finite type $\tt X_n.$
\end{enumerate}

\end{thm}
\begin{pf}
 Forward part of Statement $(1)$ follows from the Proposition \ref{untwisted01}. For the converse part let
  $\Psi=\left\{\alpha+(p_\alpha+rn_\Psi)\delta: \alpha\in \Gr(\Psi), r\in\mathbb{Z}\right\}$, where the function
  $p:\Gr(\Psi)\to \mathbb{Z}$ is $\mathbb{Z}$--linear and satisfying (\ref{palpha}) and $n_\Psi$ is a prime number. 
  It is easy to verify that $\Psi$ is a closed subroot system of $\Phi$ since $p$ is $\mathbb{Z}$--linear and satisfying (\ref{palpha}).
  Now,  suppose $\Psi\subsetneq \Delta \subseteq \Phi$, then 
  $\Gr(\Delta)=\mathring{\Phi}$ since $\Gr(\Psi)=\mathring{\Phi}$. Now,  by part $(2)$ of Lemma \ref{keylemma}, we have $n_\Delta$ divides $n_\Psi$. This implies
  $n_\Delta=1$ or $n_\Delta=n_\Psi$ since $n_\Psi$ is a prime number. If $n_\Delta=n_\Psi$, then by part $(3)$ of Lemma \ref{keylemma}, we get $\Psi = \Delta$, a contradiction. 
  So, we must get $n_\Delta=1$, this implies that  $\Delta=\Phi$.
  This completes the proof of Statement $(1).$

 Forward part of Statement $(2)$ follows from the part $(2)$ of Proposition \ref{untwisted02} and the converse part is straightforward from 
the part $(2)$ of Proposition \ref{untwisted02} and the Lemma \ref{keylemma}.
\end{pf}

\begin{rem}
 Our main classification theorem for the untwisted case is indeed an immediate corollary of the results of \cite{Dyer}, see also \cite{DyerTAMS}. 
 Essentially all the machineries were developed in \cite{Dyer} to complete the 
 classification of maximal closed subroot system of untwisted affine root system. 
 Since the purpose of their paper is to classify all the subroot systems in terms of the
 admissible
subgroups of the coweight lattice of a root system $\Psi$, and the scaling functions on $\Psi$, the authors 
 do not write Theorem \ref{mainuntwisted} as a corollary of their results. 
 The main purpose of this paper is to get a similar classification theorem of maximal subroot systems for the twisted affine root system case as well.
\end{rem}

\subsection{}\label{tabeluntwisted}
We end this section by listing out all possible types of maximal closed subroot systems of irreducible untwisted affine root systems and
 give few examples to demonstrate how one gets the this list from Theorem \ref{mainuntwisted} and Table \ref{mrs3}.
\begin{example}
Let $\Phi=\tt B_n^{(1)}$. Then $\mathring\Phi=\text{$\tt B_n$}=\{\pm\epsilon_i,\pm\epsilon_i\pm\epsilon_j : 1\leq i\neq j\leq n\}$. The root system
$\tt B_n$ has a maximal closed subroot system $\Delta$ of type $\tt B_{n-1}$ with a
simple system $\{\epsilon_2-\epsilon_3,\epsilon_3-\epsilon_4,\cdots, \epsilon_{n-1}-\epsilon_n,\epsilon_n\}$ (see \cite[Page 136]{kane}). 
By Theorem \ref{mainuntwisted}, $\widehat{\Delta}$ is a maximal closed subroot system of $\Phi$ and by Definition \ref{type}, the type of $\widehat{\Delta}$ is $\tt B_{n-1}^{(1)}.$
\end{example}

\begin{example}
Let $\Phi=\tt G_2^{(1)}$. Then $\mathring\Phi=\text{$\tt G_2$}=\{\epsilon_i-\epsilon_j,\pm(\epsilon_i+\epsilon_j-2\epsilon_k) : 1\leq i,j,k\leq 3, i\neq j\}$. The root system
$\tt G_2$ has a maximal closed subroot system $\Delta$ of type $\tt A_1\oplus A_1$ with a simple system
$\{\epsilon_1-\epsilon_2,\epsilon_1+\epsilon_2-2\epsilon_3\}$ (see \cite[Page 136]{kane}). By Theorem 
\ref{mainuntwisted}, $\widehat{\Delta}$ is a maximal closed subroot system of $\Phi$ and the type of $\widehat{\Delta}$ is $\tt A_{1}^{(1)}\oplus A_{1}^{(1)}.$
\end{example}

\begin{example}
Let $\Phi=\tt D_n^{(1)}$. Then $\mathring\Phi=\text{$\tt D_n$}=\{\pm\epsilon_i\pm\epsilon_j : 1\leq i\neq j\le n\}$. The root system
$\tt D_n$ has a maximal closed subroot system $\Delta$ of type $\tt D_{n-1}$ with a simple system
$\{\epsilon_2-\epsilon_3,\epsilon_3-\epsilon_4,\cdots, \epsilon_{n-1}-\epsilon_n,\epsilon_{n-1}+\epsilon_n\}$
(see \cite[Page 136]{kane}). By Theorem 
\ref{mainuntwisted}, $\widehat{\Delta}$ is a maximal closed subroot system of $\Phi$ and the type of $\widehat{\Delta}$ is $\tt D_{n-1}^{(1)}.$
\end{example}

\noindent
The following table is immediate from Theorem \ref{mainuntwisted} and Table \ref{mrs3}.
\begin{table}[ht]
\caption{Types of maximal closed subroot systems of irreducible untwisted affine root systems}
\centering 
\begin{tabular}{|c|c|c|}
\hline
Type & Reducible & Irreducible \\
\hline 
$\tt A_n^{(1)}$ & $\tt A_r^{(1)} \oplus A_{n-r-1}^{(1)}\; (0 \leq r \leq n-1)$ & $\tt A_n^{(1)}$  \\[1ex]
$\tt B_n^{(1)}$ & $\tt B_r^{(1)} \oplus D_{n-r}^{(1)}\;\;\;\; (1 \leq r \leq n-2)$ & $\tt B_{n-1}^{(1)}$, $\tt D_n^{(1)}$, $\tt B_n^{(1)}$ \\[1ex]
$\tt C_n^{(1)}$ & $\tt C_r^{(1)} \oplus C_{n-r}^{(1)}\;\;\;\; (1 \leq r \leq n-1)$ & $\tt A_{n-1}^{(1)}$, $\tt C_n^{(1)}$ \\[1ex]
$\tt D_n^{(1)}$ & $\tt D_r^{(1)} \oplus D_{n-r}^{(1)}\;\;\; (2 \leq r \leq n-2)$ & $\tt A_{n-1}^{(1)}$, $\tt D_{n-1}^{(1)}$, $\tt D_n^{(1)}$ \\[1ex]
$\tt E_6^{(1)}$ & $\tt A_5^{(1)} \oplus A_{1}^{(1)}$, $\tt A_2^{(1)} \oplus A_{2}^{(1)} \oplus A_2^{(1)}$ & $\tt D_{5}^{(1)}$, $\tt E_6^{(1)}$ \\[1ex]
$\tt E_7^{(1)}$ & $\tt A_5^{(1)} \oplus A_{2}^{(1)}$, $\tt A_1^{(1)} \oplus D_{6}^{(1)} $ & $\tt E_{6}^{(1)}$, $\tt A_7^{(1)}$, $\tt E_7^{(1)}$ \\[1ex]
$\tt E_8^{(1)}$ & $\tt A_1^{(1)} \oplus E_{7}^{(1)}$, $\tt E_6^{(1)} \oplus A_{2}^{(1)}, A_4^{(1)} \oplus A_{4}^{(1)} $ & $\tt D_{8}^{(1)}$, $\tt A_8^{(1)}$, $\tt E_8^{(1)}$ \\ [1ex]
$\tt F_4^{(1)}$ & $\tt A_2^{(1)} \oplus A_{2}^{(1)}$, $\tt A_1^{(1)} \oplus C_{3}^{(1)} $ & $\tt B_{4}^{(1)}$, $F_4^{(1)}$ \\[1ex]
$\tt G_2^{(1)}$ & $\tt A_1^{(1)} \oplus A_{1}^{(1)}$ & $\tt A_{2}^{(1)}$, $\tt G_2^{(1)}$ \\[1ex]

\hline
\end{tabular}
\label{mrs}
\end{table}

\begin{rem}
 The Table \ref{mrs} has already appeared in \cite{FRT} and note that the authors of \cite{FRT} have 
 omitted the possibility of a maximal closed subroot system $\tt D_{n-1}^{(1)}\subset \tt D_{n}^{(1)}$
in their list. 
\end{rem}

\section{Twisted Case }\label{twistedcase}
Throughout this section we assume that $\Phi$ is an irreducible twisted affine root system which is not of type $\tt A^{(2)}_{2n}$. 
Let $\Psi\le \Phi$ be a closed subroot system. Unlike in untwisted case, we have three choices for $\mathrm{Gr}(\Psi)$ in this case.
Indeed because of this fact, the classification of maximal closed subroot systems of twisted affine root systems becomes more technical.
\subsection{} We begin with the definition of the third possible case.
\begin{defn}\label{defnsemiclosed}
 A subroot system $\mathring{\Psi}$ of $\mathring{\Phi}$ is said to be {\em{semi-closed}} if
 \begin{enumerate}
  \item $\mathring{\Psi}$ is not closed in $\mathring{\Phi}$ and
  \item 
if $\alpha, \beta\in \mathring{\Psi}$ such that $\alpha+\beta\in\mathring{\Phi}\backslash \mathring{\Psi}$, then $\alpha$ and $\beta$ must be short roots and 
  $\alpha+\beta$ must be a long root.
 \end{enumerate}

\end{defn}
The condition (1) in Definition \ref{defnsemiclosed} implies
 that there must exist two roots $\alpha, \beta \in \mathring\Psi$ 
 such that $\alpha+\beta \in \mathring\Phi\setminus\mathring\Psi$ and the condition (2) ensures that $\alpha$ and $\beta$ are short roots and 
  $\alpha+\beta$ is a long root. Thus,  if $\mathring\Psi$ is semi-closed in $\mathring\Phi$, then there exists short roots $\alpha$ and $\beta$
  such that their sum $\alpha+\beta$ is a long root and  $\alpha+\beta\in\mathring{\Phi}\backslash \mathring{\Psi}$.

\begin{prop}\label{twistedgradient}
Let $\Phi$ be an irreducible twisted affine root system not of type $\tt A_{2n}^{(2)}$ and let $\Psi\le \Phi$ be a subroot system. If $\Psi\le \Phi$ is a closed subroot system, then either
 \begin{enumerate}
 \item $\Gr(\Psi)=\Gr(\Phi)$ or
  \item $\Gr(\Psi)$ is a proper closed subroot system of $\Gr(\Phi)$ or
  \item $\Gr(\Psi)$ is a proper semi-closed subroot system of $\Gr(\Phi)$.
 \end{enumerate}
\end{prop}
\begin{pf}
 Let $\Gr(\Psi)$ neither be equal to $\Gr(\Phi)$ nor be a proper closed subroot system of $\Gr(\Phi)$. 
 Then there must exist two roots $\alpha_1,\alpha_2 \in \Gr(\Psi)$ such that $\alpha_1+\alpha_2 \in \Gr(\Phi)\setminus\Gr(\Psi)$. 
 We claim that the roots $\alpha_1,\alpha_2$ must be short roots and their sum $\alpha_1+\alpha_2$ must be a long root.
 Since $\alpha_1,\alpha_2 \in \Gr(\Psi)$,
 there exists $u,v \in\mathbb{Z}$ such that $\alpha_1+u\delta$, $\alpha_2+v\delta \in \Psi$. As $\Psi$ is closed and $\alpha_1+\alpha_2 \in \Gr(\Phi)\setminus\Gr(\Psi)$, 
 we have $\alpha_1+\alpha_2+(u+v)\delta \notin \Phi$.
 This implies that  $\alpha_1+\alpha_2$ is a long root.
 
 Suppose that both $\alpha_1$ and $\alpha_2$ are long roots. Then both $u$ and $v$ are integer multiples of $m$, 
 and hence so is $u+v$, 
 which contradicts the fact that  $\alpha_1+\alpha_2+(u+v)\delta \notin \Phi$. 
 So, they can not be both long. Since a sum of a short root and a long root can not be a long root, we have both $\alpha_1$ and $\alpha_2$ are short roots. 
 This proves that $\Gr(\Psi)$ must be a proper semi-closed subroot system of $\Gr(\Phi)$.
\end{pf}

\subsection{} Now,  we assume that $\Psi\le \Phi$ is a maximal closed subroot system. 
Then by Proposition \ref{twistedgradient}, we have three choices for $\mathrm{Gr}(\Psi)$.
First two cases of Proposition \ref{twistedgradient}  are easier to study and they are similar to the untwisted affine root systems.
The case (3) of Proposition \ref{twistedgradient} requires a case-by-case analysis.
In this section,
we study the easier cases $(1)$ and $(2)$ and in Sections \ref{2Dncase}, \ref{2A2n-1}, \ref{3D4} and \ref{2E6} we will treat the
case $(3)$ for all affine root systems $\Phi$ distinct from $\tt A^{(2)}_{2n}$. Root systems of type $\tt A^{(2)}_{2n}$ will be considered separately in Section
\ref{2A2n} for $n\ge 2$ and in Section \ref{2A2} for $n=1$.

\begin{prop}\label{twistedpropergradient}
 Let $\Phi, \Psi$ as before. If $\Psi\le \Phi$ is a maximal closed subroot system and $\Gr(\Psi)$ is a proper closed subroot system of $\Gr(\Phi)$, 
 then $\Gr(\Psi)<\Gr(\Phi)$ is a maximal closed subroot system such that it contains at least one short root. In this case, we have $\Psi= \wh{\Gr(\Psi)}$.
 \end{prop}
\begin{pf}
  The proof that $\Gr(\Psi)<\Gr(\Phi)$ is a maximal closed subroot system follows immediately from the part (2) of Proposition \ref{untwisted02}. 
  Now,  suppose that $\mathrm{Gr}(\Psi)$ contains only long roots.
  Then it is easy to see that $\Psi\le \widehat{\mathring{\Phi}_\ell}$. But $\Psi\le\widehat{\mathring{\Phi}_\ell}\subsetneq 
  \Omega= \left\{\a+mr\delta:\a\in\mathring{\Phi},\ r\in\mathbb{Z}\right\}$ and $\Omega$ is a closed subroot system of $\Phi$, which is a contradiction to the fact that
  $\Psi$ is maximal closed.
\end{pf}

Now,  we present our main classification theorem for the maximal closed subroot systems of twisted affine root system $\Phi$ (which is not of type $\tt A^{(2)}_{2n}$) whose gradient subroot system is equal to
$\mathring{\Phi}$ or is a proper closed subroot system of $\mathring{\Phi}$.

\begin{thm}\label{twistedgradient=PhinotA2}
Let $\Phi$ be an irreducible twisted affine root system which is not of type $\tt A^{(2)}_{2n}$ and 
let  $\Psi$ be a maximal closed subroot system of $\Phi$.
\begin{enumerate}
 \item 
If $\Gr(\Psi)= \mathring{\Phi}$,
then there exists a $\mathbb{Z}$--linear function $p:\Gr(\Psi)\to \mathbb{Z}$ and a prime number $n_\Psi$ such that p satisfies the condition (\ref{palpha}), $p_\a\in m\mathbb{Z}$ for long roots $\a$ and
 $$\Psi=
 \begin{cases}
   \{\alpha+(p_\alpha+rn_\Psi)\delta, \beta+(p_\beta+mrn_\Psi)\delta : \alpha\in \mathring{\Phi}_s,  \beta\in \mathring{\Phi}_\ell,  r\in\mathbb{Z}\} \  \text{if} \ m \neq n_\Psi,\\
 \{\alpha+(p_\alpha+rn_\Psi)\delta  : \alpha\in \mathring{\Phi}, \ r\in\mathbb{Z} \}\ \text{if} \ m = n_\Psi. \\
 \end{cases}
$$
Conversely,   given a prime number $n_\Psi$ and a $\mathbb{Z}$--linear function $p:\mathring{\Phi}\to \mathbb{Z}$
satisfying  $p_\a\in m\mathbb{Z}$ for long roots $\a\in\mathring{\Phi}_\ell$ and (\ref{palpha}), the subroot system $\Psi$ defined above gives us a maximal closed subroot system of $\Phi.$
 
\vskip 1.5mm
\item If $\Gr(\Psi)\subsetneq \mathring{\Phi}$ is a proper closed subroot system,  then $\Gr(\Psi)<\mathring{\Phi}$ is a maximal closed subroot system 
such that it contains at least one short root and 
in this case $\Psi= \wh{\Gr(\Psi)}$. Conversely,   if $\mathring{\Psi}\subsetneq \mathring{\Phi}$ is a maximal closed subroot system with a short root then $\widehat{\mathring{\Psi}}$
is a maximal closed subroot system. 
\end{enumerate}

\end{thm}

\begin{rem}
For Case $(1)$ i.e., $\Gr(\Psi)=\mathring\Phi$, the type of $\Psi$ is $\tt X_n^{(2)}$ if the type of $\mathring\Phi$ is $\tt X_n$ and $m\neq n_{\Psi}$ and the type of $\Psi$ is
$\tt X_n^{(1)}$ if the type of $\mathring\Phi$ is $\tt X_n$  and $m = n_{\Psi}$.
For Case $(2)$, the type of $\Psi$ is $\tt X_{n_1}^{(r_1)}\oplus\cdots\oplus X_{n_s}^{(r_s)}$, where
$\tt X_{n_i}$'s are irreducible components of $\Gr(\Psi)$ and $r_i=1$ if $\tt X_{n_i}$ is simply-laced else it is $2$.
\end{rem}
\emph{Proof of Statement $(1)$}.
The forward part of Statement (1) is clear from the parts $(2)$ and $(3)$ of Proposition \ref{twisted01}. Converse part of Statement $(1)$
will be proved case by case.
\vskip 1.5mm
\noindent
Case (1.1). First assume that $n_\Psi$ is a prime number such that $n_\Psi\neq m$ and $\Psi=\{\alpha+(p_\alpha+rn_\Psi)\delta, \beta+(p_\beta+mrn_\Psi)\delta : \alpha\in \mathring{\Phi}_s,  \beta\in \mathring{\Phi}_\ell,  r\in\mathbb{Z}\}$
where $p_\a$ satisfies the condition (\ref{palpha}) and $p_\a\in m\mathbb{Z}$ for long roots $\a$. It is easy to verify that $\Psi$ is a closed subroot system of $\Phi.$ 
By the definition of $\Psi$, we have $$\text{$Z_\a(\Psi)=p_\a+n_\Psi\mathbb{Z}$ for $\a\in \mathring{\Phi}_s$
and $Z_\a(\Psi)=p_\a+mn_\Psi\mathbb{Z}$ for $\a\in \mathring{\Phi}_\ell$.}$$
Now,  we will prove that $\Psi$ is a maximal closed subroot system of $\Phi.$
Suppose $\Psi\subsetneq \Delta \subseteq \Phi$ for some closed subroot system $\Delta$ of $\Phi.$
Then we claim that $\Delta$ must be equal to $\Phi.$ 
Since $\Psi\subseteq \Delta$, we have $\Gr(\Psi)=\Gr(\Delta)=\mathring{\Phi}$.
By part $(2)$ of Proposition \ref{twisted01}, $n_\Delta$ determines 
the subgroups $Z_\a'(\Delta)$ and hence the cosets $Z_\a(\Delta)$.
But by part $(2)$ Lemma \ref{keylemma}, we get $n_\Delta$ divides $n_\Psi.$ This implies that  either $n_\Delta=1$ or $n_\Delta=n_\Psi$ since $n_\Psi$ is a prime number.
Assume first that $n_\Delta=n_\Psi$, then we get $n_\a^\Psi=n_\a^\Delta$ for all $\a\in \mathring{\Phi}$ by part $(2)$ of Proposition \ref{twisted01}.
Since $\Gr(\Psi)=\Gr(\Delta)$ and $n_\a^\Psi=n_\a^\Delta$ for all $\a\in \Gr(\Delta)$, we have $\Psi=\Delta$ using part $(3)$ of Lemma \ref{keylemma}, a contradiction.
So, $n_\Delta$ must be equal to $1$. In this case,  
we get $n_\a^\Delta=n_\a^\Phi$ for all $\a\in \mathring{\Phi}$
 again using the part $(2)$ of Proposition \ref{twisted01}.
This immediately implies that $\Delta=\Phi$ by part $(3)$ of Lemma \ref{keylemma}, since $\Gr(\Delta)=\Gr(\Phi)$. 
\vskip 1.5mm
\noindent
Case (1.2). Now
assume that $n_\Psi=m$ and $\Psi=\{\alpha+(p_\alpha+rm)\delta  : \alpha\in \mathring{\Phi}, \ r\in\mathbb{Z} \}.$ One easily sees that $\Psi$ is a closed subroot system of $\Phi.$
So, it remains to show that $\Psi$ is a maximal closed subroot system of $\Phi.$
Suppose $\Psi\subsetneq \Delta \subseteq \Phi$ for some closed subroot system $\Delta$ of $\Phi.$ Then we need to prove that $\Delta$ must be equal to $\Phi.$
Since $\Psi\subseteq \Delta$, we get $\Gr(\Delta)=\mathring{\Phi}$ and by part $(2)$ of Lemma \ref{keylemma}, we get $n_{\Delta}=m$ or  $n_{\Delta}=1$.
If  $n_\Delta=m$, then by part $(2)$ of Proposition \ref{twisted01}, we get $n_\a^{\Delta}=n_\a^{\Psi}$ for all $\a\in \mathring{\Phi}$. This forces $\Delta=\Psi$, a contradiction.
So, this case does not arise. Hence
we must have  $n_\Delta=1$ which implies that  $\Delta=\Phi$ as before in Case 1.1. This completes the proof of Statement $(1)$.

\vskip 1.5mm
\noindent
\emph{Proof of Statement $(2)$}.
The forward part of
Statement (2) is clear from the Proposition \ref{twistedpropergradient}.
 Conversely,  suppose ${\mathring{\Psi}}$ is a maximal closed subroot system in ${\mathring{\Phi}}$ such that it contains at least one short root, say $\beta \in \mathring{\Psi}$, then
 we claim that the lift $\widehat{\mathring{\Psi}}$ in $\Phi$ must be a maximal closed subroot system.
 Let $\Delta$ be a closed subroot system in $\Phi$ such that 
 $$\widehat{\mathring{\Psi}}\subsetneq\Delta\subseteq\Phi.$$ Then we need to prove that $\Delta$ must be equal to $\Phi.$
We observe the following facts first.
\begin{enumerate}
\item By considering respective gradients, we have $\mathring{\Psi}\subseteq \Gr(\Delta)\subseteq \mathring\Phi$. This implies  that
$\mathrm{rank}(\mathring{\Psi})\le \mathrm{rank}(\Gr(\Delta))\le \mathrm{rank}(\mathring\Phi).$
 \item  By Proposition \ref{twistedgradient}, we know that $\Gr(\Delta)$ is either closed in $\mathring\Phi$ or semi-closed in $\mathring\Phi$.
 \item Since $\mathring{\Psi}$ contains the short root $\beta$, we have $\beta+r\delta\in \widehat{\mathring{\Psi}}\subseteq \Delta$ for all $r\in \mathbb{Z}$. 
\item  $\mathring{\Psi}$ can be both irreducible and reducible subroot system of $\mathring{\Phi}$ (see Table \ref{mrs3} and  \cite[Page 136]{kane}). 
\end{enumerate}
Now,  we will deal with all possible cases of $\Delta$. We begin with the easiest case.
 
\vskip 1.5mm
\noindent
Case (2.1). Assume that $\Gr(\Delta)$ is closed in $\mathring{\Phi}$.  Then we claim that $\Gr(\Delta)=\mathring\Phi$.
Since ${\mathring{\Psi}}$ is maximal closed in $\mathring{\Phi}$ and $\mathring{\Psi}\subseteq \Gr(\Delta)\subseteq \mathring\Phi$, we must have
either $\Gr(\Delta)=\mathring{\Psi}$ or $\Gr(\Delta)=\mathring{\Phi}$. If $\Gr(\Delta)=\mathring{\Psi}$, then we have  $\Delta\subseteq \widehat{\mathring{\Psi}}$, a contradiction.
So, we must have $\Gr(\Delta)=\mathring\Phi$.
Since $n_\Delta=1$, we get $n_\ell^\Delta=m$ by part $(2)$ Proposition \ref{twisted01}. Hence, we get $\Delta=\Phi$ by part $(3)$ of Lemma \ref{keylemma}.

\vskip 1.5mm
\noindent
Case (2.2).  Now,  we are left with the case that $\Gr(\Delta)$ is not closed but semi-closed in $\mathring\Phi$. We will prove that this case also can not arise. 
Let $\Gr(\Delta)$ be not closed but semi-closed in $\mathring\Phi$. 
By Proposition \ref{twistedgradient}, there exists short roots $\alpha_1, \alpha_2\in \Gr(\Delta)$ such that $\a_1+\a_2$ is a 
long root and $\a_1+\a_2\in \mathring{\Phi}\backslash \Gr(\Delta)$, fix these short roots $\a_1$ and $\a_2\in \Gr(\Delta)$.
First we observe that $\Gr(\Delta)$ can not be irreducible. 
Otherwise,  $\Gr(\Delta)$ is irreducible and $\beta+r\delta\in\Delta$ for all $r\in \mathbb{Z}$ would imply $n_\Delta=1$ and hence $n_\ell^\Delta=m$ by part $(2)$ of Proposition \ref{twisted01}.
Since we have  $\alpha_1+r\delta,\alpha_2+r\delta\in \Delta$ for all $r\in \mathbb{Z}$, which implies that  $(\alpha_1+\alpha_2)+m\delta=(\alpha_1+(m-1)\delta)+(\alpha_2+\delta)\in \Delta$,
a contradiction to the fact that $\a_1+\a_2\notin \Gr(\Delta)$. So, $\Gr(\Delta)$
must be reducible.  Let $\Gr(\Delta)=\Delta_1\oplus \cdots \oplus \Delta_k$ be the decomposition of $\Gr(\Delta)$ into irreducible components.
Then it is immediate that $\mathrm{rank}(\Gr(\Delta))=\mathrm{rank}(\Delta_1)+ \cdots + \mathrm{rank}(\Delta_k).$ 

\vskip 1.5mm
\noindent
Case (2.2.1). We now consider the case when $\mathring{\Psi}$ is irreducible. Since $\mathring{\Psi}$ is irreducible, it must be contained in one of components of $\Gr(\Delta)$. Without loss of generality
we can assume that $\mathring{\Psi}\subseteq \Delta_1.$ 
We have either $$\mathrm{rank}(\mathring{\Psi})=\mathrm{rank}(\mathring{\Phi}) \ \text{or}  \ \mathrm{rank}(\mathring{\Psi})=\mathrm{rank}(\mathring{\Phi})-1$$
since $\mathring{\Psi}$ is irreducible maximal closed subroot system of $\mathring{\Phi}$
 (see Table \ref{mrs3} and \cite[Page 136]{kane}). If  
 $\mathrm{rank}(\mathring{\Psi})=\mathrm{rank}(\mathring{\Phi})$, then we get $\mathrm{rank}(\Delta_i)=0,\ \text{for all} \ i=2,\cdots, k$ which is a contradiction to the fact that 
$\Gr(\Delta)$ is reducible. So, we get $\mathrm{rank}(\mathring{\Psi})=\mathrm{rank}(\mathring{\Phi})-1$. Since
$$\mathrm{rank}(\Delta_2)+ \cdots + \mathrm{rank}(\Delta_k)\le \mathrm{rank}(\mathring{\Phi})-\mathrm{rank}(\mathring{\Psi})=1,$$
we must have $k=2$ and $\mathrm{rank}(\Delta_2)=1$. 
This implies that   $\Gr(\Delta)=\Delta_1\oplus A_1$ with $\mathring{\Psi}\subseteq \Delta_1$. 
Since $\beta+r\delta\in \Delta$ for all $r\in\mathbb{Z}$ and $\Delta_1$ is irreducible, we have $n_s^{\Delta_1}(\Delta)=1$. 
In particular $\a+r\delta\in \Delta$ for all the short roots $\a\in \Delta_1$ and $r\in\mathbb{Z}.$
Clearly,  one of the short roots $\a_j$, $j=1, 2$ must be in $\Delta_1$, say $\a_1\in \Delta_1$.
Since $\alpha_2\in \Gr(\Delta)$, there exists $r\in \mathbb{Z}$ such that $\alpha_2+r\delta\in \Delta$. 
Now,  $(\alpha_1+\alpha_2)+m\delta=(\alpha_2+r\delta)+(\alpha_1+(m-r)\delta)\in \Delta$ since $\Delta$ is closed and $(\alpha_1+(m-r)\delta)\in \Delta$ because $n_s^{\Delta_1}=1$.
This is again contradicting the fact that $\a_1+\a_2\notin \Gr(\Delta)$.

\vskip 1.5mm
\noindent
Case (2.2.2). We are now left with the case $\mathring{\Psi}$ is reducible. 
Recall that $\mathring{\Phi}$ is non simply-laced irreducible finite crystallographic root system.
So, by the classification of maximal closed subroot systems of the finite root systems 
(see Table \ref{mrs3} and \cite[Page 136]{kane}),
we know that we must have $\mathrm{rank}(\mathring{\Psi})=\mathrm{rank}(\mathring{\Phi})$ and $\mathring{\Psi}=\Psi_1\oplus \Psi_2$, where 
$\Psi_1, \Psi_2$ are irreducible components of
$\mathring{\Psi}$ except in the case that when $\mathring{\Phi}=\tt B_n$ and $(\Psi_1, \Psi_2)=(\tt B_{n-2}, A_1\oplus A_1)$.
We will treat the cases $\mathring{\Phi}=\tt B_n$ and $(\Psi_1, \Psi_2)=(\tt B_{n-2}, A_1\oplus A_1)$ separately. 
Since  $\mathrm{rank}(\mathring{\Psi})=\mathrm{rank}(\Gr(\Delta))=\mathrm{rank}(\mathring\Phi)$ and $\Gr(\Delta)$
is reducible, $\mathring\Psi$ can not be contained in 
one single irreducible component of $\Gr(\Delta)$.

\vskip 1.5mm
\noindent
Subcase 1. Assume that $\Psi_1, \Psi_2$ are irreducible, i.e., $(\Psi_1, \Psi_2)\neq (\tt B_{n-2}, A_1\oplus A_1).$  Since $\mathring\Psi$ can not be contained in 
one single irreducible component of $\Gr(\Delta)$ and $\Psi_1, \Psi_2$ are irreducible, we may assume that $\Psi_1\subseteq \Delta_1$, $\Psi_2\subseteq \Delta_2$.
Then $$\mathrm{rank}(\mathring{\Phi})=\mathrm{rank}(\Psi_1)+\mathrm{rank}(\Psi_2)\le \mathrm{rank}(\Delta_1)+ \cdots + \mathrm{rank}(\Delta_k)\le \mathrm{rank}(\mathring{\Phi})$$
implies that $k=2$ and $\mathrm{rank}(\Psi_1)=\mathrm{rank}(\Delta_1)$, $\mathrm{rank}(\Psi_2)=\mathrm{rank}(\Delta_2)$.
Since $\beta\in \mathring{\Psi}$, it must be either in $\Psi_1$ or in $\Psi_2.$
Assume that $\beta\in \Psi_1$, then as before in the Case 2.2.1 we get $n_s^{\Delta_1}(\Delta)=1$. 
Hence, by previous arguments which appear in the Case 2.2.1, we observe that $\Delta_2$ must contain those short roots $\a_1$ and $\a_2$.
Now,  since $\Delta_2$ contains the short roots $\a_1$ and $\a_2$ we  observe that $\Psi_2$ must contain only long roots. 
Otherwise,  we will get $n_s^{\Delta_2}(\Delta)=1$ (since $\widehat{\mathring{\Psi}}\subseteq \Delta$) which will again lead to the contradiction  $\a_1+\a_2\in \Gr(\Delta)$. 
Hence, $\Delta_2$ must be non simply-laced.
Again by the classification, see Table \ref{mrs3} and \cite[Page no. 136]{kane}, we can have only the following possibilities of $(\mathring{\Phi}, \mathring{\Psi})$
such that $\mathring{\Psi}=\Psi_1\oplus \Psi_2$ with simply laced $\Psi_2$: 
$$(\tt B_n, \tt B_{n-1}\oplus A_1), (\tt B_n, \tt B_{n-i}\oplus D_i),\ 3\le i\le n-2, \ \ (F_4, C_3 \oplus A_1), \ (F_4, A_2\oplus A_2), \ \ (G_2, A_1\oplus A_1).$$
 We will prove that these possibilities can not occur. Hence, the case \say{$\mathring{\Psi}$ is reducible} is not possible and hence the case $\Gr(\Delta)$ is semi-closed in $\mathring{\Phi}$ is not possible.
 Recall that $\Psi_2\subseteq \Delta_2$ satisfying the following properties:
 \begin{itemize}
  \item  $\mathrm{rank}(\Psi_2)=\mathrm{rank}(\Delta_2)$, $\Psi_2$ is simply laced and $\Psi_2$ contains only long roots
  \item $\Delta_2$ is non simply-laced
  \item $\Delta_2$ contains the short roots $\a_1$ and $\a_2$ whose sum $\a_1+\a_2$ is a long root in $\mathring{\Phi}$.
 \end{itemize}
 This immediately implies that the cases $(\mathring{\Phi}, \mathring{\Psi})=(\tt B_n, \tt B_{n-1}\oplus A_1), (F_4, C_3\oplus A_1),$ and $(\tt G_2, A_1\oplus A_1)$ are not possible.   
If $(\mathring{\Phi}, \mathring{\Psi})=(\tt F_4, A_2\oplus A_2)$, then $\Delta_2$ must contain $\tt A_2$ properly which implies that  $\Delta_2$ must be $\tt G_2$.
But $\tt G_2$ can not be a subroot system of $\tt F_4$, so this case also does not occur.

\vskip 1.5mm
Now,  consider the case $(\mathring{\Phi}, \mathring{\Psi})=(\tt B_n, \tt B_{n-i}\oplus D_{i})$ with $3\le i\le n-2$. Then we have $\Psi_1=\tt B_{n-i}$ and $\Psi_2=D_{i}$.
Since $\Delta_2$ is non simply-laced irreducible finite root system, the only possibilities of $\Delta_2$ are $\tt B_i, C_i, F_4 $ and $\tt G_2.$ 
We will directly prove that these possibilities can not occur.
By counting the number of short roots in $\tt B_{n-4}\oplus F_4$ and $\tt B_n$, one can easily see that $\tt B_{n-4}\oplus F_4$ can not occur as a subroot system of $\tt B_n.$
Similarly,  $\tt B_{n-2}\oplus G_2$ does not occur as a subroot system of $\tt B_n.$
Since $\tt B_{n-4}\oplus F_4$ and $\tt B_{n-2}\oplus G_2$  can not occur as subroot systems of $\tt B_n$, we can not have $\Delta_2=\tt G_2$ or $\tt F_4$. 
So, we are left with the cases $\Delta_2=\tt B_i$ or $\tt C_i.$ The $\tt D_i$ can not occur as subroot system of $\tt C_i$ with only consisting of long roots, hence 
$\Delta_2$ can not be $\tt C_i.$ Thus
$\Delta_2=\tt B_i$ is the only case remaining, in this case $\tt D_i$ must be the subroot system of $\tt B_i$ consisting of all long roots of $\tt B_i$.
Since  $\Delta_2=\tt B_i$ and $\a_1+\a_2$ is a long root in $\mathring{\Phi}$, we have
$\a_1+\a_2\in \text{$\tt D_i$}\subseteq \Gr(\Delta)$, a contradiction. 

\vskip 1.5mm
\noindent
Subcase 2. Finally we are left with the case $\mathring\Phi=\tt B_n$ and $(\Psi_1, \Psi_2)=(\tt B_{n-2}, A_1\oplus A_1).$ Since  
$\mathring\Psi$ can not be contained in 
one single irreducible component of $\Gr(\Delta)$, we may have two cases. 
\begin{itemize}
 \item[(i)] $\Psi_1=\text{$\tt B_{n-2}$}\subseteq \Delta_1$ and $\Psi_2=\text{$\tt A_1\oplus A_1$} \subseteq \Delta_2$. 
 In this case,  $k=2$, $\mathrm{rank}(\Delta_1)=n-2$ and $\mathrm{rank}(\Delta_2)=2.$
 Since $\beta\in \mathring{\Psi}$, we have either $\beta\in \Psi_1$ or $\beta\in \Psi_2$.
Let $\beta\in \Psi_1.$ This implies that  $n_s^{\Delta_1}(\Delta)=1$ which implies that  $\a_1, \a_2 \notin \Delta_1.$ 
Hence, $\a_1, \a_2\in \Delta_2$ and $\Psi_2$ can not have short roots and must contain only long roots. Thus,  $\Delta_2=\text{$\tt B_2$} \ \text{or} \ \tt G_2,$
not possible like in Subcase 1.
So, $\beta\in \Psi_2\implies n_s^{\Delta_2}(\Delta)=1\implies \a_1, \a_2\in \Delta_1$. This implies that  $\Psi_1$ can not contain short roots and only contain long roots and
$\Delta_1$ must be non simply-laced.
But $\Psi_1=\tt B_{n-2}$ is non simply-laced for $n\ge 4$, so it contains a short root of $\mathring\Phi$. 
If $n=3$ then $\mathrm{rank}(\Delta_1)=1$,  which implies that  $\Psi_1=\Delta_1$. So $\Delta_1$ can not be non simply-laced in this case, again a contradiction. 
So this case is not possible.

\item[(ii)] $\Psi_1=\text{$\tt B_{n-2}$}\subseteq \Delta_1$ and $\Psi_2 \subseteq \Delta_2\oplus \Delta_3$. In this case,  $k=3$, $\mathrm{rank}(\Delta_1)=n-2$ and 
$\Delta_2=\Delta_3=\tt A_1.$
Since sum of two roots from $\Delta_2\oplus \Delta_3=\tt A_1\oplus A_1$ can not be a root again, we must have one of the $\a_j$, $j=1, 2$ in $\Delta_1.$
So, $\beta$ can not be in  $\Psi_1$. Thus,  $\beta\in \Psi_2.$ But this can not happen like in the case (i).

\end{itemize}

This completes the proof.
\qed

\subsection{}
We end this section by listing out all possible types of maximal closed subroot systems of irreducible twisted affine root systems which has closed gradient subroot systems
and we demonstrate how to get this list from the Theorem \ref{twistedgradient=PhinotA2} by a few examples.
\begin{example}
Let $\Phi=\tt D_{n+1}^{(2)}$. Then $\mathring\Phi=\text{$\tt B_n$}=\{\pm\epsilon_i,\pm\epsilon_i\pm\epsilon_j; 1\leq i\neq j\leq n\}$. 
The root system $\tt B_n$ has a maximal closed subroot system  $\Delta$ of type $\tt B_{n-1}$ with a simple system 
$\{\epsilon_2-\epsilon_3,\epsilon_3-\epsilon_4,\cdots\epsilon_{n-1}-\epsilon_n,\epsilon_n\}$ (see \cite[Page 136]{kane}). Note that $\Delta$ contains short roots. 
By Theorem \ref{twistedgradient=PhinotA2}, $\widehat{\Delta}$ is a maximal closed subroot system of $\Phi$ and the type of $\widehat{\Delta}$ is $\tt D_{n}^{(2)}.$
\end{example}

\begin{example}
Let $\Phi=\tt E_{6}^{(2)}$. Then $\mathring\Phi=\text{$\tt F_4$}=
\big\{\pm\epsilon_i, \pm\epsilon_i\pm\epsilon_j, \frac{1}{2}(\lambda_1\epsilon_1+\lambda_2\epsilon_2+\lambda_3\epsilon_3+\lambda_4\epsilon_4) : 
\lambda_i=\pm1, 1\leq i\neq j\leq 4\big\}$.
The root system $\tt F_4$ has maximal closed subroot system  $\Delta_1$ of type $\tt A_2\oplus A_2$ with a simple system 
$\{\epsilon_1+\epsilon_2,\epsilon_2-\epsilon_3,\epsilon_4,\frac{1}{2}(\epsilon_1-\epsilon_2-\epsilon_3-\epsilon_4)\}$
 and $\Delta_2$ of type $\tt B_4$ with a simple system $\{\epsilon_1+\epsilon_2,\epsilon_2-\epsilon_3,\epsilon_3-\epsilon_4,\epsilon_4\}$
(see \cite[Page 136]{kane}). Note that $\Delta_1, \Delta_2$ both contains short roots.
By Theorem \ref{twistedgradient=PhinotA2}, $\widehat{\Delta_1}$ is a maximal closed subroot system of $\Phi$ of type $\tt A_2^{(1)}\oplus A_2^{(1)}$ and 
$\widehat{\Delta_2}$ is a maximal closed subroot system of $\Phi$ of type $\tt D_5^{(2)}.$
\end{example}

\noindent
The following table is immediate from the Theorem \ref{twistedgradient=PhinotA2} and Table \ref{mrs3}.
\begin{table}[ht]
\caption{Types of maximal closed subroot systems of irreducible twisted affine root systems (not of type $\tt A_{2n}^{(2)}$) with closed gradient subroot systems}
\centering 
\begin{tabular}{|c|c|c|}
\hline
Type & Reducible & Irreducible \\
\hline 
$\tt D_{n+1}^{(2)}$ & $\tt D_{r+1}^{(2)} \oplus D_{n-r}^{(1)}\;$  $\tt (2\le r\le n-2)$ & $\tt B_{n}^{(1)}$, $\tt D_{n+1}^{(2)}$, $D_{n}^{(2)}$\\[1ex]
$\tt A_{2n-1}^{(2)}$ & $\tt A_{2r-1}^{(2)} \oplus A_{2n-2r-1}^{(2)}\;$ $\tt (1\le r\le n-1)$  & $\tt A_{2n-1}^{(2)}$, $\tt C_n^{(1)}$, $\tt A_{n-1}^{(1)}$ \\[1ex]
$\tt E_6^{(2)}$ & $\tt A_1^{(1)} \oplus A_{5}^{(2)}$, $\tt A_{2}^{(1)} \oplus A_{2}^{(1)}$& $\tt E_6^{(2)}$, $\tt F_4^{(1)}$, $D_5^{(2)}$\\[1ex]
$\tt D_4^{(3)}$ & $\tt A_1^{(1)} \oplus A_{1}^{(1)}$ & $\tt D_4^{(3)}$, $\tt G_2^{(1)}$, $\tt A_2^{(1)}$\\[1ex]

\hline
\end{tabular}
\label{mrs7}
\end{table}
\begin{rem}
Note that the Table \ref{mrs7} gives us only the part of the classification. The list in the Table \ref{mrs7} has already appeared in \cite{FRT} (see \cite[Table 1 \& 2]{FRT}) and note that the authors of \cite{FRT} have 
 omitted the possibility of a maximal closed subroot system $\tt A_{2}^{(1)}\oplus \tt A_{2}^{(1)}\subset \tt E_{6}^{(2)}$ and $\tt D_{5}^{(2)}\subset \tt E_{6}^{(2)}$
in their list. 
\end{rem}
\vskip 1mm
\noindent
{\em{We are now left with the case $(3)$ of Proposition \ref{twistedgradient} (in twisted affine root systems which is not of type $\tt A^{(2)}_{2n}$) and the type $\tt A^{(2)}_{2n}$
in completing the classification theorem.
The aim of the remaining part of this paper is to consider the case $(3)$ of Proposition \ref{twistedgradient} and the type $\tt A^{(2)}_{2n}$. 
The case $(3)$ of Proposition \ref{twistedgradient} requires a type by type analysis so, in Section \ref{D2n}, \ref{2A2n-1}, \ref{3D4}, \ref{2E6}  we consider the types 
$\tt D^{(2)}_{n+1}$, $\tt A^{(2)}_{2n-1}$,$\tt D^{(3)}_{4}$ and $\tt E^{(2)}_{6}$ separately.
Finally we will deal the types $\tt A^{(2)}_{2n}, n\neq 1$ and $\tt A^{(2)}_{2}$  in Sections  \ref{2A2n} and \ref{2A2}. We will denote
$I_n=\{1, \cdots, n\}$ in what follows next.}}


\section{The case $\tt D^{(2)}_{n+1}$}\label{D2n}
Throughout this section we assume that $\Phi$ is of type $\tt D^{(2)}_{n+1}$. In particular, the gradient root system of $\tt D^{(2)}_{n+1}$ is of type $\tt B_n$.
We have the following explicit description of $\tt D^{(2)}_{n+1}$, see \cite[Page no. 545, 579]{carter}:
$$\Phi=\left\{\pm\epsilon_i+r\delta,\pm\epsilon_i\pm\epsilon_j+2r\delta : r \in \mathbb{Z}, 1\leq i\neq j\leq n \right\}$$ and 
$$\mathring{\Phi}=\left\{\pm\epsilon_i, \pm\epsilon_i\pm\epsilon_j : 1\leq i\neq j\leq n\right\}.$$
\subsection{} We need the following definition.
\begin{defn}\label{defnpsiI2Dn+1}
 For a subset $I\subseteq I_n$, we define
  \begin{align*}
\text{$\Psi_I(\tt D^{(2)}_{n+1})$}=&\big\{\pm\epsilon_s+2r\delta: s\in I, r \in \mathbb{Z}\big\}\cup  \big\{\pm\epsilon_s+(2r+1)\delta: s\notin I, r \in \mathbb{Z}\big\} &\\&
 \cup\big\{\pm\epsilon_s\pm\epsilon_t+2r\delta:  s\neq t,\ \ s,t\in I \ \text{or} \ s, t\notin I,\ r \in \mathbb{Z}\big\}.
   \end{align*}

\end{defn}

\begin{lem}\label{closedpsiI2Dn+1}
 $\text{$\Psi_I(\tt D^{(2)}_{n+1})$}$ is a closed subroot system of $\Phi$  for any subset $I\subseteq I_n$. 
\end{lem}
\begin{pf}
Set $J=I_n \backslash I$. Write $\Psi_I^{\mathrm{even}}=\big\{\pm\epsilon_s+2r\delta: s\in I, r \in \mathbb{Z}\big\}$, $\Psi_J^{\mathrm{odd}}=\big\{\pm\epsilon_s+(2r+1)\delta: s\notin I, r \in \mathbb{Z}\big\}$ and 
 $\Psi_{I\times J}^{\mathrm{even}}=\big\{\pm\epsilon_s\pm\epsilon_t+2r\delta:  s\neq t,\ \ s,t\in I \ \text{or} \ s, t\notin I,\ r \in \mathbb{Z}\big\}.$
 Since the integers appear in the $\delta$ part of elements of $\Psi_I^{\mathrm{even}}$ and  $\Psi_J^{\mathrm{odd}}$ have different parities, their sum can not be a root in $\Phi$ again.
It is clear that if the sum of two roots $\a, \beta\in \Psi_I^{\mathrm{even}}$ (or $\in\Psi_J^{\mathrm{odd}}$)  is again a root in $\Phi$ then $\a+\beta$ must be in $\Psi_{I\times J}^{\mathrm{even}}.$
 Similarly,  if $\a\in \Psi_{I\times J}^{\mathrm{even}}$, $\beta\in \Psi_I^{\mathrm{even}}$ (resp. $\beta\in \Psi_J^{\mathrm{odd}}$)
 and $\a+\beta\in \tt D^{(2)}_{n+1}$ then we must have $\a+\beta\in \Psi_I^{\mathrm{even}}$ (resp. $\a+\beta\in \Psi_J^{\mathrm{odd}}$).
 
 Finally consider the case $\a, \beta\in \Psi_{I\times J}^{\mathrm{even}}$. Write $\a=\pm\epsilon_s\pm\epsilon_t+2r\delta$ and $\beta=\pm\epsilon_u\pm\epsilon_v+2r'\delta$.
 Suppose $\a+\beta\in \tt D^{(2)}_{n+1}$, then we must have $|\{s, t\}\cap \{u, v\}|=1$ and in this case the sign of this common element in $\a$ and $\beta$ must be opposite.
 Since either both $s, t\in I$ or both $s, t\in J$ (and it is true for $u, v$ as well), we must have $\a+\beta\in \Psi_{I\times J}^{\mathrm{even}}.$
 \end{pf}

\begin{prop}\label{2Dncase}
Suppose $\Phi$ is of type $\tt D^{(2)}_{n+1}$ and $\Psi\le \Phi$ is a maximal closed subroot system with
 proper semi-closed gradient subroot system  $\mathrm{Gr}(\Psi)<\mathring{\Phi}$, then there exist a set $I\subsetneq I_n$ such that
$\Psi=\text{$\Psi_I(\tt D^{(2)}_{n+1})$}.$
\end{prop}
\begin{proof}
Since $\Gr(\Psi)$ is a semi-closed subroot system, there exist $i,j \in I_n$ such that $\epsilon_i,\epsilon_j \in \Gr(\Psi)$ but $\epsilon_i+\epsilon_j \notin \Gr(\Psi)$.
We claim that elements of $Z_{\epsilon_i}(\Psi)$ and $Z_{\epsilon_j}(\Psi)$ can not have same parities.
Suppose $Z_{\epsilon_i}(\Psi)$ and  $Z_{\epsilon_j}(\Psi)$ contain same parity elements, say $2r+1\in Z_{\epsilon_i}(\Psi), 2s+1\in Z_{\epsilon_j}(\Psi).$ Then we have
$\epsilon_i+\epsilon_j+2(r+s+1)\delta\in \Psi$ since $\Psi$ is closed, a contradiction to the choices of $i, j$. Proof is same for even integers. 
Hence, without loss of generality we can assume that $Z_{\epsilon_i}(\Psi) \subseteq 2\mathbb{Z}$ and $Z_{\epsilon_j}(\Psi) \subseteq 2\mathbb{Z}+1$.

Now,  we claim that for each $\epsilon_k \in \Gr(\Psi)$ either $Z_{\epsilon_k}(\Psi) \subseteq 2\mathbb{Z}$ or $Z_{\epsilon_k}(\Psi) \subseteq 2\mathbb{Z}+1$. 
Suppose there exists $s, r \in \mathbb{Z}$ such that $\epsilon_k+2s\delta,\epsilon_k+(2r+1)\delta \in \Psi$ with $k \not = i, j$. Then one immediately sees
that $\epsilon_k+\epsilon_i,\epsilon_j-\epsilon_k\in\mathrm{Gr}(\Psi)$ since $\Psi$ is closed and $Z_{\epsilon_i}(\Psi) \subseteq 2\mathbb{Z}$ and $Z_{\epsilon_j}(\Psi) \subseteq 2\mathbb{Z}+1$. This implies that  $\epsilon_i+\epsilon_j \in \Gr(\Psi)$, a contradiction. 
Hence, 
either $Z_{\epsilon_k}(\Psi) \subseteq 2\mathbb{Z}$ or $Z_{\epsilon_k}(\Psi) \subseteq 2\mathbb{Z}+1$ for each $\epsilon_k \in \Gr(\Psi)$.
Define $$I=\left\{k\in I_n : Z_{\epsilon_k}(\Psi)\subseteq 2\mathbb{Z}\right\}.$$
Since $j\notin I$, we have $\text{$\Psi_I(\tt D^{(2)}_{n+1})$}\subsetneq \Phi$. We claim that $\Psi\subseteq \text{$\Psi_I(\tt D^{(2)}_{n+1})$}$.   
Suppose, we have $\pm \epsilon_s\pm \epsilon_t+2r\delta\in \Psi$ with $s\in I$ and $t\notin I$. Since $s\in I$,
we have $\mp\epsilon_s+2r'\delta \in \Psi$ for some $r'\in \mathbb{Z}$. Then we get
$$ \text{$(\pm \epsilon_s\pm \epsilon_t+2r\delta)+(\mp\epsilon_s+2r'\delta)\in \Phi$ implies that  $\pm \epsilon_t+2(r+r')\delta\in \Psi$} $$since $\Psi$ is closed.
This implies that  $2(r+r')\in Z_{\epsilon_t}(\Psi)$,
a contradiction to the choice of $t.$ Since 
$\Psi$ is maximal closed subroot system, we have $\Psi=\text{$\Psi_I(\tt D^{(2)}_{n+1})$}.$ This completes the proof.

\end{proof}

\vskip 1.5mm
\noindent
Conversely, given a proper subset $I\subsetneq I_n$, we will show that $\text{$\Psi_I(\tt D^{(2)}_{n+1})$}$ defined above in the Definition \ref{defnpsiI2Dn+1}
must be a maximal closed subroot system of $\Phi.$
\begin{prop} Suppose $\Phi$ is of type $\tt D^{(2)}_{n+1}$.
For $I\subsetneq I_n$, we have $\text{$\Psi_I(\tt D^{(2)}_{n+1})$}$
is a maximal closed subroot system of $\Phi.$ The type of $\text{$\Psi_I(\tt D^{(2)}_{n+1})$}$ is $\tt B_{r}^{(1)}\oplus B_{n-r}^{(1)}$, where $|I|=r$.
\end{prop}

\begin{proof}
We have already seen in Lemma \ref{closedpsiI2Dn+1} that $\text{$\Psi_I(\tt D^{(2)}_{n+1})$}$ is a closed subroot system of $\Phi$. So, it only remains to prove that $\text{$\Psi_I(\tt D^{(2)}_{n+1})$}$ is maximal closed in $\Phi$.
Suppose $\Omega$ is a closed subroot system of $\Phi$ such that
$\text{$\Psi_I(\tt D^{(2)}_{n+1})$}\subsetneq\Omega\subseteq \Phi$, then we claim that $\Omega=\Phi.$ Since $\text{$\Psi_I(\tt D^{(2)}_{n+1})$}\subsetneq \Omega$, there are three possibilities
for elements of $\Omega\backslash \text{$\Psi_I(\tt D^{(2)}_{n+1})$}$. We have either
 \begin{enumerate}
  \item  $\epsilon_s+(2r+1)\delta \in\Omega$ for some $r \in \mathbb{Z}$ and $s\in I$ or
  \item  $\epsilon_s+2r\delta \in\Omega$ for some $r \in \mathbb{Z}$ and $s \notin I$ or
  \item  $\epsilon_s\pm \epsilon_t+2r\delta \in \Omega$ for some $r \in \mathbb{Z}$, $s \in I$ and $t \notin I$.
 \end{enumerate}
In each of the cases, we repeatedly use the fact that $\Omega$ is closed in $\Phi$ and  $\text{$\Psi_I(\tt D^{(2)}_{n+1})$}\subseteq \Omega$ and prove that $\Omega=\Phi.$
\vskip 1.5mm
\noindent
Case $(1)$. Suppose there exists $\epsilon_s+(2r+1)\delta \in\Omega$ for some $r \in \mathbb{Z}$ and $s \in I$. 
 By adding $$ \text{ $\epsilon_s+(2r+1)\delta$ with 
 $\epsilon_t+(2\mathbb{Z}+1)\delta$ for $t\notin I$,
 we get  $\epsilon_s+\epsilon_t+2\mathbb{Z}\delta\subseteq \Omega$ for all $t \notin I$.}$$
 And by  adding $-\epsilon_s-2r\delta\in \Omega$ with $\epsilon_s+\epsilon_t+2\mathbb{Z}\delta$ for $t\notin I$, we get
 $\epsilon_t+2\mathbb{Z}\delta\subseteq \Omega$ for all $t \notin I$ which implies that  
 $\epsilon_t+\mathbb{Z}\delta\subseteq \Omega$ for all $t\notin I$. 
 Similarly,  by adding 
 $-\epsilon_s-(2r+1)\delta \in\Omega$ with  $\epsilon_s+\epsilon_t+2\mathbb{Z}\delta\subseteq \Omega$ for $t\in I$, where $s\not = t$, 
 we get $\epsilon_t+(2\mathbb{Z}+1)\delta\subseteq \Omega$ for all $t\in I$ with $s\not = t.$ 
 Now,  fix $t\notin I$ and by adding $-\epsilon_t-(2r+1)\delta\in \Omega$ with
 $\epsilon_s+\epsilon_t+2\mathbb{Z}\delta\subseteq \Omega$, we get $\epsilon_s+(2\mathbb{Z}+1)\delta\subseteq \Omega.$
 This implies that  $\epsilon_t+\mathbb{Z}\delta\subseteq \Omega$ for all $t\in I$.
 Thus,  we have $\epsilon_t+\mathbb{Z}\delta\subseteq \Omega$ for all $t\in I_n.$ Since $\Omega$ is closed subroot system, this immediately implies that $\Omega=\Phi.$

\vskip 1.5mm
\noindent
Case $(2)$. Suppose there exists $\epsilon_s+2r\delta \in\Omega$ for some $r \in \mathbb{Z}$ and $s \notin I$. 
 By adding $\epsilon_s+2r\delta$ with 
 $\epsilon_t+2\mathbb{Z}\delta$ for $t\in I$,
 we get  $\epsilon_s+\epsilon_t+2\mathbb{Z}\delta\subseteq \Omega$ for all $t\in I$.
 And by  adding $-\epsilon_s-(2r+1)\delta\in \Omega$ with $\epsilon_s+\epsilon_t+2\mathbb{Z}\delta$ for $t\in I$, we get
 $\epsilon_t+(2\mathbb{Z}+1)\delta\subseteq \Omega$ for all $t \in I$. This implies that  
 $\epsilon_t+\mathbb{Z}\delta\subseteq \Omega$ for all $t\in I$. 
 Similarly,  by adding $-\epsilon_s-2r\delta \in\Omega$ with  $\epsilon_s+\epsilon_t+2\mathbb{Z}\delta\subseteq \Omega$ for $t\notin I$, where $s\not = t$, 
 we get $\epsilon_t+2\mathbb{Z}\delta\subseteq \Omega$ for all $t\notin I$ with $s\not = t.$ 
 Now,  fix $t\in I$ such that $t\neq s$ and by adding $-\epsilon_t-2r\delta\in \Omega$ with
 $\epsilon_s+\epsilon_t+2\mathbb{Z}\delta\subseteq \Omega$ we get $\epsilon_s+2\mathbb{Z}\delta\subseteq \Omega.$ 
 This implies that  
 $\epsilon_t+\mathbb{Z}\delta\subseteq \Omega$ for all $t\notin I$. 
 Thus,  we proved $\epsilon_t+\mathbb{Z}\delta\subseteq \Omega$ for all $t\in I_n.$ Since $\Omega$ is closed subroot system, we immediately get $\Omega=\Phi.$

\vskip 1.5mm
\noindent
Case $(3).$ Finally assume that $\epsilon_s\pm \epsilon_t+2r\delta \in \Omega$ for some $r \in \mathbb{Z}$, $s \in I$ and $t \notin I$. 
 Add $\mp \epsilon_t-(2r+1)\delta\in\Omega$ with $\epsilon_s\pm \epsilon_t+2r\delta \in \Omega$ then we get $\epsilon_s+\delta\in \Omega.$ Thus,  we are back to the Case $(1)$.
This completes the proof.
\end{proof}

\begin{rem}
 The authors of \cite{FRT} have 
 omitted the possibility of a maximal closed subroot system $\tt B_{r}^{(1)}\oplus \tt B_{n-r}^{(1)}\subset \tt D_{n+1}^{(2)}$
in their classification list, see \cite[Table 1 \& 2]{FRT}. 
\end{rem}

\section{The case $\tt A^{(2)}_{2n-1}$}\label{2A2n-1}
Throughout this section we assume that $\Phi$ is of type $\tt A^{(2)}_{2n-1}$. In particular, the gradient root system of $\tt A^{(2)}_{2n-1}$ is of type $\tt C_n$.
We have the following explicit description of $\tt A^{(2)}_{2n-1}$, see \cite[Page no. 547, 573]{carter}:
$$\Phi=\left\{\pm2\epsilon_i+2r\delta,\pm\epsilon_i\pm\epsilon_j+r\delta : r \in \mathbb{Z},1\leq i\neq j\leq n\right\}$$ and 
$\mathring{\Phi}=\{\pm2\epsilon_i, \pm\epsilon_i\pm\epsilon_j: 1\leq i\neq j\leq n\}.$

Consider $\mathring{\Phi}_s=\big\{\pm\epsilon_i\pm\epsilon_j:  i,j \in I_n, i\neq j\big\}=:\mathcal{D}_n$. Clearly,  the short roots $\mathring{\Phi}_s$ for a root system of type $\tt D_{n}$ (see \cite[Page no. 146]{carter}) and
$$\Gamma_n=\{\alpha_1=\epsilon_{1}-\epsilon_{2},\cdots , \alpha_{n-1}=\epsilon_{n-1}-\epsilon_{n}, \alpha_n=\epsilon_{n-1}+\epsilon_{n}\}$$ is a simple root system of $\mathcal{D}_n$.
It is easy to see that $\epsilon_{s}-\epsilon_t=\alpha_s+\cdots +\alpha_{t-1}$ and 
$$\epsilon_{s}+\epsilon_t=
\begin{cases}
   \alpha_s+\cdots +\alpha_{t-2}+\alpha_{t}& \text{if} \ t=n, \\
   \alpha_s+\cdots +\alpha_{t-1}+2(\alpha_t+\cdots +\alpha_{n-2})+\alpha_{n-1}+\alpha_n & \text{if} \ t<n.
\end{cases}$$
\subsection{}
Let $p: \Gamma_n\to \{0, 1\}$ be a function such that $p_{\alpha_{n-1}}$ and $p_{\alpha_n}$ have different parity and  
let $p:\mathcal{D}_n\to \mathbb{Z}$ be its $\mathbb{Z}$--linear extension  given by $\pm\epsilon_s\pm\epsilon_t\mapsto p_{\pm\epsilon_s\pm\epsilon_t}$. 
Since the map $p$ is $\mathbb{Z}$--linear, we have $$ p_{\epsilon_{s}+\epsilon_t}=
\begin{cases}
 p_{\epsilon_{s}-\epsilon_t}-p_{\a_{t-1}}+p_{\alpha_{t}}& \text{if} \ t=n, \\
   p_{\epsilon_{s}-\epsilon_t}+2(p_{\alpha_t}+\cdots +p_{\alpha_{n-2}})+p_{\alpha_{n-1}}+p_{\alpha_n} & \text{if} \ t<n.
\end{cases}
$$
This implies
$p_{\epsilon_{s}-\epsilon_t}$ and $p_{\epsilon_{s}+\epsilon_t}$ have different parity for $s<t$.
Since $p_{\epsilon_{s}-\epsilon_t}=-p_{\epsilon_{t}-\epsilon_s}$, we conclude that $p_{\epsilon_{s}-\epsilon_t}$  and $p_{\epsilon_{s}+\epsilon_t}$ also have different parity for $s>t$.
Now,  define
$$ \text{$\Psi_p(\tt A^{(2)}_{2n-1})$}:= 
\big\{\pm\epsilon_s\pm\epsilon_t+(p_{\pm\epsilon_s\pm \epsilon_t}+2r)\delta : 1\leq t\neq s \leq n, r\in \mathbb{Z}\big\}.$$
\begin{lem}\label{lem2A2n-1}
Let $p: \mathcal{D}_{n}\to \mathbb{Z}$ be a  $\mathbb{Z}-$linear function
   such that $p_{\epsilon_{s}-\epsilon_t}$ and $p_{\epsilon_{s}+\epsilon_t}$ have different parity for each $1\leq s\neq t \leq n$. Then
  $\Psi_p(\tt A^{(2)}_{2n-1})$ is a maximal closed subroot system of $\Phi$.
\end{lem}
\begin{pf} Since $p$ is $\mathbb{Z}$--linear, we have $$\bold s_{\a+(p_\a+2r)\delta}(\beta+(p_\beta+2r')\delta)=\bold s_\a(\beta)+(p_{\bold s_\a(\beta)}+2(r'-r\langle \beta, \a^\vee \rangle))\delta$$
for $\a, \beta\in \mathcal{D}_n$ and $r, r'\in \mathbb{Z}$, where $\bold s_\a$ is the reflection with respect to $\a$ defined in Section \ref{Weylgroup}.
This implies that   $\Psi_p(\tt A^{(2)}_{2n-1})$ is a subroot system of $\Phi$.
For $s, t\in I_n$, $t\neq s$,
we can not have $2\epsilon_s+(p_{\epsilon_{s}-\epsilon_t}+p_{\epsilon_{s}+\epsilon_t}+2r)\delta\in \Phi$ for any $r\in \mathbb{Z}$, 
since $p_{\epsilon_{s}-\epsilon_t}$ and $p_{\epsilon_{s}+\epsilon_t}$ have different parity. This implies that $\Psi_p(\tt A^{(2)}_{2n-1})$ is a closed subroot system of $\Phi$.
Now,  suppose there is a closed subroot system $\Delta$ of $\Phi$ such that $\text{$\Psi_p(\tt A^{(2)}_{2n-1})$}\subsetneq \Delta\subseteq \Phi.$ Then we claim that $\Delta=\Phi$.
Since $\text{$\Psi_p(\tt A^{(2)}_{2n-1})$}\subsetneq \Delta$, we have two possibilities for elements of $\Delta\backslash \text{$\Psi_p(\tt A^{(2)}_{2n-1})$}$. We have either 
\begin{enumerate}
 \item $2\epsilon_s+2r\delta\in\Delta$ for some $s\in I_n$ and $r\in\mathbb{Z}$ or 
 \item $\epsilon_s \pm\epsilon_t+(p_{\epsilon_{s}\pm\epsilon_t}+2r+1)\delta\in\Delta$ 
for some $s \neq t\in I_n$  and $r\in \mathbb{Z}$.
\end{enumerate}

\vskip 1.5mm
\noindent
Case (1).  Suppose there exists $s\in I_n$ such that $2\epsilon_s+2r\delta\in\Delta$ for some $r\in\mathbb{Z}$. Then
since $\epsilon_t-\epsilon_s+(p_{\epsilon_t-\epsilon_s}+2\mathbb{Z})\delta\subseteq \Delta$ for any $t\neq s$, we have
$$\epsilon_t+\epsilon_s+(p_{\epsilon_t-\epsilon_s}+2\mathbb{Z})\delta=(2\epsilon_s+2r\delta)+\epsilon_t-\epsilon_s+(p_{\epsilon_t-\epsilon_s}+2\mathbb{Z})\delta\subseteq \Delta.$$ 
for all $t\in I_n$ with $t\neq s$.
As $\epsilon_t+\epsilon_s+(p_{\epsilon_t+\epsilon_s}+2\mathbb{Z})\delta\subseteq \text{$\Psi_p(\tt A^{(2)}_{2n-1})$}$ and $p_{\epsilon_t+\epsilon_s}$ and $p_{\epsilon_t-\epsilon_s}$ have different parity, 
we get $(\epsilon_t+\epsilon_s)+\mathbb{Z}\delta\subseteq\Delta$ for all $t\neq s.$
This in turn implies that
$$(\epsilon_t+\epsilon_s+p_{\epsilon_t-\epsilon_s}\delta)+\epsilon_t-\epsilon_s+(p_{\epsilon_t-\epsilon_s}+2\mathbb{Z})\delta=2\epsilon_t+2\mathbb{Z}\delta\subseteq \Delta$$ for all $t\in I_n$ with $t\neq s$.
Now,  $\epsilon_t-\epsilon_s+\mathbb{Z}\delta=(2\epsilon_t+2\mathbb{Z}\delta)-(\epsilon_t+\epsilon_s+\mathbb{Z}\delta)\subseteq \Delta$ for all $t\neq s$.
So far we have proved that $2\epsilon_s+2r\delta\in\Delta$ implies that  $\pm\epsilon_t\pm \epsilon_s+\mathbb{Z}\delta$, $\pm2\epsilon_t+2\mathbb{Z}\delta\subseteq \Delta$ for all $t\in I_n$ such that $t\neq s.$ 
By repeating the earlier arguments with all possible $t\in I_n$ such that $t\neq s$, we see that $\Delta=\Phi$. 

\vskip 1.5mm
\noindent
Case (2).
Now,  assume that there exists $s,t\in I_n$ such that $\epsilon_{s}\pm\epsilon_t+(p_{\epsilon_{s}\pm\epsilon_t}+2r+1)\delta\in\Delta$ 
for some $r\in \mathbb{Z}$. Since $\epsilon_{s}\mp\epsilon_t+(p_{\epsilon_{s}\mp\epsilon_t}+2r')\delta\in\Delta$ for all $r'\in\mathbb{Z}$ and
$p_{\epsilon_{s}\pm\epsilon_t}$, $p_{\epsilon_{s}\mp\epsilon_t}$ have different parity, we have
$2\epsilon_s+2r\delta\in\Delta.$
So, we are back to the Case (1) and hence $\Delta=\Phi$. 
This completes the proof.

\end{pf}

\begin{prop}
 Let $\Phi$ be an irreducible affine root system of type $\tt  A^{(2)}_{2n-1}$. Then $\Psi\le \Phi$ is a maximal closed subroot system with a
 proper semi-closed gradient subroot system  $\mathrm{Gr}(\Psi)<\mathring{\Phi}$ if and only if there exist $\mathbb{Z}-$linear function
 $p: \mathcal{D}_{n}\to \mathbb{Z}$  such that $p_{\epsilon_{s}-\epsilon_t}$ and $p_{\epsilon_{s}+\epsilon_t}$ have different parity for each $1\leq s\neq t \leq n$ and
$$ \Psi=\text{$\Psi_p(\tt A^{(2)}_{2n-1})$}=\big\{\pm\epsilon_s\pm\epsilon_t+(p_{\pm\epsilon_s\pm \epsilon_t}+2r)\delta :  1 \leq t\neq s \leq n, r\in \mathbb{Z}\big\}.$$ 
The affine type of $\text{$\Psi_p(\tt A^{(2)}_{2n-1})$}$ is $\tt D_n^{(1)}$.
 \end{prop}
\begin{proof}
Let
$\Psi\le \Phi$ be a maximal closed subroot system with a
 proper semi-closed gradient subroot system  $\mathrm{Gr}(\Psi)<\mathring{\Phi}$.
By Proposition \ref{twistedgradient}, there exist $s,t \in I_n$ such that $\epsilon_s+\epsilon_t,\epsilon_s-\epsilon_t \in \Gr(\Psi)$ but $2\epsilon_s \notin \Gr(\Psi)$. 
Define
$$I=\left\{i\in I_n : 2\epsilon_i\in \Gr(\Psi)\right\}$$ Then it is immediate that $I\subsetneq I_n$ by previous observation.
Suppose that $I\neq \emptyset$.
Then we will prove that $\Psi\subseteq \Psi_I\subsetneq \Phi$, where
$$\Psi_I=\left\{\pm2\epsilon_i+2r\delta,\pm\epsilon_k\pm\epsilon_\ell+r\delta,\pm\epsilon_{k'}\pm\epsilon_{\ell'}+r\delta : 
i\in I_n,k\neq \ell \in I,\ {k'}\neq {\ell'} \notin I, r \in \mathbb{Z}\right\}.$$
It is easy to see that $\Psi_I$ is the lift of the closed subroot system 
$$\{\pm2\epsilon_i, \pm\epsilon_k\pm\epsilon_\ell,\pm\epsilon_{k'}\pm\epsilon_{\ell'} : i\in I_n,\  k, \ell\in I, k\neq \ell,\ \ {k'}, {\ell'} \notin I, {k'}\neq {\ell'} \}$$ of $\mathring{\Phi}$.
So, $\Psi_I$ is a closed subroot system of $\Phi$ by Lemma \ref{closedlemma} and since $I\subsetneq I_n$, it is proper if $I\neq \emptyset$.
Suppose that $\epsilon_i\pm\epsilon_j+r\delta\in\Psi$, for some $i\in I,j\notin I,r\in\mathbb{Z}$. Then 
since $i\in I$, we have $2\epsilon_i+2r'\delta\in \Psi$ for some $r'\in \mathbb{Z}.$
Since $\Psi$ is closed, we have
$$\epsilon_i\mp\epsilon_j+(2r'-r)\delta=2\epsilon_i+2r'\delta-(\epsilon_i\pm\epsilon_j+r\delta)\in\Psi.$$
This implies that  that $\pm(2\epsilon_j+2(r-r')\delta)=(\epsilon_i\pm\epsilon_j+r\delta)-(\epsilon_i\mp\epsilon_j+(2r'-r)\delta)\in\Psi$, a contradiction to the fact that $j\notin I$. So, we have $\Psi\subseteq \Psi_I$. 
Since $\Psi_I$ is a closed subroot system, we must have $\Psi=\Psi_I$ which is absurd as the gradient root system of $\Psi_I$ is closed. So, we must have $I=\emptyset$. 

Since $2\epsilon_i\notin \Gr(\Psi)$ for all $i\in I_n$, the elements in $Z_{\epsilon_i+\epsilon_j}(\Psi)$ and $Z_{\epsilon_i-\epsilon_j}(\Psi)$ must have different parity
for all $1\le i\neq j \le n$. 
Otherwise,  we will get $2\epsilon_i+(r+r')\delta=(\epsilon_i+\epsilon_j+r\delta)+(\epsilon_i-\epsilon_j+r'\delta)\in \Psi$ for some
$r, r'\in\mathbb{Z}$ such that $r\equiv r' \mathrm{mod}\ 2$.
This is contradicting the fact that $2\epsilon_i\notin \Gr(\Psi)$ for all $i\in I_n$.
Hence, by the discussion in Section \ref{pexis}, there exists
$\mathbb{Z}-$linear function $p^\Psi: \mathcal{D}_{n}\to \mathbb{Z}$ such that for each $1\leq i\neq j \leq n$, we have 
$Z_{\epsilon_i+\epsilon_j}(\Psi)\subseteq p_{\epsilon_{i}-\epsilon_j}^\Psi+2\mathbb{Z}$ and 
$Z_{\epsilon_i-\epsilon_j}(\Psi)\subseteq p_{\epsilon_{i}+\epsilon_j}^\Psi+2\mathbb{Z}$ with 
$p_{\epsilon_{i}-\epsilon_j}^\Psi\not\equiv p_{\epsilon_{i}+\epsilon_j}^\Psi(\mathrm{mod}\ 2)$ 
 and
$$\Psi\subseteq \Psi_{p^\Psi}(\tt A^{(2)}_{2n-1})=
\big\{\text{$\pm\epsilon_i\pm\epsilon_j+(p_{\pm\epsilon_i\pm \epsilon_j}^\Psi+2r)\delta :  1 \leq i, j \leq n, i\neq j, r\in \mathbb{Z}$}\big\}$$
Since $\Psi_{p^\Psi}(\tt A^{(2)}_{2n-1})$ is a closed subroot system in $\Phi$ by Lemma \ref{lem2A2n-1}, 
we have the equality $\Psi=\text{$\Psi_{p^{\Psi}}(\tt A^{(2)}_{2n-1})$}$. Converse part is immediate from the Lemma \ref{lem2A2n-1}. 
 This completes the proof.
\end{proof}

\begin{rem}
 The authors of \cite{FRT} have 
 omitted the possibility of a maximal closed subroot system $\tt D_{n}^{(1)}\subset \tt A_{2n-1}^{(2)}$
in their classification list, see \cite[Table 1 \& 2]{FRT}. 
\end{rem}

\section{The case $\tt D^{(3)}_{4}$}\label{3D4}
Throughout this section we assume that $\Phi$ is of type $\tt D^{(3)}_{4}$. In particular, the gradient root system of $\Phi$ is of type $\tt G_2$.
We have the following explicit description of $\tt D^{(3)}_{4}$, see \cite[Page no. 559, 608]{carter}:
$$\Phi=\big\{\epsilon_i-\epsilon_j+r\delta,\pm(\epsilon_i+\epsilon_j-2\epsilon_k)+3r\delta : i, j, k\in I_3, i\neq j,\ r \in \mathbb{Z}\big\}$$ and 
$\mathring{\Phi}=\big\{\epsilon_i-\epsilon_j,\pm(\epsilon_i+\epsilon_j-2\epsilon_k): i, j, k\in I_3, i\neq j\big\}.$
\subsection{} We need the following lemma.
\begin{lem}\label{3D4gradient}
 Suppose $\Phi$ is of type $\tt D^{(3)}_{4}$ and $\Psi\le \Phi$ is a maximal closed subroot system with a proper semi-closed gradient subroot system,
 then $\Gr(\Psi)=\mathring{\Phi}_s.$
 
\end{lem}
\begin{pf}
Since $\Gr(\Psi)$ is semi-closed, then by Proposition \ref{twistedgradient} there exists two short roots $\alpha,\beta \in \Gr(\Psi)$ 
such that $\alpha+\beta \notin\Gr(\Psi)$. 
Since $\bold{s}_{\alpha}(\beta)\in \Gr(\Psi)$ and is another short root different from $\alpha$ and $\beta$, we have $\mathring{\Phi}_s\subseteq \Gr(\Psi)$. 
Since $\mathring{\Phi}_s$ is a maximal subroot system of $G_2$ and $\Gr(\Psi)\neq \mathring{\Phi}$,
we get $\Gr(\Psi)=\mathring{\Phi}_s$.
\end{pf}
\subsection{}
Let $\{i, j, k\}$ be a permutation of $I_3=\{1, 2, 3\}$ and $\ell\in \mathbb{Z}$. Define 
 $$\Psi^+(i, j, k; \ell):=\big\{\epsilon_i-\epsilon_j+3r\delta,\epsilon_j-\epsilon_k+(3r+\ell)\delta,\epsilon_i-\epsilon_k+(3r+\ell)\delta: r\in \mathbb{Z}\big\}$$ and 
 $\Psi(i, j, k; \ell):=\Psi^+(i, j, k; \ell)\cup (-\Psi^+(i, j, k; \ell)).$
\begin{lem}\label{3D4subroot}
 $\Psi(i, j, k; \ell)$ is a subroot system of $\Phi$ for any permutation $\{i, j, k\}$ of $I_3$ and $\ell\in \mathbb{Z}$. 
\end{lem}
\begin{pf}
 Write $\a_1=\epsilon_i-\epsilon_j$, $\a_2=\epsilon_j-\epsilon_k$ and $\a_3=\epsilon_i-\epsilon_k$. Then $(\a_1, \a_2)=-1$ and
 $(\a_1, \a_3)=(\a_2, \a_3)=1.$ This implies that  $\bold s_{\a_1+3r\delta}(\a_2+(3r'+\ell)\delta)= \a_3 + (3(r+r')+\ell)\delta$,
 $\bold s_{\a_1+3r\delta}(\a_3+(3r'+\ell)\delta)= \a_2+ (3(r'-r)+\ell)\delta$ and $\bold s_{\a_2+(3r+\ell)\delta}(\a_3+(3r'+\ell)\delta)=\a_1+3(r'-r)\delta$ are in $\Psi(i, j, k; \ell)$. Similarly,  
we see that $\bold s_\a(\beta)\in \Psi(i, j, k; \ell)$ for all $\a, \beta\in \Psi(i, j, k; \ell).$ This proves that $\Psi(i, j, k; \ell)$ is a subroot system of $\Phi$.
 \end{pf}

 \begin{prop}\label{3D4converse}
$\Psi(i, j, k; \ell)$ is a maximal closed subroot system of $\Psi$ for any permutation $\{i, j, k\}$ of $I_3$ and $\ell\in \mathbb{Z}$ such that $\ell\equiv 1\ \mathrm{or}\ 2 \ (\mathrm{mod}\ 3)$. \end{prop}
 \begin{pf}
 Lemma \ref{3D4subroot} implies that $\Psi(i, j, k; \ell)$ is a subroot system of $\Phi.$
 Since $\ell\equiv 1\ \mathrm{or}\ 2 \ (\mathrm{mod}\ 3)$, we have 
 $$(\epsilon_j-\epsilon_k+(3r+\ell)\delta)+(\epsilon_i-\epsilon_k+(3r'+\ell)\delta)=(\epsilon_i+\epsilon_j-2\epsilon_k+(3(r+r')+2\ell)\delta)\notin \Phi.$$
  It is easy to check that $\a+\beta\in \Phi$ for $\a, \beta \in \Psi(i, j, k; \ell)$ implies that  $\a+\beta\in \Psi(i, j, k; \ell)$ in remaining cases.
 This proves that $\Psi(i, j, k; \ell)$ is a closed subroot system of $\Phi$ when  $\ell\equiv 1\ \mathrm{or}\ 2 \ (\mathrm{mod}\ 3)$.
 So, it remains to prove that it is maximal closed subroot system in $\Phi$.
  Let $\Delta$ be a closed subroot system of $\Phi$ such that 
 $\Psi(i, j, k; \ell)\subsetneq\Delta\subseteq\Phi$. Observe that $\Delta\backslash \Psi(i, j, k; \ell)$ may contain a short root or a long root. 
 There are three possibilities for short roots of $\Delta\backslash \Psi(i, j, k; \ell)$ and it will be considered in the Cases (1), (2) and (3). The possibility
of $\Delta\backslash \Psi(i, j, k; \ell)$ containing a long root is considered in Case (4).
 \vskip 1.5mm
 \noindent
Case (1). Let $\epsilon_i-\epsilon_j+(3r+r')\delta\in\Delta$ for some $r, r'\in\mathbb{Z}$ such that $r'\not\equiv 0 \ (\mathrm{mod}\ 3)$.
 This implies that  $$(\epsilon_i-\epsilon_j+(3r+r')\delta)+(\epsilon_j-\epsilon_k+(\ell+3\mathbb{Z})\delta)=\epsilon_i-\epsilon_k+(\ell+r'+3\mathbb{Z})\delta\subseteq \Delta.$$
 So, $(\epsilon_i-\epsilon_k+(\ell+r'+3\mathbb{Z})\delta)+(\epsilon_k-\epsilon_j+(-\ell+3\mathbb{Z})\delta)=\epsilon_i-\epsilon_j+(r'+3\mathbb{Z})\delta\subseteq \Delta$, and
 $$(\epsilon_j-\epsilon_i+3\mathbb{Z}\delta)+(\epsilon_i-\epsilon_k+(\ell+r'+3\mathbb{Z})\delta)=\epsilon_j-\epsilon_k+(\ell+r'+3\mathbb{Z})\delta\subseteq\Delta.$$ 
 Summing these two we have $\epsilon_i-\epsilon_k+(\ell+2r'+3\mathbb{Z})\delta\subseteq \Delta$.
 This implies that  $\epsilon_i-\epsilon_k+\mathbb{Z}\delta\subseteq \Delta$
 and using this we get $\alpha+r\delta\in\Delta$ for all short roots $\alpha$ and $r \in\mathbb{Z}$.
 Since any long root of $G_2$ can be written as sum of two short roots, we have $\Delta=\Phi$.
 
 \vskip 1.5mm
 \noindent
Case (2). Let $\epsilon_j-\epsilon_k+(3r+r'+\ell)\delta\in\Delta$ for some $r, r'\in\mathbb{Z}$ such that $r'\not\equiv 0 \ (\mathrm{mod}\ 3)$.
Then $$(\epsilon_i-\epsilon_k+\ell\delta)+(\epsilon_k-\epsilon_j-(3r+r'+\ell)\delta)=\epsilon_i-\epsilon_j+(-3r-r')\delta\in \Delta.$$
So, we are back to Case $(1).$ Thus,  we get $\Delta=\Phi.$

 \vskip 1.5mm
 \noindent
Case (3). Let $\epsilon_i-\epsilon_k+(3r+r'+\ell)\delta\in\Delta$  for some $r, r'\in\mathbb{Z}$ such that $r'\not\equiv 0 \ (\mathrm{mod}\ 3)$.
Then $$(\epsilon_i-\epsilon_k+(3r+r'+\ell)\delta)+(\epsilon_k-\epsilon_j-\ell\delta)=\epsilon_i-\epsilon_j+(3r+r')\delta\in \Delta.$$
Again we are back to Case $(1).$ Thus,  we get $\Delta=\Phi.$

 \vskip 1.5mm
 \noindent
Case (4). Finally assume that $\Delta$ contains a long root and let  $\epsilon_s+\epsilon_t-2\epsilon_u+3r\delta\in \Delta$ for some $r\in\mathbb{Z}$ and a permutation $\{s, t, u\}$ of $I_3$. 
Then subtracting a suitable short root from $\epsilon_s+\epsilon_t-2\epsilon_u+3r\delta$ will bring us back to one of the three previous cases and we get $\Delta=\Phi.$

 Hence, $\Psi(i, j, k; \ell)$ is a maximal closed subroot system of $\Phi$.
 \end{pf}
\subsection{}
Conversely,  we prove that any maximal closed subroot system $\Psi$ of $\Phi$ must be of the form 
 $\Psi=\Psi(i, j, k; \ell)$
 for some permutation $\{i, j, k\}$ of $I_3$ and $\ell\in \mathbb{Z}$ satisfying $\ell\equiv 1\ \mathrm{or}\ 2 \ (\mathrm{mod}\ 3)$.
\begin{prop}
  Let $\Phi$ be the affine root system of type $\tt D^{(3)}_{4}$. Then $\Psi\le \Phi$ is a maximal closed subroot system with a proper semi-closed gradient subroot system $\Gr(\Psi)$ 
  if and only if $\Gr(\Psi)=\mathring{\Phi}_s$  and  $\Psi=\Psi(i, j, k; \ell)$ for some 
  permutation $\{i, j, k\}$ of $I_3$ and $\ell\in \mathbb{Z}$ satisfying $\ell\equiv 1\ \mathrm{or}\ 2 \ (\mathrm{mod}\ 3)$. The type of $\Psi(i,j,k;\ell)$ is $\tt A_2^{(1)}$.
\end{prop}
\begin{pf}
Let $\Psi$ be a maximal closed subroot system of $\Phi.$ Then by Lemma \ref{3D4gradient}, we get $\Gr(\Psi)=\mathring{\Phi}_s$ and it is irreducible.
This also implies that  $\Psi$ can not contain any long root of $\Phi.$
 From the Proposition \ref{keypropositionnon2A2n}, we see that $\Psi$ must contain the roots
 $$ \big\{\epsilon_1-\epsilon_2+(p_1+n_sr)\delta,\epsilon_2-\epsilon_3+(p_2+n_sr)\delta, \epsilon_1-\epsilon_3+(p_3+n_sr)\delta: r\in \mathbb{Z}\big\}$$ for some 
 $p_1, p_2, p_3\in \mathbb{Z}$ and $n_s\in \mathbb{Z}$. Since $\Psi$ is closed and does not contain any long roots, we get 
 $p_1-p_2\not\equiv 0\ (\mathrm{mod}\ 3)$ as 
 $\epsilon_1+\epsilon_3-2\epsilon_2+(p_1-p_2)\delta=(\epsilon_1-\epsilon_2+p_1\delta)+(\epsilon_3-\epsilon_2-p_2\delta)\notin \Psi$.
 Similarly,  we get
 $p_2+p_3\not\equiv 0 \ (\mathrm{mod}\ 3)$ and $p_1+p_3\not\equiv 0 \ (\mathrm{mod}\ 3)$. This implies that  $p_1(\mathrm{mod}\ 3)$, $p_2(\mathrm{mod}\ 3)$ and
 $-p_3(\mathrm{mod}\ 3)$ are distinct elements.  Hence, one of the $p_i$ must be $\equiv 0 \ (\mathrm{mod}\ 3)$.
 We claim that there exists a
permutation $\{i, j, k\}$ of $I_3$ such that
 $$\Psi =\big\{\pm(\epsilon_i-\epsilon_j+(q_1+n_sr)\delta), \pm(\epsilon_j-\epsilon_k+(q_2+n_sr)\delta), \pm(\epsilon_i-\epsilon_k+(q_3+n_sr)\delta): r\in \mathbb{Z}\big\},$$ where 
 $q_1, q_2$ and $q_3$ satisfy
 $q_1\equiv 0 \ (\mathrm{mod}\ 3)$, $q_2\equiv q_3 \ (\mathrm{mod}\ 3)$ and $q_3\not\equiv 0 \ (\mathrm{mod}\ 3)$.
  If $p_1\equiv 0 \ (\mathrm{mod}\ 3)$, then take $(q_1, q_2, q_3)=(p_1, p_2, p_3)$ and  take the permutation to be identity.
 If $p_2\equiv 0 \ (\mathrm{mod}\ 3)$, then take $(q_1, q_2, q_3)=(p_2, -p_3, -p_1)$ and take the permutation to be they cycle $(1\ 2\ 3)$ and if
 $p_3\equiv 0 \ (\mathrm{mod}\ 3)$, then take $(q_1, q_2, q_3)=(p_3, -p_2, p_1)$ and take the permutation to be the cycle $(2\ 3).$
 
 Now,  we claim that $n_s \equiv 0 \ (\mathrm{mod}\ 3)$. Suppose not, then 
 there exists $r\in\mathbb{Z}$ such that $rn_s\equiv q_2 \ (\mathrm{mod}\ 3)$ which implies that  
 $\epsilon_i+\epsilon_k-2\epsilon_j+(q_1+rn_s-q_2)\delta=(\epsilon_i-\epsilon_j+(q_1+rn_s)\delta)+(\epsilon_k-\epsilon_j-q_2\delta)\in \Psi$, a contradiction.
 Thus,  there exists a
 permutation $\{i, j, k\}$ of $I_3$ and $\ell\equiv 1 \ \mathrm{or} \ 2\ (\mathrm{mod} \ 3)$ such that $\Psi\subseteq \Psi^+(i, j, k; \ell)$.
Since $\Psi(i, j, k; \ell)$ is closed, we get that  $\Psi=\Psi(i, j, k; \ell)$. This proves the forward part. The converse is clear from the Proposition \ref{3D4converse}.
\end{pf}

\section{The case $\tt E^{(2)}_{6}$}\label{2E6}
Throughout this section we assume that $\Phi$ is of type $\tt E^{(2)}_{6}$. In particular, the gradient root system $\mathring\Phi$ of $\tt E^{(2)}_{6}$ is of type $\tt F_4$.
We have the following explicit description of $\tt E^{(2)}_{6}$, see \cite[Page no. 557, 604]{carter}:
$$\Phi=\big\{\pm\epsilon_i+r\delta, \pm\epsilon_i\pm\epsilon_j+2r\delta, \tfrac{1}{2}(\lambda_1\epsilon_1+\lambda_2\epsilon_2+\lambda_3\epsilon_3+\lambda_4\epsilon_4)+r\delta, : 
\lambda_i=\pm1, 1\leq i\neq j\leq 4, r \in \mathbb{Z} \big\}$$
The short roots of $\mathring{\Phi}$ form a root system of type $\tt D_4$ (\cite[Page no. 147]{carter}). We set
$\mathcal{D}_4:=\mathring{\Phi}_s =\big\{\pm\epsilon_i,\tfrac{1}{2}(\lambda_1\epsilon_1+\lambda_2\epsilon_2+\lambda_3\epsilon_3+\lambda_4\epsilon_4): i\in I_4,  \lambda_j=\pm1, \forall\ j \in I_4\}$ and
$\Gamma_4=\{\epsilon_{2}, \epsilon_{3},\epsilon_{4},\tfrac{1}{2}(\epsilon_1-\epsilon_2-\epsilon_3-\epsilon_4)\}$ is a simple root system of 
$\mathcal{D}_4$.
\subsection{}
Let $p: \Gamma_4\to \mathbb{Z}$  be a function and  
let $p:\mathcal{D}_4\to \mathbb{Z}$ be its $\mathbb{Z}$--linear extension, such that exactly two $p_{\epsilon_i}$ are even and the rest two are odd. 
Define
$$ \text{$\Psi_p(\tt E^{(2)}_{6})$}:= 
\big\{\alpha+(p_{\alpha}+2r)\delta :  \alpha\in \mathcal{D}_4, r\in \mathbb{Z}\big\}\cup \big\{\pm\epsilon_i\pm{\epsilon_j}+2r\delta :p_{\epsilon_i}+p_{\epsilon_j}\in2\mathbb{Z}, r\in \mathbb{Z}\big\}.$$ 
Note that $p_{-\epsilon_i}=-p_{\epsilon_i}$ and  $p_{\epsilon_i}+p_{\epsilon_j}\in2\mathbb{Z}$ if and only if $p_{\epsilon_i}, p_{\epsilon_j}$ have the same parity.
\begin{lem}
 $\Psi_p(\tt E^{(2)}_{6})$ is a closed subroot system of $\Phi$. 
\end{lem}
\begin{pf}
First we prove that  $\Psi_p(\tt E^{(2)}_{6})$ is a subroot system of $\Phi$. 
Since $p$ is $\mathbb{Z}$--linear and satisfies the equation \ref{palpha}, we have
$$\bold s_{\a+(p_\a+2r)\delta}(\beta+(p_\beta+2r')\delta)=\bold s_\a(\beta)+(p_{\bold s_\a(\beta)}+2(r'-r\langle \beta, \a^\vee \rangle))\delta\in\text{$\Psi_p(\tt E^{(2)}_{6})$}.$$
Suppose $\pm\epsilon_i\pm{\epsilon_j}\in \text{$\Psi_p(\tt E^{(2)}_{6})$}$, we have
 $p_{\epsilon_i}$ and $p_{\epsilon_j}$ have the same parity since $p_{\epsilon_i}+p_{\epsilon_j}\in2\mathbb{Z}$. This implies
$p_{\epsilon_k}$ and $p_{\epsilon_\ell}$ also have the same parity by our choice of $p$, where $\{k, \ell\}=I_4\backslash \{i, j\}$. 
So, $p_{\epsilon_k}+p_{\epsilon_\ell}\in2\mathbb{Z}$, and hence 
$\pm\epsilon_k\pm{\epsilon_\ell}+2r\delta\in \text{$\Psi_p(\tt E^{(2)}_{6})$}$ for all $r\in \mathbb{Z}$.
We have
$$\bold s_{\a+(p_\a+2r)\delta}(\pm\epsilon_i\pm{\epsilon_j}+2r'\delta)=\bold s_\a(\pm\epsilon_i\pm{\epsilon_j})+2(r'-(\pm\epsilon_i\pm{\epsilon_j}, \a)(p_\a+2r))\delta,$$ for $\a\in \mathcal{D}_4$ and $r, r'\in \mathbb{Z}$.
Now,  since $$\bold s_{\a}(\pm\epsilon_i\pm{\epsilon_j})\ \text{is a root of the form}\begin{cases}
                                                                \pm\epsilon_i\pm{\epsilon_j}\ \text{if}\ \a=\pm\epsilon_k\\
                                                                \pm\epsilon_i\pm{\epsilon_j}\ \text{or}\ \pm\epsilon_k\pm{\epsilon_\ell}\ \text{if}\ \a=\sum\limits_{r=1}^4\lambda_r\epsilon_r,
                                                               \end{cases}
$$ where $\{k, \ell\}=I_4\backslash \{i, j\}$,
we have
$\bold s_{\a+(p_\a+2r)\delta}(\pm\epsilon_i\pm{\epsilon_j}+2r'\delta)\in\text{$\Psi_p(\tt E^{(2)}_{6})$}.$
It is easy to see that,
$$\bold s_{\pm\epsilon_k\pm\epsilon_\ell+2r\delta}(\epsilon_i\pm\epsilon_j+2r'\delta)=\pm\epsilon_i\pm{\epsilon_j}+2r'\delta\in\text{$\Psi_p(\tt E^{(2)}_{6})$}.$$
Since $p_{\pm\epsilon_i}$ and $p_{\pm\epsilon_j}$ have the same parity, we have $$p_{\a-(\a, \pm\epsilon_i\pm\epsilon_j)(\pm\epsilon_i\pm\epsilon_j)}=p_\a-(\a, \pm\epsilon_i\pm\epsilon_j)(p_{\pm\epsilon_i}+p_{\pm\epsilon_j})\equiv p_\a \ (\mathrm{mod}\ 2).$$
This implies that  
$$\bold s_{\pm\epsilon_i\pm\epsilon_j+2r\delta}(\a+(p_\a+2r')\delta)=(\a-(\a, \pm\epsilon_i\pm\epsilon_j)(\pm\epsilon_i\pm\epsilon_j))+(p_\a+2(r'-(\a, \pm\epsilon_i\pm\epsilon_j)r))\delta\in\text{$\Psi_p(\tt E^{(2)}_{6})$}$$
for $\a\in \mathcal{D}_4$ and $r, r'\in\mathbb{Z}$ since $(\a-(\a, \pm\epsilon_i\pm\epsilon_j)(\pm\epsilon_i\pm\epsilon_j))\in \mathcal{D}_4$ for $\a\in \mathcal{D}_4$.
This proves that 
$\text{$\Psi_p(\tt E^{(2)}_{6})$}$ is a subroot system of $\Phi.$
Now,  we prove that $\text{$\Psi_p(\tt E^{(2)}_{6})$}$ is closed in $\Phi.$ We have the following cases.
\vskip 1.5mm
\noindent
Case (1).
Let $x=(\a+(p_\a+2r)\delta) + (\beta+(p_\beta+2r')\delta)\in \Phi$ for some $\a, \beta\in \mathcal{D}_4$.
If $\a+\beta\in \mathcal{D}_4$, then it is easy to see that $x=(\a+\beta)+(p_{\a+\beta}+2(r+r'))\delta\in\text{$\Psi_p(\tt E^{(2)}_{6})$}$.
If $\a+\beta\notin \mathcal{D}_4$, then $p_\a$ and $p_\beta$ are of the same parity. 
We have the following possibilities when $\a+\beta\notin \mathcal{D}_4$:
\begin{itemize}
\item if $\a=\pm\epsilon_i, \beta=\pm\epsilon_j\in \mathcal{D}_4$, then $x=(\pm\epsilon_i\pm\epsilon_j)+(p_\a+p_\beta+2(r+r'))\delta\in\text{$\Psi_p(\tt E^{(2)}_{6})$}$ 
 since $p_\a\equiv p_{\epsilon_i}\ (\mathrm{mod} \ 2)$ and $p_\beta \equiv p_{\epsilon_j}\ (\mathrm{mod} \ 2)$ have the same parity.

\item if $\a=\tfrac{1}{2}(\lambda_i\epsilon_i+\lambda_j\epsilon_j)+\tfrac{1}{2}(\lambda_k\epsilon_k+\lambda_\ell\epsilon_\ell)$ and
 $\beta=\tfrac{1}{2}(\lambda_i\epsilon_i+\lambda_j\epsilon_j)-\tfrac{1}{2}(\lambda_k\epsilon_k+\lambda_\ell\epsilon_\ell)$, then 
 we have $\a-(\lambda_k\epsilon_k+\lambda_\ell\epsilon_\ell)=\beta$ which implies that   $p_\a-(\lambda_kp_{\epsilon_k}+\lambda_\ell p_{\epsilon_\ell})=p_\beta$. 
 Since $p_\a\equiv p_\beta \ (\mathrm{mod}\ 2)$, we must have $p_{\epsilon_k}\equiv p_{\epsilon_\ell} \ (\mathrm{mod} \ 2)$. Hence, 
$p_{\epsilon_i}$ and  $p_{\epsilon_j}$ are of the same parity by our choice of the function $p$.
This implies
 $x=(\lambda_i\epsilon_i+\lambda_j\epsilon_j)+(p_\a+p_\beta+2(r+r'))\delta\in\text{$\Psi_p(\tt E^{(2)}_{6})$}$.
\end{itemize}

\vskip 1.5mm
\noindent
Case (2). Let $x=(\a+(p_\a+2r)\delta) + (\pm\epsilon_i\pm\epsilon_j+2r'\delta)\in \Phi$ for some $\a\in \mathcal{D}_4$ and $(\pm\epsilon_i\pm\epsilon_j+2r'\delta)\in\text{$\Psi_p(\tt E^{(2)}_{6})$}$.
Since $\a +(\pm\epsilon_i\pm\epsilon_j)\in \Gr(\Phi)$, we have  $\a +(\pm\epsilon_i\pm\epsilon_j)\in\mathcal{D}_4$.
Since $p_{\epsilon_i}$ and $p_{\epsilon_j}$ have the same parity, we have $p_{\a-(\pm\epsilon_i\pm\epsilon_j)}=p_\a-(p_{\pm\epsilon_i}+p_{\pm\epsilon_j})\equiv p_\a \ (\mathrm{mod} \ 2)$.
This implies that  $x=\a +(\pm\epsilon_i\pm\epsilon_j)+(p_\a+2(r+r'))\delta\in \Phi$, since $\a +(\pm\epsilon_i\pm\epsilon_j)\in\mathcal{D}_4$.

\vskip 1.5mm
\noindent
Case (3). Let $x=(\a+2r\delta)+(\beta+2r'\delta)\in \Phi$ for some $\a+2r\delta, \beta+2r'\delta\in \text{$\Psi_p(\tt E^{(2)}_{6})$}$ with $\a, \beta\notin \mathcal{D}_4$. Then we must have
$\a=\pm\epsilon_i\pm\epsilon_j$ and $\beta=\mp\epsilon_j\pm\epsilon_k$ for some $i\neq j, j\neq k\in I_4$. 
Since $p_{\epsilon_i}$ and $p_{\epsilon_j}$ have the same parity and $p_{\epsilon_j}$ and $p_{\epsilon_k}$ have the same parity,
we have $i=k$ by our choice of the function $p$. In this case,  $x$ can not be in $\Phi$, so this case is not possible.
This completes the proof.
\end{pf}

\noindent
Note that $Z_\a(\text{$\Psi_p(\tt E^{(2)}_{6})$})=p_\a+2\mathbb{Z}$ for all $\a\in \mathcal{D}_4$ and 
$Z_{\pm\epsilon_i\pm\epsilon_j}(\text{$\Psi_p(\tt E^{(2)}_{6})$})=2\mathbb{Z}$ for $\pm\epsilon_i\pm\epsilon_j\in \Gr(\text{$\Psi_p(\tt E^{(2)}_{6})$})$, in particular $Z_\a(\text{$\Psi_p(\tt E^{(2)}_{6})$})=2\mathbb{Z}\ \text{or}\ 1+2\mathbb{Z}$ 
depending on $p_\a$ being even or odd.
\begin{lem}\label{2E6maxclosed}
 $\Psi_p(\tt E^{(2)}_{6})$ is a maximal closed subroot system of $\Phi$. 
\end{lem}

 \begin{pf}
Suppose there is a closed subroot system $\Delta$ of $\Phi$ such 
that $\text{$\Psi_p(\tt E^{(2)}_{6})$}\subsetneq \Delta\subseteq \Phi.$ This implies that  that $\Gr(\text{$\Psi_p(\tt E^{(2)}_{6})$})\subseteq \Gr(\Delta)$ and
$Z_\a(\text{$\Psi_p(\tt E^{(2)}_{6})$})\subseteq Z_\a(\Delta)$. 
Note that $Z_\a(\text{$\Psi_p(\tt E^{(2)}_{6})$})=Z_\a(\Delta)=2\mathbb{Z}$ for $\a=\pm\epsilon_i\pm\epsilon_j\in \Gr(\text{$\Psi_p(\tt E^{(2)}_{6})$})$.

So, there are three possibilities for elements of $\Delta\backslash \text{$\Psi_p(\tt E^{(2)}_{6})$}$. 
\vskip 1.5mm
\noindent
Case (1). Suppose $Z_\a(\text{$\Psi_p(\tt E^{(2)}_{6})$})\subsetneq Z_\a(\Delta)$ for some $\a=\epsilon_s$.
Then there exists $r_1,r_2\in\mathbb{Z}$ such that $\epsilon_s+r_1\delta,\epsilon_s+r_2\delta\in\Delta$
and $r_1,r_2$ have different parity. Then either $(p_{\epsilon_t}+r_1)\in 2\mathbb{Z}$ or $(p_{\epsilon_t}+r_2)\in 2\mathbb{Z}$ for each $t\in I_4$ with $t\neq s$.
Hence, $\epsilon_t+\epsilon_s+2\mathbb{Z}\delta=\epsilon_t+(p_{\epsilon_t}+2\mathbb{Z})\delta+\epsilon_s+r_i\delta\subseteq\Delta$ for $i=1$ or 2.
Similarly,  one sees that
$$\pm\text{$\epsilon_t\pm\epsilon_s+2\mathbb{Z}\delta\subseteq\Delta$ for all $t\in I_4, t\neq s$.}$$ Choose $t\in I_4$ such that $p_{\epsilon_t}$ and  $p_{\epsilon_s}$ have 
different parity. Then
$\epsilon_s-(p_{\epsilon_t}+2\mathbb{Z})\delta=\epsilon_t+\epsilon_s+2\mathbb{Z}\delta-(\epsilon_t+p_{\epsilon_t}\delta)\subseteq\Delta$.
This implies that  $\epsilon_s+\mathbb{Z}\delta\subseteq\Delta$. This implies
$\epsilon_t+\mathbb{Z}\delta=(\epsilon_t+\epsilon_s+2\mathbb{Z}\delta)+(-\epsilon_s+\mathbb{Z}\delta)\subseteq\Delta$ for all $t\neq s$. Hence, we have 
$\epsilon_t+\mathbb{Z}\delta\subseteq \Delta$ for all $t\in I_4$. From this it is easy to see that $\Delta=\Phi$.

\vskip 1mm
\noindent
Case (2). Suppose $Z_\a(\text{$\Psi_p(\tt E^{(2)}_{6})$})\subsetneq Z_\a(\Delta)$ for some $\a=\sum_{i=1}^{4}\lambda_i\epsilon_i$. 
Then there exists $r_1,r_2\in\mathbb{Z}$ with different parity such that $\frac{1}{2}\big(\sum_{i=1}^{4}\lambda_i\epsilon_i\big)+r_1\delta$, 
$\frac{1}{2}\big(\sum_{i=1}^{4}\lambda_i\epsilon_i\big)+r_2\delta\in\Delta$. So, we have
$$\lambda_1\epsilon_1+(r_k+s)\delta=\frac{1}{2}\big(\sum\limits_{i=1}^{4}\lambda_i\epsilon_i\big)+r_k\delta+\frac{1}{2}\big(\lambda_1\epsilon_1+
\sum\limits_{i=2}^{4}(-\lambda_i)\epsilon_i\big)+s\delta\in\Delta$$
for $k=1,2$ and $s=p_{\frac{1}{2}\big(\lambda_1\epsilon_1+
\sum\limits_{i=2}^{4}(-\lambda_i)\epsilon_i\big)}$. Since $(r_1+s)$ and $(r_2+s)$ have different parity, we are back to Case (1) and hence $\Delta=\Phi$.

\vskip 1mm
\noindent
Case (3). Suppose $\Gr(\text{$\Psi_p(\tt E^{(2)}_{6})$})\subsetneq \Gr(\Delta)$. 
Then there exists $i,j\in I_4, i\neq j,$ such that  $p_{\epsilon_i}$ and $p_{\epsilon_j}$ have different parity and
$\epsilon_{i}\pm\epsilon_j+2r\delta\in\Delta$ for some $r\in \mathbb{Z}$. Since $\mp\epsilon_j+p_{\mp\epsilon_j}\delta\in\text{$\Psi_p(\tt E^{(2)}_{6})$}$, we get
$\epsilon_i+(p_{\mp\epsilon_j}+2r)\delta\in\Delta.$ Since $p_{\epsilon_{i}}$, $p_{\mp\epsilon_{j}}+2r$ have different parity and $\epsilon_i+p_{\epsilon_i}\delta\in\Delta$,
 we are back to the Case (1) again and hence $\Delta=\Phi$.
 This completes the proof.

\end{pf}

\begin{prop}
 Suppose $\Phi$ is of type $\tt  E^{(2)}_{6}$. Then $\Psi\le \Phi$ is a maximal closed subroot system with a
 proper semi-closed gradient subroot system  if and only if there exists a 
 $\mathbb{Z}-$linear function $p: \mathcal{D}_{4}\to \mathbb{Z}$ such that
$\Psi=\text{$\Psi_p(\tt E^{(2)}_{6})$}$ and exactly two of $p_{\epsilon_i}$ are even. The type of $\Psi_p(\tt E^{(2)}_{6})$ is $\tt C_4^{(1)}$.
 \end{prop}
\begin{proof}
Since  $\Psi$ is a maximal closed subroot system in $\Phi$ and not contained in the proper closed subroot system $\Psi_0$ of $\Phi$, where
$\Psi_0:=\{\pm\epsilon_i+r\delta, \pm\epsilon_i\pm\epsilon_j+2r\delta, : 1\leq i\neq j\leq 4, r \in \mathbb{Z}\},$ 
there is  a short root of the form $\frac{1}{2}(\sum_{j=1}^4\nu_j\epsilon_j)$ in $\Gr(\Psi)$, fix this short root in $\Gr(\Psi)$. 
Now,  define $I:=\{i\in I_4: \epsilon_i\in \Gr(\Psi)\}$.

\vskip 1.5mm
First, we prove that $I$ must be non-empty subset of $I_4$.
Assume that $I=\emptyset$. Since $\Gr(\Psi)$ is semi-closed, there exist short roots $\a_1$ and $\a_2$ such that $\a_1+\a_2$ is a long root and 
$\a_1+\a_2\in\mathring{\Phi}\backslash \Gr(\Psi)$. Since $I=\emptyset$, we can take $\a_1=\frac{1}{2}(\lambda_{i_1}\epsilon_{i_1}+\lambda_{i_2}\epsilon_{i_2}+\lambda_{i_3}\epsilon_{i_3}+\lambda_{i_4}\epsilon_{i_4})$ and
$\a_2=\frac{1}{2}(-\lambda_{i_1}\epsilon_{i_1}-\lambda_{i_2}\epsilon_{i_2}+\lambda_{i_3}\epsilon_{i_3}+\lambda_{i_4}\epsilon_{i_4})$. 
Since $\epsilon_i\notin\Gr(\Psi)$ for all $i\in I_4$ and $\Psi$ is a closed subroot system, the only short roots that $\Gr(\Psi)$ can contain are $\a_1, \a_2$,
$\a_3=\frac{1}{2}(-\lambda_{i_1}\epsilon_{i_1}+\lambda_{i_2}\epsilon_{i_2}-\lambda_{i_3}\epsilon_{i_3}+\lambda_{i_4}\epsilon_{i_4})$ and 
$\a_4=\frac{1}{2}(-\lambda_{i_1}\epsilon_{i_1}+\lambda_{i_2}\epsilon_{i_2}+\lambda_{i_3}\epsilon_{i_3}-\lambda_{i_4}\epsilon_{i_4})$ along with their negatives.
For example,  if $\beta=\frac{1}{2}(-\lambda_{i_1}\epsilon_{i_1}+\lambda_{i_2}\epsilon_{i_2}+\lambda_{i_3}\epsilon_{i_3}+\lambda_{i_4}\epsilon_{i_4})\in\Gr(\Psi)$, then
$\a_1+(-\beta)=\lambda_{i_1}\epsilon_{i_1}\in\Gr(\Psi)$ since $\a_1+(-\beta)$ is a short root and $\Gr(\Psi)$ is semi-closed. This is clearly a contradiction to our assumption that $I=\emptyset$.
So, $\Gr(\Psi)\subseteq \Delta:=\{\pm\a_i:i\in I_4\} \cup \mathring\Phi_\ell.$
But $\Delta$ is a closed subroot system of $\mathring{\Phi}$ and hence $\wh\Delta$ is a closed subroot system in $\Phi$. 
Since $\Psi\subseteq \wh\Delta$, we must have $\Psi=\widehat{\Delta}$ and $\Gr(\Psi)=\Delta$, 
a contradiction to the fact that $\Gr(\Psi)$ is a proper semi-closed subroot system of $\mathring{\Phi}$. This proves that $I$ must be non-empty.
Indeed we will prove that $|I|$ must be $4$, hence $I=I_4$. We will rule out all other possibilities one by one.

\vskip 1.5mm
\noindent
Case (1).
We claim that we must have $|I|\ge 2$, hence $|I|\neq 1$. Let $i\in I.$
As before, since $\Gr(\Psi)$ is semi-closed there exist short roots $\a$ and $\beta$ such that $\a+\beta$ is a long root and 
$\a+\beta\in\mathring{\Phi}\backslash \Gr(\Psi)$. Now,  both these short roots must 
lie in $\big\{\frac{1}{2}\sum_{j=1}^{4}\lambda_j\epsilon_j: \lambda_j=\pm 1 \big\}$, otherwise we are done. 
So, without loss of generality we assume that $\a=
\frac{1}{2}(\lambda_{i_1}\epsilon_{i_1}+\lambda_{i_2}\epsilon_{i_2}+\lambda_{i_3}\epsilon_{i_3}+\lambda_{i_4}\epsilon_{i_4})$ and
$\beta=\frac{1}{2}(-\lambda_{i_1}\epsilon_{i_1}-\lambda_{i_2}\epsilon_{i_2}+\lambda_{i_3}\epsilon_{i_3}+\lambda_{i_4}\epsilon_{i_4})$.
If $i_3=i$, then 
$\bold{s}_{\epsilon_i}(\a)=\a-\lambda_{i_3}\epsilon_i=\frac{1}{2}(\lambda_{i_1}\epsilon_{i_1}+\lambda_{i_2}\epsilon_{i_2}-\lambda_{i_3}\epsilon_{i_3}+\lambda_{i_4}\epsilon_{i_4})$.
Since $\Gr(\Psi)$ is semi-closed, we have $\beta+\bold{s}_{\epsilon_i}(\a)=\lambda_{i_4}\epsilon_{i_4}\in \Gr(\Psi)$. 
Similarly,  if $i_4=i$, then we get $\lambda_{i_3}\epsilon_{i_3}\in \Gr(\Psi)$. 
Now,  if $i_1=i$, then $\bold{s}_{\epsilon_i}(\a)+(-\beta)=\lambda_{i_2}\epsilon_{i_2}\in \Gr(\Psi)$. Similarly,  
if $i_2=i$, then we get $\lambda_{i_1}\epsilon_{i_1}\in \Gr(\Psi)$. This proves that we must have $|I|\ge 2$ in all cases.

\vskip 1.5mm
\noindent
Case (2). Now,  we claim that $|I| \neq 3$. Suppose that $|I|=3$ and $I=I_4\backslash \{k\}$ for some $k\in I_4$.
Recall that we have a short root of the form $\frac{1}{2}(\sum_{j=1}^4\nu_j\epsilon_j)$ in $\Gr(\Psi)$.
For $j\in I_4$ such that $j\neq k$, there exists $r_j\in\mathbb{Z}$ such that $\nu_j\epsilon_j+r_j\delta\in \Psi$ since $\nu_j\epsilon_j\in \Gr(\Psi)$.
Since $|I|=3$, there exists $j_1, j_2\in I$ such that $r_{j_1}+r_{j_2}\in 2\mathbb{Z}$. This implies that 
${\sum_{\ell=1}^2(\nu_{j_\ell}\epsilon_{j_\ell}+r_{j_\ell}\delta)}\in\Psi$ since $\Psi$ is closed in $\Phi$.
Now,  since 
$\frac{1}{2}(\sum_{j=1}^4\nu_j\epsilon_j)+r\delta\in\Psi$ for some $r\in\mathbb{Z}$ and $\Psi$ is closed, we have 
$\frac{1}{2}(-\nu_{j_1}\epsilon_{j_1}-\nu_{j_2}\epsilon_{j_2}+\nu_{j_3}\epsilon_{j_3}+\nu_k\epsilon_k)+\left(r-{\sum\limits_{\ell=1}^{2}r_{j_\ell}}\right)\delta= 
\frac{1}{2}\left(\sum\limits_{j=1}^4\nu_j\epsilon_j\right)+r\delta-{\sum\limits_{\ell=1}^2(\nu_{j_\ell}\epsilon_{j_\ell}+r_{j_\ell}\delta)}\in\Psi.$
Adding  
$\frac{1}{2}(-\nu_{j_1}\epsilon_{j_1}-\nu_{j_2}\epsilon_{j_2}+\nu_{j_3}\epsilon_{j_3}+\nu_k\epsilon_k)+(r-{\sum_{\ell=1}^{2}r_{j_\ell}})\delta$ and
$\frac{1}{2}(\sum_{j=1}^4\nu_j\epsilon_j)+r\delta\in\Psi$ we get 
$(\nu_{j_3}\epsilon_{j_3}+\nu_k\epsilon_k)+(2r-{\sum\limits_{\ell=1}^2 r_{j_\ell}})\delta\in\Psi$.
Again adding $-\nu_{j_3}\epsilon_{j_3}-r_{j_3}\delta$ with $(\nu_{j_3}\epsilon_{j_3}+\nu_k\epsilon_k)+(2r-{\sum\limits_{\ell=1}^2 r_{j_\ell}})\delta\in\Psi$,
we get 
$\nu_k\epsilon_k+(2r-{\sum\limits_{k=1}^3 r_{j_k}})\delta\in \Psi$
which contradicts the assumption that $k\notin I$.
This proves that $|I| \neq 3$. So, we proved that $|I|=2$ or $4$ are the only possibilities.

\vskip 1.5mm
\noindent
Case (3). Now,  assume that $|I|=4$, hence $I=I_4$. In this case, we claim that there exists a $\mathbb{Z}-$linear function $p: \mathcal{D}_{4}\to \mathbb{Z}$ with the property that
exactly two $p_{\epsilon_i}$ are even and the rest two are odd such that
$ \Psi=\text{$\Psi_p(\tt E^{(2)}_{6})$}$. Since $\Gr(\Psi)$ contains $\pm\epsilon_i, 1\le i\le 4,$  and a short root of the form
$\frac{1}{2}\sum_{j=1}^{4}\nu_j\epsilon_j$, $\Gr(\Psi)$ must contain all the short roots of $F_4$. 
 We now claim that for each short root $\a\in \Gr(\Psi)$, $Z_{\a}(\Psi)$ contains either only odd integers or even integers, i.e., it can not contain integers with different parity.
We will do this case by case.
\begin{itemize}
 \item  Suppose there is 
$i\in I_4$ such that $2r_1, 2r_2+1\in Z_{\epsilon_i}(\Psi)$ for some $r_1, r_2\in \mathbb{Z}$. Using this, one easily sees that
 there exist $s^\pm_{ij}\in\mathbb{Z}$ such that  $\epsilon_i\pm\epsilon_j+2s^\pm_{ij}\delta\in\Psi$
for all $j\neq i$,   since $I=I_4$ and $\Psi$ is closed.  Hence, 
 $$\pm\epsilon_j\mp\epsilon_k+({2s^\pm_{ij}- 2s^\pm_{ik})\delta}=\epsilon_i\pm\epsilon_j+2s^\pm_{ij}\delta-(\epsilon_i\pm\epsilon_k+2s^\pm_{ik}\delta)\in\Psi$$
  for all $j\neq k$ contradicting our assumption on $\Gr(\Psi)$ that it is semi-closed. This proves that $Z_{\epsilon_i}(\Psi)$ contains either only odd integers or only even integers.
  
\item Now,  assume that $Z_\alpha(\Psi)$ contains both odd and even integers for some $\alpha=\frac{1}{2}(\sum_{j=1}^4\mu_j\epsilon_j)$, i.e. $\exists \ r_1,r_2\in\mathbb{Z}$ with different parity such that
$\frac{1}{2}(\sum_{j=1}^4\mu_j\epsilon_j)+r_1\delta, \frac{1}{2}(\sum_{j=1}^4\mu_j\epsilon_j)+r_2\delta\in\Psi.$
Then this implies that  
$\frac{1}{2}(\mu_1\epsilon_1-\mu_2\epsilon_2+\mu_3\epsilon_3+\mu_4\epsilon_4)+(r_1-k_2)\delta=
\frac{1}{2}(\sum\limits_{j=1}^4\mu_j\epsilon_j)+r_1\delta-(\mu_2\epsilon_2+k_2\delta)\in\Psi,$ where $k_2\in Z_{\epsilon_2}(\Psi).$
Similarly,  we get $\frac{1}{2}(\mu_1\epsilon_1-\mu_2\epsilon_2-\mu_3\epsilon_3-\mu_4\epsilon_4)+(r_1-k_2-k_3-k_4)\delta\in \Psi$, where $k_j\in Z_{\epsilon_j}(\Psi).$
Which in turn implies that
$$(\a+r_1\delta)+\left(\tfrac{1}{2}(\mu_1\epsilon_1-\sum_{j=2}^{4}\mu_j\epsilon_j)+(r_1-\sum_{j=2}^4k_j)\delta\right)=\mu_1\epsilon_1+(2r_1-\sum_{j=2}^4k_j)\delta\in \Psi,$$ since $\Psi$ is closed in $\Phi.$ Similarly,  we have $\mu_1\epsilon_1+(r_1+r_2-\sum\limits_{j=2}^4k_j)\delta\in\Psi.$ 
Which means $Z_{\epsilon_1}(\Psi)$ contains integers of different parity which by Case(1) is impossible. This proves our claim. 
\end{itemize}
Let $p$ a function 
$p: \Gamma_4 \to \mathbb{Z}$ such that $p_{\beta}\in Z_{\beta}(\Psi)$ for each $\beta$ in $\Gamma_4$, where
 $\Gamma_4$ is a simple root system of $\mathcal{D}_4$ defined in \ref{2E6}. 
Extend the function $p$ to $\mathcal{D}_4$ $\mathbb{Z}-$linearly, denote this extension again by $p$. We now claim that exactly two $p_{\epsilon_i}$ are even. 
Suppose all $p_{\epsilon_i}$ have the same parity, then \say{$\Psi$ is closed in $\Phi$} would imply that $\Gr(\Psi)=\mathring\Phi$. 
This is a contradiction to our assumption that $\Gr(\Psi)$ is semi-closed.
So, all $p_{\epsilon_i}$ can not have the same parity.
Now,  assume that there exists $k\in I_4$ such that $p_{\epsilon_i}$ have the same parity for all $i\neq k$, and $p_{\epsilon_k}$ has different parity. 
Let $\beta_1=\frac{1}{2}(\sum_{i\neq k}\epsilon_i+\epsilon_k)+r\delta\in\Psi$ for some $r\in\mathbb{Z}$. 
Since $\Psi$ is closed, we have $$\beta_2=\frac{1}{2}(\sum\limits_{i\neq k}(-\epsilon_i)+\epsilon_k)+(r-\sum\limits_{i\neq k}p_{\epsilon_i})\delta\in\Psi$$
and hence we get $\beta_1+\beta_2=\epsilon_k+(2r-\sum\limits_{i\neq k}p_{\epsilon_i})\delta\in\Psi$. 
This implies that  $p_{\epsilon_k}$ and $(2r-\sum_{i\neq k}p_{\epsilon_i})$ are in $Z_{\epsilon_k}(\Psi)$. But $p_{\epsilon_k}$ and $(2r-\sum_{i\neq k}p_{\epsilon_i})$
have different parity, which is a contradiction to our previous observation that $Z_\a(\Psi)$ contains only either odd integers or even integers. 
Thus,  we proved that exactly two $p_{\epsilon_i}$ are even and the rest are odd.
Now,  using the arguments in the proof of Case (3) in Lemma \ref{2E6maxclosed}, we see that there is no $i, j\in I_4$ with $i\neq j$ such that 
$p_{\epsilon_i}$ and $p_{\epsilon_j}$ have different parity and $\pm\epsilon_i\pm\epsilon_j\in \Gr(\Psi)$.
This implies that
$\Psi\subseteq\text{$\Psi_p(\tt E^{(2)}_{6})$}$. Since $\Psi$ is maximal closed, we have $\Psi=\text{$\Psi_p(\tt E^{(2)}_{6})$}$.
 
\vskip 1.5mm
\noindent
Case (4).
Finally assume that $|I|=2$ and $I=\{i,j\}$. Since $\Gr(\Psi)$ is semi-closed, then we claim that we have
 $$\Gr(\Psi)\cap \mathcal{D}_4=\big\{\pm\epsilon_i,\pm\epsilon_j,\pm\tfrac{1}{2}(\sum_{r=1}^4\mu_r\epsilon_r):\mu_r=\nu_r, r\neq i,j\big\}.$$
Since $\a=\tfrac{1}{2}(\sum_{r=1}^{4}\nu_r\epsilon_r)\in \Gr(\Psi)$, 
we have $\bold s_{\epsilon_i}(\a)=\a-\nu_i\epsilon_i\in \Gr(\Psi)$ and $\bold s_{\epsilon_j}(\a)=\a-\nu_i\epsilon_j\in \Gr(\Psi)$. This proves that
 $\Gr(\Psi)\cap \mathcal{D}_4\supseteq \big\{\pm\epsilon_i,\pm\epsilon_j,\pm\frac{1}{2}(\sum_{r=1}^4\mu_r\epsilon_r):\mu_r=\nu_r, r\neq i,j\big\}.$
Suppose $\beta=\frac{1}{2}(\sum_{r=1}^4\mu_r\epsilon_r)\in \Gr(\Psi)$ such that $\mu_k\neq \nu_k$ for some $k\neq i,j$. Let $\ell\in I_4\backslash \{i, j, k\}$.
If $\mu_\ell\neq \nu_\ell$, then $-\beta$ satisfies the required condition, i.e., $-\mu_k=\nu_k$ and $-\mu_\ell= \nu_\ell$. So, assume that
$\mu_\ell=\nu_\ell$, then
$\frac{1}{2}(-\mu_i\epsilon_i-\mu_j\epsilon_j-\mu_k\epsilon_k+\mu_\ell\epsilon_\ell)\in \Gr(\Psi)$ and $\Psi$ is closed, so we have $\epsilon_\ell\in \Gr(\Psi)$.
This is clearly a contradiction to our assumption that $I=\{i, j\}$. This proves that 
 $$\Gr(\Psi)\cap \mathcal{D}_4=\big\{\pm\epsilon_i,\pm\epsilon_j,\pm\tfrac{1}{2}(\sum_{r=1}^4\mu_r\epsilon_r):\mu_r=\nu_r, r\neq i,j\big\}.$$
From this one easily sees that the only long roots $\Gr(\Psi)$ can contain are  $\pm\epsilon_i\pm \epsilon_j$ and $\pm\epsilon_k\pm \epsilon_\ell$,  where $\{k, \ell\}=I_4\backslash \{i, j\}$. 
 Note that $\pm\epsilon_k$ and $\pm\epsilon_\ell$ can not be written as sum of elements from $\Gr(\Psi)\cap \mathcal{D}_4$. 
 We now claim that $Z_{\alpha}(\Psi)$ does not contain elements of different parity for each short root $\alpha$ in $\Gr(\Psi)$. Assume this claim for time being.
 Then for each $\a\in \Gr(\Psi)\cap \mathcal{D}_4$, we have $Z_{\alpha}(\Psi)\subseteq p_\a+2\mathbb{Z}$ for some $p_\a\in Z_{\alpha}(\Psi)$.
 Note that $p_\a$ is determined by $\Psi$ for $\a\in  \Gr(\Psi)\cap \mathcal{D}_4$.
 Now we extend this function $p:\Gr(\Psi)\cap \mathcal{D}_4\rightarrow\mathbb{Z}$ to entire $\mathcal{D}_4$, $\mathbb{Z}$--linearly by defining $p_{\epsilon_k}$ in the following way:
 \begin{itemize}
 \item If both $p_{\epsilon_i}$ and $p_{\epsilon_j}$ have the same parity, then define $p_{\epsilon_k}$ to be an integer with different parity than $p_{\epsilon_i}$.
 \item If $p_{\epsilon_i}$ and $p_{\epsilon_j}$ have different parity, then define $p_{\epsilon_k}$ arbitrarily.
 \end{itemize}
 The extended function $p:\mathcal{D}_4\to \mathbb{Z}$, then satisfies
the conditions that (1) exactly two $p_{\epsilon_r}$ are even and the rest two $p_{\epsilon_r}$ are odd and (2) it takes the same values $p_\a$ which was determined by $\Psi$ for 
$\a\in \Gr(\Psi)\cap \mathcal{D}_4.$
Note that the parity of $p_{\epsilon_\ell}$ is completely determined by the parity of $p_{\epsilon_i}$, $p_{\epsilon_j}$, $p_{\epsilon_k}$ and $p_{\frac{1}{2}\sum_{r=1}^4\nu_r\epsilon_r)}$.
By the choice of $p$, we have $\Psi\subsetneq\text{$\Psi_p(\tt E^{(2)}_{6})$}$. This proves that $\Psi$ can not be maximal closed subroot system in $\Phi$. Hence, the 
case $|I|=2$ is not possible.

\vskip 1.5mm
{\em{Proof of the claim:}}
Now,  we will complete the proof of the claim that $Z_{\alpha}(\Psi)$ does not contain elements of different parity for each short root $\alpha$ in $\Gr(\Psi)$.
 Let $\a_1$, $\a_2$ be two short roots 
 in $\Gr(\Psi)$ such that $\a_1+\a_2 $ is a long root and $\a_1+\a_2\in\mathring{\Phi}\backslash\Gr(\Psi)$.  
 We now prove  that if $Z_{\beta}(\Psi)$ contains elements of different parity for some short root $\beta$ in $\Gr(\Psi)$, 
 then $Z_{\a}(\Psi)$ must contain elements of different parity for all short roots $\a$ in $\Gr(\Psi)$. 
 This will contradict the fact that $\a_1+\a_2\in\mathring{\Phi}\backslash\Gr(\Psi)$, hence the claim follows. 
 \begin{itemize}
  \item Assume that
 $Z_{\epsilon_i}(\Psi)$ contains elements of different parity,
then we have $\pm\epsilon_i\pm\epsilon_j\in\Gr(\Psi)$ as $\Psi$ is closed. This implies that  $Z_{\epsilon_j}(\Psi)$ also contains elements of different parity. 
Let $\a=\frac{1}{2}(\sum\limits_{r=1}^4\mu_r\epsilon_r)\in \Gr(\Psi)$. We have $\frac{1}{2}(\sum_{r\neq s}\mu_r\epsilon_r-\mu_s\epsilon_s)\in\Gr(\Psi)$ for $s=i, j$.
Since for $s=i,j$, $$\frac{1}{2}(\sum_{r\neq s}\mu_r\epsilon_r-\mu_s\epsilon_s)+r_1\delta+\mu_s\epsilon_s+r_2\delta=\a+(r_1+r_2)\delta$$ and $Z_{\epsilon_s}(\Psi)$ 
contains elements of different parity, we have $Z_{\a}(\Psi)$ also contains elements of different parity.

\item Now,  assume that
$Z_{\a}(\Psi)$ contains elements of different parity for $\a=\frac{1}{2}(\sum_{r=1}^4\mu_r\epsilon_r)$ with $\mu_r=\nu_r, r\neq i,j.$
Since we have $\frac{1}{2}(\sum\limits_{r\neq i}\mu_r\epsilon_r-\mu_i\epsilon_i)\in \Gr(\Psi)$,
we get $Z_{\epsilon_i}(\Psi)$ contains elements of different parity. So, we are back to previous case.
\end{itemize}

 This completes the proof.
 \end{proof}


\section{The case $\tt A^{(2)}_{2n}$}\label{2A2n}
Throughout this section we assume that $\Phi$ is of type $\tt A^{(2)}_{2n}$ and $n\ge 2$. 
In particular, the gradient root system $\Gr(\Phi)$ of $\tt A^{(2)}_{2n}$ is of type $\tt BC_n$.
We have the following explicit description of $\tt A^{(2)}_{2n}$, see \cite[Page no. 547, 583]{carter}:
$$\Phi=\big\{\pm\epsilon_i+(r+\tfrac{1}{2})\delta, \pm2\epsilon_i+2r\delta, \pm\epsilon_i\pm\epsilon_j+r\delta, : 1\leq i\neq j\leq n, r \in \mathbb{Z} \big\}$$ and 
$\mathrm{Gr}(\Phi)=\big\{\pm\epsilon_i, \pm2\epsilon_i, \pm\epsilon_i\pm\epsilon_j: 1\leq i\neq j\leq n\big\}=\mathring{\Phi}\cup \frac{1}{2}\mathring{\Phi}_\ell$. In particular, we have three root lengths in 
$\mathrm{Gr}(\Phi)$ and we denote the short, intermediate and long roots of $\mathrm{Gr}(\Phi)$ by  $\mathrm{Gr}(\Phi)_s$,  $\mathrm{Gr}(\Phi)_{\mathrm{im}}$ and  $\mathrm{Gr}(\Phi)_\ell$ respectively.
Let $\Gamma=\{\a_1=\epsilon_1-\epsilon_2, \cdots, \a_{n-1}=\epsilon_{n-1}-\epsilon_n, \a_n=\epsilon_n\}$ be the simple system for $\Gr(\Phi)$. 

\subsection{}
Before we proceed further we fix some notations. 
 For $I\subseteq I_n$, we set
 \begin{align*}
 &\Psi^+(I, \tfrac{1}{2}):=\big\{\epsilon_i+(2r+\tfrac{1}{2})\delta, (\epsilon_k+\epsilon_\ell)+(2r+1)\delta, (\epsilon_k-\epsilon_\ell)+2r\delta:   i,k,\ell\in I, k\neq \ell, \ r \in \mathbb{Z}  \big\},\\ 
 &\Psi^+(I, \tfrac{3}{2}):=\big\{\epsilon_i+(2r+\tfrac{3}{2})\delta, (\epsilon_k+\epsilon_\ell)+(2r+1)\delta, (\epsilon_k-\epsilon_\ell)+2r\delta:   i,k,\ell\in I, k\neq \ell, \ r \in \mathbb{Z}  \big\}\ \text{and} \\ 
 &\Psi^+(I, 0, 1):=\big\{(\epsilon_k+\epsilon_\ell)+2r\delta, (\epsilon_k-\epsilon_\ell)+(2r+1)\delta: k\in I, \ell\in I_n\backslash I, r \in \mathbb{Z}\big\}.
 \end{align*}
 Now,  define
 
$\text{$\Psi_I(\tt A^{(2)}_{2n})$}:=\Psi^+\big(I, \frac{1}{2}\big)\cup (-\Psi^+\big(I, \frac{1}{2}\big))\cup \Psi^+(I, 0, 1)\cup (-\Psi^+(I, 0, 1))
\cup \Psi^+\big(I_n\backslash I, \frac{3}{2}\big)\cup (-\Psi^+\big(I_n\backslash I, \frac{3}{2}\big)).
$ Note that $\Gr(\text{$\Psi_I(\tt A^{(2)}_{2n})$})=\{\pm\epsilon_i, \pm\epsilon_k\pm\epsilon_\ell: i, k, \ell\in I_n, k\neq \ell \}$ for a root system of type $B_n.$

\begin{prop} For $I\subseteq I_n$, 
$\text{$\Psi_I(\tt A^{(2)}_{2n})$}$ is a maximal closed subroot system of $\Phi.$
\end{prop}
\begin{proof}
It is easy to check that $\text{$\Psi_I(\tt A^{(2)}_{2n})$}$ is a closed subroot system of $\Phi.$ We prove that it is a maximal closed subroot system in $\Phi$.
Let $\Delta$ be a closed subroot system of $\Phi$ such that
$\text{$\Psi_I(\tt A^{(2)}_{2n})$}\subsetneq\Delta\subseteq\Phi$.
The following are the possibilities for elements of $\Delta\backslash\text{$\Psi_I(\tt A^{(2)}_{2n})$}$: if $\a\in \Delta\backslash\text{$\Psi_I(\tt A^{(2)}_{2n})$}$, then $\a$ must be equal to either
\begin{itemize}
 \item $\epsilon_i+(2r+\frac{3}{2})\delta \in \Delta$ or $2\epsilon_i+2r \delta$, where $i\in I, \ r\in\mathbb{Z}$ 
 \item $\epsilon_i+(2r+\frac{1}{2})\delta \in \Delta$ or $2\epsilon_i+2r \delta\in \Delta$, where $i \notin I, \ r\in\mathbb{Z}$ 
 \item $(\epsilon_k+\epsilon_\ell)+2r\delta\in \Delta$ or $(\epsilon_k-\epsilon_\ell)+(2r+1)\delta\in \Delta$, where $k,\ell\in I$ and $r\in \mathbb{Z}$
 \item $(\epsilon_k+\epsilon_\ell)+2r\delta\in \Delta$ or $(\epsilon_k-\epsilon_\ell)+(2r+1)\delta\in \Delta$, where $k,\ell\notin I$ and $r\in \mathbb{Z}$
 \item $(\epsilon_k+\epsilon_\ell)+(2r+1)\delta\in \Delta$ or $(\epsilon_k-\epsilon_\ell)+2r\delta\in \Delta$, where $k\in I$, $\ell\notin I$ and $r\in \mathbb{Z}$ 
\end{itemize}

Suppose 
there exists $i \in I$ such that $\epsilon_i+(2r+\frac{3}{2})\delta \in \Delta$ for some $r \in \mathbb{Z}$. Then
since  $\epsilon_i+(2\mathbb{Z}+\frac{1}{2})\delta \subseteq \Delta$, we have 
$(\epsilon_i+(2\mathbb{Z}+\frac{1}{2})\delta) +(\epsilon_i+(2r+\frac{3}{2})\delta=2\epsilon_i+2\mathbb{Z} \delta\subseteq \Delta.$ 
This implies that  $(2\epsilon_i+2\mathbb{Z} \delta)-(\epsilon_i+(2r+\frac{3}{2})\delta)=\epsilon_i+(\mathbb{Z}+\tfrac{1}{2})\delta\subseteq\Delta$. For $j \in I$,
$\bold{s}_{\epsilon_i-\epsilon_j}(2\epsilon_i+2\mathbb{Z}\delta)=2\epsilon_j+2\mathbb{Z}\delta\subseteq\Delta$. 
Similarly,  for $j\notin I$ we have
$\bold{s}_{\epsilon_i+\epsilon_j}(2\epsilon_i+2\mathbb{Z}\delta)=-2\epsilon_j+2\mathbb{Z}\delta\subseteq\Delta$.
As before this implies that  $\epsilon_j+\frac{2\mathbb{Z}+1}{2}\delta\in\Delta$ for all $j \in I_n$. Hence, $\Delta=\Phi$. 
Suppose there exists $i \in I$ such that $2\epsilon_i+2r\delta \in \Delta$ for some $r \in \mathbb{Z}$. Then 
$\epsilon_i+\frac{3}{2}\delta =(2\epsilon_i+2r\delta)+(-\epsilon_i-(2(r-1)+\frac{1}{2})\delta)\in \Delta$, so we are back to the first case. Hence, $\Delta=\Phi.$

All the remaining cases are done similarly. For example,  if $(\epsilon_k+\epsilon_\ell)+2r\delta\in \Delta$ for some $r\in \mathbb{Z}$ and $k, \ell\in I$, then we have
$\epsilon_k+\frac{3}{2}\delta =(\epsilon_k+\epsilon_\ell)+2r\delta+(-\epsilon_\ell-(2(r-1)+\frac{1}{2})\delta)\in \Delta$, so we are back to first case.
This completes the proof.

\end{proof}

\subsection{}
We now see another possible maximal closed subroot system of $\Phi$. 
For $J\subsetneq I_n$, define 
$$A_J:=\big\{\pm2\epsilon_i,  \pm\epsilon_s\pm\epsilon_t: i\in I_n\backslash J, \ s\neq t\in I_n\backslash J \big\}\cup 
\big\{\pm2\epsilon_j, \pm\epsilon_j, \pm\epsilon_k\pm\epsilon_\ell: j\in J,\ k\neq \ell\in J \big\}$$ and denote by
$\wh{A_J}$ the lift of $A_J$ in $\Phi.$ 
Here we make the convention that 
$$A_J=\begin{cases}
       \big\{\pm2\epsilon_i, \pm\epsilon_s\pm\epsilon_t: i\in I_n,\ s\neq t\in I_n \big\}\ \text{if}\ J=\emptyset \\
       \big\{\pm2\epsilon_i,\pm\epsilon_j : i\in I_n, j\in J\big\} \ \text{if $|J|=1$ and $n=2$}\\
        \big\{\pm2\epsilon_i,\pm\epsilon_j, \pm\epsilon_s\pm\epsilon_t: i\in I_n,\ s\neq t\in I_n\backslash J \big\} \ \text{if  $J=\{j\}$ and $n>2$}\\
        \big\{\pm2\epsilon_i,\pm\epsilon_j, \pm\epsilon_k\pm\epsilon_\ell: i\in I_n,\ j\in J, \ k\neq \ell\in J \big\} \ \text{if  $|I_n\backslash J|=1$}\\

      \end{cases}
$$
Note that $A_J$ is a proper closed subroot system of $BC_n$ for any $J\subsetneq I_n$ and it is of 
type $\text{$C_{n-r}\oplus BC_{r}$}$ if $|J|=r$. Hence, the lift $\wh{A_J}$ of $A_J$ is a closed subroot system in $\Phi$. 
We have, \begin{prop}
The lift $\wh{A_J}$ of $A_J$ in $\Phi$ is a maximal closed subroot system $\Phi$ for $J\subsetneq I_n$.
\end{prop}
\begin{proof}
Let $\Delta$ be a closed subroot system of $\Phi$ such that 
$\wh{A_J}\subsetneq\Delta\subseteq\Phi$. 
Then there are three possibilities for elements of $\Delta\backslash\wh{A_J}$.
\vskip 1.5mm
\noindent
Case(1). Suppose $\epsilon_i+(r+\tfrac{1}{2})\delta\in\Delta$ for some $i\notin J$ and $r\in\mathbb{Z}$. Then 
since $\pm\epsilon_i\pm\epsilon_s+\mathbb{Z}\delta\subseteq\Delta$ for all $s\notin J$ with $i\neq s$, 
we have $$ \text{$(\epsilon_i+(r+\tfrac{1}{2})\delta)+(-\epsilon_i\pm\epsilon_s+\mathbb{Z}\delta)=\pm\epsilon_s+(\mathbb{Z}+\tfrac{1}{2})\delta\subseteq\Delta$
for all $s\notin J$, $i\neq s$.}$$
If $J=\emptyset$, then we get  $\pm\epsilon_i+(\mathbb{Z}+\tfrac{1}{2})\delta\subseteq\Delta$ by repeating the earlier argument with the choice of $s\in I_n$ which is different from $i$.
If $J\neq \emptyset$, then $\epsilon_j+(\mathbb{Z}+\tfrac{1}{2})\delta\subseteq \Delta$ for all $j\in J$. Fix $j\in J$.
Then we have $(\epsilon_i+(r+\tfrac{1}{2})\delta)+(\epsilon_j+(\mathbb{Z}+\tfrac{1}{2})\delta)=\epsilon_i+ \epsilon_j+\mathbb{Z}\delta\subseteq \Delta$.
Now,  $$\epsilon_i+(\mathbb{Z}+\tfrac{1}{2})\delta=(-\epsilon_j+(\mathbb{Z}+\tfrac{1}{2})\delta)+(\epsilon_i+ \epsilon_j+\mathbb{Z}\delta)\subseteq \Delta.$$
This proves that $\epsilon_s+(\mathbb{Z}+\tfrac{1}{2})\delta\subseteq\Delta$
for all $s\notin J$. Hence, we have $\pm\epsilon_s+(\mathbb{Z}+\tfrac{1}{2})\delta\subseteq\Delta$
for all $s\in I_n$. This implies that $\Delta=\Phi$. 
\vskip 1.5mm
\noindent
Case(2). Suppose $\epsilon_j+\epsilon_k+r\delta\in\Delta$ for some $j\in J, k\notin J$ and $r\in\mathbb{Z}$. Then
since $-\epsilon_j+(\mathbb{Z}+\tfrac{1}{2})\delta\subseteq\Delta$, we have $$\epsilon_k+(\mathbb{Z}+\tfrac{1}{2})\delta=(\epsilon_j+\epsilon_k+r\delta)+(-\epsilon_j+(\mathbb{Z}+\tfrac{1}{2})\delta) \subseteq\Delta.$$
So, we are back to the Case (1), hence  $\Delta=\Phi$. 
\vskip 1.5mm
\noindent
Case(3). Suppose $\epsilon_k-\epsilon_j+r\delta\in\Delta$ for some $j\in J, k\notin J$ and $r\in\mathbb{Z}$. Then
since $\epsilon_j+(\mathbb{Z}+\tfrac{1}{2})\delta\subseteq\Delta$, we have $$\epsilon_k+(\mathbb{Z}+\tfrac{1}{2})\delta=(\epsilon_k-\epsilon_j+r\delta)+(\epsilon_j+(\mathbb{Z}+\tfrac{1}{2})\delta) \subseteq\Delta.$$
So, we are back to the Case (1), hence  $\Delta=\Phi$. 
\end{proof}

\subsection{}
Let $\Psi\le \Phi$ be a maximal subroot system. Now,  we are ready to state our final classification theorem for the case $\tt A^{(2)}_{2n}$.
\begin{thm}\label{thm2A2n}
 Suppose $\Phi$ is of type $\tt A^{(2)}_{2n}$ and $\Psi\le \Phi$ is a maximal closed subroot system. Then 
 \begin{enumerate}
  \item[(i)]  $\Psi=$ the lift of $A_J$ for some $J\subsetneq I_n=\wh{A_J}$ or
  \item[(ii)] $\Psi=\text{$\Psi_I(\tt A^{(2)}_{2n})$}$ for some $I\subseteq I_n$ or
\item[(iii)] there exist an odd prime number $n_s$ and a $\mathbb{Z}$--linear function $p:\Gr(\Phi)_s\cup \Gr(\Phi)_{\mathrm{im}} \to \tfrac{1}{2}\mathbb{Z}$ satisfying the equation (\ref{palpha}) such that
 \end{enumerate} 
\begin{align*} \Psi(p, n_s):=&\big\{\pm\epsilon_i\pm(p_{\epsilon_i}+rn_s)\delta, 
\pm2\epsilon_i\pm(2p_{\epsilon_i}+n_s+2rn_s)\delta:i\in I_n,\ r\in \mathbb{Z}\big\}&\\& \cup \big\{\pm\epsilon_i\pm\epsilon_j+(\pm p_{\epsilon_i}\pm p_{\epsilon_j}
+rn_s)\delta : i, j\in I_n, i\neq j,\ r \in \mathbb{Z} \big\}.
\end{align*}
Conversely,   all the subroot systems defined above are maximal closed subroot systems of $\Phi.$ 
\end{thm}
\begin{pf} Define $J=\left\{i \in I_n: \epsilon_i \in \Gr(\Psi)\right\}$. Now,  two cases are possible: $J\subsetneq I_n$ or $J=I_n$.

\vskip 1.5mm
\noindent 
Case (1). First consider the case $J\subsetneq I_n$. In this case, we claim that $\Psi=\wh{A_J}$. 
This is immediate if we prove that $\Gr(\Psi) \subseteq A_J$. Suppose $\Gr(\Psi) \nsubseteq A_J$, then there must exist $k \in J$ and $\ell \notin J$ such that $\epsilon_k\pm\epsilon_\ell \in \Gr(\Psi)$. 
This means that there exists $r, r' \in \mathbb{Z}$ such that
$\epsilon_k\pm\epsilon_\ell +r\delta \in \Psi ,\epsilon_k + (r'+\tfrac{1}{2}) \delta \in \Psi.$
Since $\Psi$ is closed in $\Phi$, we get $$(\epsilon_k\pm\epsilon_\ell +r\delta)+(-\epsilon_k - (r'+\tfrac{1}{2}))\delta=\pm \epsilon_\ell +(r-r'- \tfrac{1}{2})\delta \in \Psi,$$ which contradicts the fact that $\ell \notin J$. 
So, $\Gr(\Psi) \subseteq A_J$ and hence $\Psi \subseteq \wh{A_J}$. Since $\wh{A_J}$ is closed in $\Phi$, we have $\Psi=\wh{A_J}$.

\vskip 1.5mm
\noindent 
Case (2). Now,  consider the case $J=I_n$. Since $\Psi$ is closed, we have $\pm \epsilon_i\pm \epsilon_j \in \Gr(\Psi)$ for all $1\leq i\neq j \leq n$.
It is easy to see that if $\Gr(\Psi)$ contains $2\epsilon_i$ for some $i \in I_n$, then
it contains $\pm2\epsilon_j$ for all $j \in I_n$ as $\bold{s}_{\epsilon_i-\epsilon_j}(2\epsilon_i)=2\epsilon_j$. So, we get either 
$\Gr(\Psi)=\left\{\pm\epsilon_i,\pm\epsilon_i\pm\epsilon_j: i, j\in I_n, i\neq j\right\}$
or $\Gr(\Psi)=\Gr(\Phi)$.

\vskip 1.5mm
\noindent 
Case (2.1). Suppose 
$\Gr(\Psi)=\left\{\pm\epsilon_i,\pm\epsilon_i\pm\epsilon_j: i, j\in I_n, i\neq j\right\}$,
then we claim that 
$\Psi=\text{$\Psi_I(\tt A^{(2)}_{2n})$}$ for some $I\subseteq I_n$.
By Proposition \ref{keypropositionnon2A2n}, we have
$$\text{$\exists\ \ k_i\in\mathbb{Z}$ such that $Z_{\epsilon_i}(\Psi)=(k_i+\tfrac{1}{2})+n_s\mathbb{Z}$, for each $i \in I_n$.}$$
Since $Z_{\epsilon_i}(\Psi)+Z_{\epsilon_i}(\Psi)=(2k_i+1)+n_s\mathbb{Z}$ and $2\epsilon_i\notin\Gr(\Psi)$, we must have $n_s\in 2\mathbb{Z}$. Set
$I=\left\{i\in I_n:k_i \in 2\mathbb{Z}\right\}$, then we immediately get
 $\Psi\subseteq \text{$\Psi_I(\tt A^{(2)}_{2n})$}$. Since $\text{$\Psi_I(\tt A^{(2)}_{2n})$}$ is closed, we have $\Psi=\text{$\Psi_I(\tt A^{(2)}_{2n})$}$.

\vskip 1.5mm
\noindent 
Case (2.2). Finally assume that $J=I_n$ and $\Gr(\Psi)=\Gr(\Phi).$ Then by Proposition \ref{keypropositionnon2A2n}, we have
$n_\a\in \mathbb{N}$ and $p_\a\in Z_\a(\Psi)$ such that $Z_\a(\Psi)=p_\a+n_\a\mathbb{Z}$ for all $\a\in \Gr(\Phi)$. By Proposition \ref{2A2n01}, we have 
$n_s=n_{\mathrm{im}},$ $n_\ell=2n_s$ and $n_s$ is an odd prime number. 
Conversely,   let $n_s$ be a given odd prime number and $p:\Gr(\Phi)_s\cup \Gr(\Phi)_{\mathrm{im}}\to \tfrac{1}{2}\mathbb{Z}$ be a given $\mathbb{Z}-$linear map satisfying the condition \ref{palpha}.
It is a straightforward checking that $\Psi(p, n_s)$ is a closed subroot system of $\Phi.$
Now,  we prove that $\Psi(p, n_s)$ must be a maximal closed subroot system in $\Phi$. 
Suppose there is a maximal subroot system $\Delta$ such that $\Psi(p, n_s)\subseteq \Delta \subsetneq \Phi$. Then since $\Gr(\Delta)=\Gr(\Phi)$
(by earlier arguments) $\Delta$ must be of the form $\Psi(p', n_s')$ for some function $p':\Gr(\Phi)_s\cup \Gr(\Phi)_{\mathrm{im}}\to \tfrac{1}{2}\mathbb{Z}$ and odd prime number $n_s'$.
Now,  $$Z_{\a}(\Psi)\subseteq  Z_{\a}(\Delta), \a\in \Gr(\Phi)$$ implies that $n_s=n_s'$
and $p_\a\equiv p'_\a(\mathrm{mod}\ n_s)$ for all $\a\in \Gr(\Phi)$. Hence, $\Psi(p, n_s)=\Delta$. This proves that $\Psi(p, n_s)$ is a maximal subroot system of $\Phi.$ 
This completes the proof.

 \end{pf}

\begin{rem}
 One can easily check that the type of $\widehat{A_J}$ is $\tt A_{2n-1}^{(2)}$ if $J=\emptyset$ else $\tt A_{2r}^{(2)} \oplus A_{2n-2r-1}^{(2)}$, where $|J|=r$, 
 the type of $\text{$\Psi_I(\tt A^{(2)}_{2n})$}$ is $\tt B_n^{(1)}$ and the type of $\Psi(p,n_s)$ is $\tt A_{2n}^{(2)}$.
Clearly,  the root systems of type $\tt D_r^{(1)}\oplus A_{2n-2r}^{(2)}$ do not occur as a maximal closed subroot 
system of $\tt A_{2n}^{(2)}$ as it is stated in \cite[Table 1 \& 2]{FRT}.
In \cite{FRT}, the authors do not give any description of the closed subroot systems 
of type $\tt D_r^{(1)}\oplus A_{2n-2r}^{(2)}$ of $\tt A_{2n}^{(2)}$. But we presume that it must be the lift $\wh{\Delta}$ of
$$\Delta=\big\{\pm\epsilon_k\pm\epsilon_\ell: 1\le k\neq \ell\le r\big\}\cup \big\{\pm\epsilon_i, \pm2\epsilon_i, \pm\epsilon_i\pm\epsilon_j : r+1\leq i\neq j\leq n\big\}.$$
It is easy to see that $\Delta$ is a closed subroot system of $\tt BC_n$ of type $\tt D_r\oplus BC_{n-r}$. 
Hence, $\wh\Delta$ is a closed subroot system of  $\tt A_{2n}^{(2)}$ of type $\tt D_r^{(1)}\oplus A_{2n-2r}^{(2)}$. But this is not maximal as
$\Delta\subsetneq \widehat{A_J}$ for $J=\{r+1,\cdots, n\}$.

 \end{rem}


\section{The case $\tt A^{(2)}_{2}$}\label{2A2}
Throughout this section we assume that $\Phi$ is of type $\tt A^{(2)}_{2}$. 
We have the following explicit description of $\tt A^{(2)}_{2}$, see \cite[Page no. 565]{carter}:
$$\Phi=\big\{\pm\epsilon_1+(r+\tfrac{1}{2})\delta, \pm 2\epsilon_1+2r\delta: r \in \mathbb{Z} \big\}$$
and $\Gr(\Phi)=\{\pm\epsilon_1, \pm2\epsilon_1\}.$
\subsection{}
We have the following classification theorem for 
the case $\tt A^{(2)}_{2}$.
\begin{thm}
 Suppose $\Phi$ is of type $\tt A^{(2)}_{2}$ and $\Psi$ is a maximal closed subroot system of $\Phi$. Then 
one of the following holds:
 \begin{enumerate}
  \item $\Psi=\Psi(k, q):=\big\{\pm \epsilon_1\pm(k+\tfrac{1}{2}+rq)\delta, \pm2\epsilon_1\pm({2k+1}+(2r+1)q)\delta: r \in \mathbb{Z} \big\}$
  for some 
  $k\in \mathbb{Z}_+$ and odd prime number $q$  and $\Gr(\Psi)=\{\pm\epsilon_1, \pm2\epsilon_1\}$. 
  \item $\Psi=\{\pm(\epsilon_1+(2r+\tfrac{1}{2})\delta): r\in \mathbb{Z}\} \ \text{or} \ \{\pm (\epsilon_1+(2r+\tfrac{3}{2})\delta):r\in \mathbb{Z}\}$  and $\Gr(\Psi)=\{\pm\epsilon_1\}$
  \item $\Psi=\{\pm(2\epsilon_1+2r\delta): r \in \mathbb{Z} \}$ and $\Gr(\Psi)=\{\pm2\epsilon_1\}$.
   
 \end{enumerate}
 If $\Psi=\Psi(k, q)$, then the type of $\Psi$ is $\tt A_{2}^{(2)}$, otherwise it is $\tt A_{1}^{(1)}$.
\end{thm}

\begin{pf}
 
Let $\Psi$ be a maximal closed subroot system. Then we have three possibilities for $\Gr(\Psi)$:
$\text{either $\Gr(\Psi)=\{\pm\epsilon_1\}$ or $\Gr(\Psi)=\{\pm2\epsilon_1\}$ or $\Gr(\Psi)=\{\pm\epsilon_1, \pm2\epsilon_1\}$.}$

\vskip 1.5mm
\noindent
Case (1).
First let $\Gr(\Psi)=\{\pm\epsilon_1, \pm2\epsilon_1\}$. 
Then by Proposition \ref{keypropositionnon2A2n}, we have 
$$\text{ $Z_{\pm\epsilon_1}(\Psi)=\pm p_{\epsilon_1}+n_s\mathbb{Z}\subseteq \tfrac{1}{2}+\mathbb{Z}$ and $Z_{\pm2\epsilon_1}(\Psi)=\pm p_{2\epsilon_1}+n_\ell\mathbb{Z}\subseteq 2\mathbb{Z}$.}$$
for some $p_{\epsilon_1}\in \tfrac{1}{2}+\mathbb{Z}$ and $p_{2\epsilon_1}\in 2\mathbb{Z}$.
As $\Psi$ is closed and $p_{2\epsilon_1}+n_\ell \mathbb{Z}\subseteq 2\mathbb{Z}$,  we have $$(p_{2\epsilon_1}-p_{\epsilon_1})+n_\ell \mathbb{Z}\subseteq p_{\epsilon_1}+n_s\mathbb{Z}\ \text{and hence} \  
p_{2\epsilon_1}+n_\ell \mathbb{Z}\subseteq (2p_{\epsilon_1}+n_s\mathbb{Z})\cap 2\mathbb{Z}.$$
From this we conclude that $n_s$ must be an odd integer since $2p_{\epsilon_1}$ is an odd integer. Since for all $r\in \mathbb{Z}$ such that 
$2p_{\epsilon_1}+n_sr\in 2\mathbb{Z}$, we have $2p_{\epsilon_1}+n_sr\in Z_{2\epsilon_1}(\Psi)$. This implies
$$p_{2\epsilon_1}+n_\ell \mathbb{Z}= (2p_{\epsilon_1}+n_s\mathbb{Z})\cap 2\mathbb{Z}=(2p_{\epsilon_1}+n_s)+2n_s\mathbb{Z}.$$
This implies, we must have $n_\ell=2n_s$. So, $\Psi$ must be equal to $\Psi(k, n_s)$, where $k=p_{\epsilon_1}-\tfrac{1}{2}\in\mathbb{Z}_+$ and  $n_s$ is an odd integer. 
One can easily see that $\Psi(k, n_s)$ is maximal if and only if $n_s$ is an odd prime number.

\vskip 1.5mm
\noindent
Case (2).
Now,  let $\Gr(\Psi)=\{\pm\epsilon_1\}$. Then we claim that  $\Psi=\{\pm (\epsilon_1+(2r+\tfrac{1}{2})\delta):r\in \mathbb{Z}\}$ or $\{\pm (\epsilon_1+(2r+\tfrac{3}{2})\delta):r\in \mathbb{Z}\}$.
Suppose $\pm (\epsilon_1+(r+\tfrac{1}{2})\delta), \pm (\epsilon_1+(s+\tfrac{1}{2})\delta)\in \Psi$ for some $r, s\in\mathbb{Z}$, then we claim that $r$ and $s$ are of the same parity.
If they have different parity, then $(r+s+1)\in 2\mathbb{Z}$ which implies that  $\pm2\epsilon_1\in \Gr(\Psi)$, a contradiction. This proves that $$\text{either 
$\Psi\subseteq\{\pm (\epsilon_1+(2r+\tfrac{1}{2})\delta):r\in \mathbb{Z}\}$ or $\Psi\subseteq\{\pm (\epsilon_1+(2r+\tfrac{3}{2})\delta):r\in \mathbb{Z}\}$.}$$
Since both sets on the right hand side are closed in $\Phi$, we get the equality.
Now,  we prove that both sets $\{\pm (\epsilon_1+(2r+\tfrac{1}{2})\delta):r\in \mathbb{Z}\}$ and $\{\pm (\epsilon_1+(2r+\tfrac{3}{2})\delta):r\in \mathbb{Z}\}$ are maximal closed in $\Phi.$
Let $\Delta\le \Phi$ be a closed subroot system such that either
$$\{\pm(\epsilon_1+(2r+\tfrac{1}{2})\delta):r\in\mathbb{Z}\}\subsetneq \Delta \ \text{or} \ \{\pm (\epsilon_1+(2r+\tfrac{3}{2})\delta):r\in \mathbb{Z}\}\subsetneq\Delta.$$
This implies that  $\{\pm \epsilon_1\}\subseteq \Gr(\Delta)$ and hence either $\Gr(\Delta)=\{\pm \epsilon_1\}$ or $\Gr(\Delta)=\{\pm \epsilon_1, \pm2\epsilon_1\}$. 
If $\Gr(\Delta)=\{\pm \epsilon_1\}$, then by previous argument, we get $$\text{either 
$\Delta\subseteq\{\pm (\epsilon_1+(2r+\tfrac{1}{2})\delta):r\in \mathbb{Z}\}$ or $\Delta\subseteq\{\pm (\epsilon_1+(2r+\tfrac{3}{2})\delta):r\in \mathbb{Z}\}$,}$$
which is not possible. So, we must have $\Gr(\Delta)=\{\pm \epsilon_1, \pm2\epsilon_1\}$. 
Then from the proof of Case (1) we get $\Delta=\Psi(k, q)$ for some $k\in \mathbb{Z}_+$ and an odd integer $q\in \mathbb{Z}$.
But since $$\{\pm(\epsilon_1+(2r+\tfrac{1}{2})\delta):r\in\mathbb{Z}\}\subsetneq \Delta \ \text{or} \ \{\pm (\epsilon_1+(2r+\tfrac{3}{2})\delta):r\in \mathbb{Z}\}\subsetneq\Delta.$$
we have either $\tfrac{1}{2}+2\mathbb{Z}\subseteq k+\tfrac{1}{2}+ q\mathbb{Z}$ or $\tfrac{3}{2}+2\mathbb{Z}\subseteq k+\tfrac{1}{2}+ q\mathbb{Z}$  which implies that  
$2\mathbb{Z}\subseteq q\mathbb{Z}$. This implies that  $q=1$ and $\Delta=\Phi$. 

\vskip 1.5mm
\noindent
Case (3).
Finally assume that $\Gr(\Psi)=\{\pm2\epsilon_1\}$. Then it is clear that 
$\Psi\subseteq \{\pm(2\epsilon_1+2r\delta): r \in \mathbb{Z} \}$. Since $\{\pm(2\epsilon_1+2r\delta): r \in \mathbb{Z} \}$ is closed, we have 
$\Psi=\{\pm(2\epsilon_1+2r\delta): r \in \mathbb{Z} \}$.
Conversely,   $\{\pm(2\epsilon_1+2r\delta): r \in \mathbb{Z} \}$ must be closed in $\Phi$.
Let $\Delta$ be a closed subroot system of $\Phi$ such that $\{\pm(2\epsilon_1+2r\delta): r \in \mathbb{Z} \}\subsetneq \Delta$. Then we have
 $\{\pm2\epsilon_1\}\subseteq \Gr(\Delta)$ and it immediately implies that $\Gr(\Delta)=\{\pm \epsilon_1, \pm2\epsilon_1\}$ as $\{\pm(2\epsilon_1+2r\delta): r \in \mathbb{Z} \}\subsetneq \Delta$.
Then from the proof of Case (1) we get $\Delta=\Psi(k, q)$ for some $k\in \mathbb{Z}_+$ and an odd integer $q\in \mathbb{Z}$.
This implies that  $2\mathbb{Z}\subseteq 2k+1+q+2q\mathbb{Z}$ which implies that  $2\mathbb{Z}\subseteq 2q\mathbb{Z}$. Since $q$ is an odd integer, we get $q=1$ and $\Delta=\Phi.$
This completes the proof.

 \end{pf}
\subsection{}\label{tabletwisted}
Now,  we are ready to state our final classification theorem for irreducible twisted affine root systems.
\begin{table}[ht]
\caption{Types of maximal subroot system of irreducible twisted affine root systems}
\centering 
\begin{tabular}{|c|c|c|}
\hline
Type & With closed gradient & With semi-closed gradient \\
\hline 
$\tt A_{2}^{(2)}$   & $\tt A_{2}^{(2)}$ &$ A_1^{(1)}$  \\[1ex]
$\tt A_{2n}^{(2)}$ & $\tt A_{2r}^{(2)} \oplus A_{2n-2r-1}^{(2)} \; (1\le r\le n-1),\;$  $\tt A_{2n}^{(2)}$, $\tt A_{2n-1}^{(2)}$  & $\tt B_n^{(1)}$\\[1ex]
$\tt D_{n+1}^{(2)}$ & $\tt D_{r+1}^{(2)} \oplus D_{n-r}^{(1)}\;$ $\tt (1\le r\le n-2)$, $\tt B_{n}^{(1)}$, $\tt D_{n+1}^{(2)}$, $\tt D_{n}^{(2)}$ & $\tt B_r^{(1)} \oplus B_{n-r}^{(1)}\;$ $\tt (2\le r\le n-2)$\\[1ex]
$\tt A_{2n-1}^{(2)}$ & $\tt A_{2r-1}^{(2)} \oplus A_{2n-2r-1}^{(2)}\;$ $\tt (1\le r\le n-1),$   $\tt A_{2n-1}^{(2)}$, $\tt C_n^{(1)}$, $\tt A_{n-1}^{(1)}$ & $\tt D_{n}^{(1)}$\\[1ex]
$\tt E_6^{(2)}$ & $\tt A_1^{(1)} \oplus A_{5}^{(2)}$ , $\tt A_{2}^{(1)} \oplus A_{2}^{(1)}$ $\tt E_6^{(2)}$, $\tt F_4^{(1)}$, $\tt D_5^{(2)}$ & $\tt C_4^{(1)}$\\[1ex]
$\tt D_4^{(3)}$ & $\tt A_1^{(1)} \oplus A_{1}^{(1)}$ , $\tt D_4^{(3)}$, $\tt G_2^{(1)}$, $\tt A_2^{(1)}$ &$\tt A_2^{(1)}$\\[1ex]

\hline
\end{tabular}
\label{mrs2}
\end{table}

We end this section with the following remark.
\begin{rem}\label{differences}
As we pointed out in the introduction the authors of \cite{FRT} have omitted a few possible cases in their classification list for the twisted case.
We list out all the differences between our classification list and their classification list.
The following possible cases are omitted in twisted case, see \cite[Table 1, Table 2, Theorem 5.8]{FRT}:
\begin{itemize}
 \item $\tt A_{2}^{(1)} \oplus \tt A_{2}^{(1)}\subset \tt E_6^{(2)}$
 \item $\tt D_{5}^{(2)}\subset \tt E_6^{(2)}$
\item $\tt B_{r}^{(1)} \oplus \tt B_{n-r}^{(1)}\subset \tt D_{n+1}^{(2)}$
 \item $\tt D_{n}^{(1)}\subset \tt A_{2n-1}^{(2)}$
\end{itemize}
The root systems of type $\tt D_r^{(1)}\oplus A_{2n-2r}^{(2)}$ does not occur as a maximal closed subroot 
system in $\tt A_{2n}^{(2)}$, in contrast to what is stated in \cite[Table 2]{FRT}. 
\end{rem}


\section{Closed subroot systems and Regular subalgebras}\label{closed}
In this section we will describe a procedure to classify all the regular subalgebras of affine Kac--Moody subalgebras both in untwisted and twisted case.
We follow the same notations as in the preliminary section.
\subsection{}
Recall that $\Phi$ denotes the set of real roots of the affine Lie algebra $\lie g$ and $\Delta(\lie g)$ denotes the roots of $\lie g.$
We will record the following fact from \cite[Remark 3.1]{DVmax}. It is fairly standard, but we give a proof for this fact for completeness.
\begin{lem}\label{dvmax}
 Let $\Psi$ be a closed subset of $\Phi$ such that $\Psi=-\Psi$ and $\bold s_\a(\beta)\in \Psi$ for all $\a, \beta\in \Psi$ with $\beta\pm \a\in \Delta_{\mathrm{im}}(\lie g)$
 or $\beta\pm2 \a\in \Delta_{\mathrm{im}}(\lie g)$. Then $\Psi$ must be a closed subroot system of $\Phi.$
\end{lem}
\begin{pf}
We only need to prove that $\Psi$ is a subroot system. Note that all root strings in $\Phi$ are unbroken.
Let $\alpha,\beta\in \Psi$ such that $\langle\beta,\alpha^{\vee}\rangle\in\bz_+$. 
If $\beta-s\alpha\in \Delta_{\mathrm{im}}(\lie g)$ for some $s\in\mathbb{Z}_+$ 
we must have $s\in\{1,2\}$ and hence $\bold s_{\alpha}(\beta)\in\Psi$. 
Otherwise $\beta-s\alpha\in \Phi$ for all $0\leq s\leq \langle\beta,\alpha^{\vee}\rangle$. 
Since $-\alpha\in\Psi$ we get by the closedness of $\Psi$ that $\beta-s\alpha\in \Psi$.
Thus $\bold s_{\alpha}(\beta)\in\Psi$. 
The case $-\langle\beta,\alpha^{\vee}\rangle\in\bz_+$ works similarly and we omit the details.
\end{pf}

\begin{lem}\label{keylemmaequiv}
 Let $\lie g'$ be a $\lie h$--invariant subalgebra of $\lie g$ and let $\Delta(\lie g')\subseteq \Delta(\lie g)$ be the set of roots of $\lie g'$ with respect to $\lie h$.
 Let $\Psi(\lie g') = \Delta(\lie g')\cap \Phi$ be the set of real roots of $\lie g'.$
 Suppose $\Delta(\lie g')=-\Delta(\lie g')$, then $\Psi(\lie g')$ must be a closed subroot system of $\Phi.$
\end{lem}
\begin{pf}
First recall that $\mathrm{dim}(\lie g_\a)=1$ for all $\a\in \Phi.$ Since $\Phi=-\Phi$, we have $\Psi(\lie g')=-\Psi(\lie g')$.
 Suppose $\a, \beta\in \Psi(\lie g')$ and $\a+\beta\in \Phi$ then it is immediate that $\a+\beta\in \Psi(\lie g')$, since $[\lie g_\a, \lie g_\beta]=\lie g_{\a+\beta}$. This implies 
 $\Psi(\lie g')$ is closed in $\Phi.$ So, by Lemma \ref{dvmax} it remains to prove that, $\bold s_\a(\beta)\in \Psi$ for all $\a, \beta\in \Psi(\lie g')$
 with $\beta\pm \a\in \Delta_{\mathrm{im}}(\lie g)$
 or $\beta\pm2 \a\in \Delta_{\mathrm{im}}(\lie g)$.
 
Case (1). Assume that $\lie g$ is not of type $\tt A^{(2)}_{2n}$. Let $\a, \beta\in \Psi(\lie g')$ such that $\beta=\a+r\delta$ for some $r\in \mathbb{Z}$. We have $\bold s_\a(\beta)=-\a+r\delta.$
 The finite dimensional subspace $$\text{ $V=\lie g_{\a+r\delta}\oplus \lie g_{r\delta}\oplus \lie g_{-\a+r\delta}\subseteq \lie g$
is a $\mathfrak{sl}_2=\lie g_\a\oplus [\lie g_\a, \lie g_{-\a}]\oplus \lie g_{-\a}$--module}$$
since $[\lie g_\a, \lie g_{\a+r\delta}]=0$ and $[\lie g_{-\a}, \lie g_{-\a+r\delta}]=0$
and it decomposes as $V\cong_{\mathfrak{sl}_2} V(2)\oplus V(0)^{\oplus k}$,
where $V(\lambda)$ denotes the finite dimensional irreducible $\mathfrak{sl}_2$--module corresponding
to the non--negative integer $\lambda\in \mathbb{Z}_+$ and $k=\mathrm{dim}(\lie g_{r\delta})-1$.
In particular, we have $[\lie g_\beta, \lie g_{-\a}]\neq 0$ and $[[\lie g_\beta, \lie g_{-\a}], \lie g_{-\a}]=\lie g_{-\a+r\delta}=\lie g_{\bold s_\a(\beta)}$, since $\mathrm{dim}(\lie g_{\bold s_\a(\beta)})=1$.
Since $\lie g_\beta, \lie g_{-\a}\subseteq \lie g'$, we have $\lie g_{\bold s_\a(\beta)}\subseteq \lie g'.$   This implies $\bold s_\a(\beta)\in \Psi(\lie g')$.
Similarly we get $\bold s_\a(\beta)\in \Psi(\lie g')$ if $\beta=-\a+r\delta.$

Case (2). Assume that $\lie g$ is of type $\tt A^{(2)}_{2n}$. Let $\a, \beta\in \Psi(\lie g')$ such that $\beta=2\a+r\delta$ for some $r\in \mathbb{Z}$. 
We have $\bold s_\a(\beta)=-2\a+r\delta.$
 The finite dimensional subspace $$\text{ $V=\lie g_{2\a+r\delta}\oplus\lie g_{\a+r\delta}\oplus \lie g_{r\delta}\oplus \lie g_{-\a+r\delta}\oplus \lie g_{-2\a+r\delta}\subseteq \lie g$
is a $\mathfrak{sl}_2=\lie g_\a\oplus [\lie g_\a, \lie g_{-\a}]\oplus \lie g_{-\a}$--module}$$ and it decomposes as $V\cong_{\mathfrak{sl}_2} V(4)\oplus V(0)^{\oplus k}$,
where $k=\mathrm{dim}(\lie g_{r\delta})-1$.
In particular, we have $[\lie g_\beta, \lie g_{-\a}]=\lie g_{\a+r\delta}\subseteq \lie g'$ and
$\lie g'\supseteq [[\lie g_\beta, \lie g_{-\a}], \lie g_{-\a}]=[\lie g_{\a+r\delta}, \lie g_{-\a}]\neq 0$ and 
$$\lie g'\supseteq [[\lie g_\beta, \lie g_{-\a}], \lie g_{-\a}], \lie g_{-\a}]=[[\lie g_{\a+r\delta}, \lie g_{-\a}], \lie g_{-\a}]=\lie g_{-\a+r\delta},$$ 
since $\mathrm{dim}(\lie g_{-\a+r\delta})=1$ and $\lie g_\beta, \lie g_{-\a}\subseteq \lie g'$. This immediately implies that $\lie g_{\bold s_\a(\beta)}=[\lie g_{-\a+r\delta}, \lie g_{-\a}]\subseteq \lie g'.$
Hence we have $\bold s_\a(\beta)\in \Psi(\lie g')$.
The cases $\beta=\pm\a+r\delta$ or $-2\a+r\delta$ and $\lie g$ is of type $\tt A^{(2)}_{2n}$ follows using similar ideas, so we omit the details.

\end{pf}

In \cite{Dynkin}, E. B. Dynkin introduced a notion of regular semi-simple subalgebras
in order to classify all the semi-simple subalgebras of finite dimensional complex semi-simple Lie algebras. As a natural generalization of Dynkin's definition,
one can give a constructive definition of regular subalgebras in the context of affine Kac--Moody algebras as well 
(see for example \cite{FRT}).
\begin{defn}
Let $\Psi$ be a closed subroot system of $\Phi$.
The subalgebra $\lie g(\Psi)$ of $\lie g$ generated by $\lie g_\a$, for $\a\in \Psi$, is called the regular subalgebra associated with $\Psi.$
 \end{defn}

One can easily see that the definition of regular subalgebras works well for all
Kac--Moody algebras. Clearly $\lie g(\Psi)$ is invariant under the adjoint action of $\lie h$ (the Cartan subalgebra of $\lie g$). 
Moreover we have,
$$\lie g(\Psi)=\lie h(\Psi)\oplus \bigoplus\limits_{\a\in \Delta(\lie g)} (\lie g_\alpha\cap \lie g(\Psi)),$$
where $\lie h(\Psi)=\text{$\mathbb{C}-$span of $\{\alpha^\vee: \alpha \in \Psi\}$}$. 
Denote the roots of $\lie g(\Psi)$ with respect to $\lie h$ by $\Delta(\Psi):=\{\alpha\in \Delta(\lie g): \lie g_\alpha\cap \lie g(\Psi)\neq 0\}$. Then it is immediate that
$\Psi\subseteq \Delta(\Psi)\cap \Phi$. Note that for real roots $\alpha$, we have $\lie g_\alpha\cap \lie g(\Psi)=\lie g_\alpha$, but for imaginary roots we may not necessarily have equality.
As we have mentioned in the introduction, we have a bijective correspondence between regular subalgebras and closed subroot systems of $\Phi.$
We need the following proposition in order to prove this bijective correspondence.
\begin{prop}\label{sumofroots}
Let $\Psi$ be a closed subroot system of $\Phi$ and let $\Psi=\Psi_1\oplus \cdots \oplus \Psi_k$ be its direct sum decomposition of irreducible components.
Let $\beta\in \Delta(\Psi)$, then there exists $\beta_1,\cdots, \beta_r\in \Psi$ such that the following holds:
 \begin{itemize}
  \item[(1)] $\beta=\beta_1+\cdots +\beta_r$ and we have $\beta_1+\cdots +\beta_i\in \Delta(\Psi)$, for each $1\le i\le r$.
  \item[(2)] There exists $1\le i_0\le k$ such that $\beta_1,\cdots, \beta_r\in \Psi_{i_0}.$
  \item[(3)] Suppose $\beta_1+\cdots +\beta_i\in \Delta(\Psi)\cap \Phi$ for some $1\le i\le r$, then we get $\beta_1+\cdots +\beta_i\in\Psi_{i_0}$.
 \end{itemize}
\end{prop}
\begin{pf}
Since $\lie g_\a$, $\a\in \Psi$ generates $\lie g(\Psi)$, it is easy to see that
the right normed Lie words $$\{[x_{\beta_r},[x_{\beta_{r-1}},[\cdots, [x_{\beta_2}, x_{\beta_1}]]\in \lie g(\Psi): \beta=\beta_1+\cdots +\beta_r, \beta_i\in \Psi, 1\le i\le r, r\in \mathbb{N}\}$$
spans $\lie g(\Psi)_\beta$. Thus if $\beta\in \Delta(\Psi)$, 
then there exists $r\in \mathbb{N}$ and $\beta_i\in \Psi, 1\le i\le r$, such that $\beta=\beta_1+\cdots +\beta_r$ and
the right normed Lie word $[x_{\beta_r},[x_{\beta_{r-1}},[\cdots, [x_{\beta_2}, x_{\beta_1}]]\neq 0$ for some $x_{\beta_i}\in \lie g_{\beta_i}, \ 1\le i\le r$.
Fix these $x_{\beta_i}$'s.
Now it is easy to see that 
$[x_{\beta_r},[x_{\beta_{r-1}},[\cdots, [x_{\beta_2}, x_{\beta_1}]]\neq 0$ only if
$$ \text{$x_{\beta_1}\neq 0$ and 
$[x_{\beta_i},[\cdots, [x_{\beta_2}, x_{\beta_1}]]\neq 0$ for all $2\le i\le r$}$$
and hence we have $\beta_1+\cdots +\beta_i\in \Delta(\Psi), 1\le i\le r$.
This completes the proof of Statement $(1)$.

 \vskip 1.5mm
To prove Statement $(2)$ and $(3)$, first observe that the irreducible components $\Psi_1, \cdots, \Psi_k$ of $\Psi$ are closed in $\Phi.$

 \vskip 1.5mm
Case (1).
Suppose $\beta_1+\cdots +\beta_i\in \Phi$ for all $1\le i\le r$, then
the Statement $(2)$ and $(3)$ follows from induction and the fact that $\alpha+\beta\notin \Phi$ if $\alpha\in \Psi_p$ and
$\beta\in \Psi_q$ for $1\le p\neq q\le k$. In this case we have, $\beta_1\in \Psi_{i_0}\implies \beta_1,\cdots, \beta_r\in \Psi_{i_0}$ and $\beta_1+\cdots +\beta_i\in\Psi_{i_0}$ 
for all $1\le i\le r$. 

 \vskip 1.5mm
Case (2).
Suppose $\beta_1+\cdots +\beta_i\notin \Phi$ for some $2\le i\le r$. Let $i\in \{1,\cdots r\}$ be the minimum such that $\beta_1+\cdots +\beta_i\notin\Phi$, in particular we have
$\beta_1+\cdots +\beta_j\in\Phi$ for all $1\le j<i$. Then by previous argument, there exists
$i_0\in \{1,\cdots, k\}$ such that $\beta_1, \cdots, \beta_{i-1}\in \Psi_{i_0}$ and $\beta_1+\cdots +\beta_j\in\Psi_{i_0}$ for all $1\le j<i$.
Write $\beta_1+\cdots +\beta_{i-1}=\alpha+s\delta\in \Psi_{i_0}$, where $\alpha\in \Gr(\Psi_{i_0})$. Since $\beta_1+\cdots +\beta_{i}\notin \Phi$, we must have $\beta_{i}=-\alpha+s'\delta.$ Observe that 
$(\beta_1+\cdots +\beta_{i-1}, \beta_i)=-(\a, \a)\neq 0$. So we immediately get $\beta_{i}=-\alpha+s'\delta\in \Psi_{i_0}$ and
$\beta_1+\cdots +\beta_{i-1}+\beta_i=(s+s')\delta$. Suppose $\beta_{i+1}=\beta+s''\delta\notin \Psi_{i_0}$ then we get
$[x_{\beta+s''\delta}, x_{\alpha+s\delta}]=0$ and $[x_{\beta+s''\delta}, x_{-\alpha+s'\delta}]=0$ as
$(\beta+s''\delta)+(\alpha+s\delta)\notin\Delta(\Psi)$ and $(\beta+s''\delta)+(-\alpha+s'\delta)\notin\Delta(\Psi)$. 
This immediately implies that
$$[x_{\beta_{i+1}},[x_{\beta_{i}},[x_{\beta_{i-1}},[\cdots, [x_{\beta_2}, x_{\beta_1}]]=[x_{\beta+s''\delta}, [x_{\alpha+s\delta}, x_{-\alpha+s'\delta}]]=0$$
which is a contradiction to our choice of $x_{\beta_{1}}, \cdots, x_{\beta_{i+1}}$. Thus we must have $\beta_{i+1}=\beta+s''\delta\in \Psi_{i_0}$.
Now induction completes the proof of Statement $(2)$.

 \vskip 1.5mm
We only need to prove that $\beta_1+\cdots +\beta_{i}+\beta_{i+1}\in \Psi_{i_0}$ in order to complete the proof of Statement $(3)$. 
First
recall from the Proposition \ref{keypropositionnon2A2n} that there exists $n_\alpha\in \mathbb{Z}$
for $\alpha\in \Gr(\Psi)$ such that
$Z_\alpha(\Psi_{i_0})=p_\alpha+n_\alpha\mathbb{Z}$.

 \vskip 1.5mm
Case (2.1).
Suppose $n_\a=0$ for some $\a\in \Gr(\Psi_{i_0})$, then $n_\beta=0$ for all $\beta\in \Gr(\Psi_{i_0})$ by Lemma \ref{nalphaneqzero}.
Then we have $\beta_1+\cdots +\beta_j\in \Phi$ for all $1\le j\le r$ in this case, so the Statement $(3)$ is immediate in this case.

 \vskip 1.5mm
Case (2.2).
So assume that $n_\a\neq 0$ for all $\a\in \Gr(\Psi_{i_0})$.
Write $\beta_1+\cdots +\beta_{i-1}=\alpha+(p_\alpha+n_\alpha k_\alpha)\delta,\ \beta_i=-\alpha+(-p_\alpha+n_\alpha k_\alpha')\delta$
and $\beta_{i+1}=\beta+(p_\beta+n_\beta k_\beta)\delta$
Then we have $\beta_1+\cdots +\beta_{i}=n_\alpha (k_\alpha+k_\alpha')\delta.$ We need to prove that 
$\beta_1+\cdots +\beta_{i}+\beta_{i+1}=\beta+(p_\beta+n_\beta k_\beta + n_\alpha (k_\alpha+k_\alpha'))\delta$ must be in $\Psi_{i_0}.$  
 \vskip 1.5mm
Case (2.2.1). Assume that $\Phi$ is not of type $\text{$\tt A^{(2)}_{2n}$}$.
 Suppose both $\alpha$ and $\beta$ are long or short then we have $n_\alpha=n_\beta$ by Lemma \ref{nalpha}, hence 
 $\beta_1+\cdots +\beta_{i}+\beta_{i+1}\in \Psi_{i_0}$, since 
 $Z_\beta(\Psi_{i_0})=p_\beta+n_\alpha\mathbb{Z}$. If $\beta$ is short and $\alpha$ is long then we have $n_\beta=n_\alpha$ or $n_\alpha=mn_\beta$ by Statement (2) of 
 Proposition \ref{twisted01}, hence we have $\beta_1+\cdots +\beta_{i}+\beta_{i+1}\in \Psi_{i_0}$. Now assume that $\a$ is short and $\beta$ is long then we have 
 $n_\beta=n_\alpha$ if $m|n_\a$ and  $n_\beta=mn_\alpha$ if $m\nmid n_\a$. Again the claim follows easily when $n_\beta=n_\alpha$. So we are left with case $n_\beta=mn_\alpha$.
 Recall that $m=2\ \text{or}\ 3$ in this case, so it is prime number.
 Now note that $p_\beta+n_\beta k_\beta \equiv 0 \ (mod\ m)$ and $(p_\beta+n_\beta k_\beta + n_\alpha (k_\alpha+k_\alpha'))\equiv 0 \ (mod\ m)$ together implies,
 $n_\alpha (k_\alpha+k_\alpha')\equiv 0 \ (mod\ m).$ Since $m\nmid n_\alpha$, we get $k_\alpha+k_\alpha'\equiv 0 \ (mod\ m)$. This implies we have
 $n_\alpha (k_\alpha+k_\alpha')\equiv 0 \ (mod\ n_\beta)$ and hence we have $\beta_1+\cdots +\beta_{i}+\beta_{i+1}\in\Psi_{i_0}$.
 This completes the proof of Statement $(3)$ in this case. 

  \vskip 1.5mm
Case (2.2.2). Assume that $\Phi$ is of type $\text{$\tt A^{(2)}_{2n}$.}$
 Suppose both $\alpha$ and $\beta$ are long or short or intermediate then we have $n_\alpha=n_\beta$ by Lemma \ref{nalpha}, hence 
 $\beta_1+\cdots +\beta_{i}+\beta_{i+1}\in \Psi_{i_0}$, since 
 $Z_\beta(\Psi_{i_0})=p_\beta+n_\alpha\mathbb{Z}$. If $\beta$ is short (resp. intermediate) and $\alpha$ is intermediate (resp. short) then we have $n_\beta=n_\alpha$ by 
 Proposition \ref{2A2n01}, hence we have $\beta_1+\cdots +\beta_{i}+\beta_{i+1}\in \Psi_{i_0}$.  
 If $\beta$ is short or intermediate and $\alpha$ is long then we have $n_\beta=n_\alpha$ or $n_\alpha=2n_\beta$ by
 Proposition \ref{twisted01}, hence we have $\beta_1+\cdots +\beta_{i}+\beta_{i+1}\in \Psi_{i_0}$. 
 Now assume that $\a$ is short or intermediate and $\beta$ is long then we have 
 $n_\beta=n_\alpha$ if $2|n_\a$ and  $n_\beta=2n_\alpha$ if $m\nmid n_\a$. Again the claim follows easily when $n_\beta=n_\alpha$. So we are left with case $n_\beta=2n_\alpha$.
 Now note that $p_\beta+n_\beta k_\beta \equiv 0 \ (mod\ 2)$ and $(p_\beta+n_\beta k_\beta + n_\alpha (k_\alpha+k_\alpha'))\equiv 0 \ (mod\ 2)$ together implies,
 $n_\alpha (k_\alpha+k_\alpha')\equiv 0 \ (mod\ 2).$ Since $2\nmid n_\alpha$, we get $k_\alpha+k_\alpha'\equiv 0 \ (mod\ 2)$. This implies we have
 $n_\alpha (k_\alpha+k_\alpha')\equiv 0 \ (mod\ n_\beta)$ and hence we have $\beta_1+\cdots +\beta_{i}+\beta_{i+1}\in\Psi_{i_0}$.
 This completes the proof of Statement $(3)$ in this case. 

 \end{pf}

\begin{cor}\label{1-1}
Let $\Psi$ be a closed subroot system of $\Phi$ and let $\Delta(\Psi)$ be the set of roots of $\lie g(\Psi)$ with respect to $\lie h$. Then we have
$\Psi=\Delta(\Psi)\cap \Phi.$ Thus the map $\Psi\mapsto \lie g(\Psi)$ is a one-to-one correspondence between the set of closed subroot systems of $\Phi$ and the set of regular
subalgebras of $\lie g.$  
\end{cor}
\begin{pf}
Immediate from Proposition \ref{sumofroots}.
\end{pf}
\subsection{}\label{equivdefregular} E. B. Dynkin showed that linearly independent $\pi$-systems arise precisely as simple systems of 
regular subalgebras of finite dimensional semi-simple algebras. So it is natural to expect to define regular subalgebras 
in terms of $\pi-$systems in our context. 
Now we give equivalent definition of regular subalgebras in terms of $\pi-$systems.
A $\pi-$system $\Sigma$ is a finite subset of $\Phi$
satisfying the property that for each $\alpha, \beta\in \Sigma$, we have $\alpha-\beta$ is not a root (i.e.,  $\alpha-\beta\notin \Delta(\lie g)$).
Note that we do not demand $\Sigma$ to be linearly independent in the definition of $\pi-$systems.
Let $\lie g(\Sigma)$ be the subalgebra of $\lie g$ generated by $\{\mathfrak{g}_{\alpha} :\alpha\in\Sigma \cup (-\Sigma)\}$
and let  $\Delta(\Sigma)$ be the set of roots of $\lie g(\Sigma)$ with respect to $\lie h.$
Denote by $W_{\Sigma}$ the Weyl group generated by the reflections $\{\bold s_\a : \a\in \Sigma\}$. 
We refer to \cite{LBKKV} for more details and historical remarks about $\pi-$systems. 
We have a natural choice of $\pi-$system for each closed subroot system of $\Phi$. 
\begin{lem}\label{keylemmapisys}
 Let $\Psi$ be a closed subroot system of $\Phi$ and let $\Psi=\Psi_1\oplus \cdots \oplus \Psi_k$ be its direct sum decomposition of irreducible components.
Let $\Sigma_i$ be a simple system of $\Psi_i$ for each $1\le i\le k.$ Then $\Sigma=\bigcup_{i=i}^{k}\Sigma_i$ is a $\pi-$system.
\end{lem}
\begin{pf}
Let $\alpha\in\Sigma_i$ and $\beta\in\Sigma_j$, we need show that $\a-\beta$ is not a root of $\lie g$. 
If $i=j$, then clearly $\alpha-\beta$ is not a root of $\lie g$. 
Assume that $i\neq j$ and $\alpha-\beta$ is a root. Since $(\alpha, \beta)=0$, we have $(\alpha-\beta,\alpha-\beta)>0$. 
So, $\alpha-\beta$ is a real root of $\lie g$. Since $\Psi$ is closed in $\Phi$, we have $\alpha-\beta\in \Psi$.
But we have $(\alpha-\beta,\alpha)>0$ and $(\alpha-\beta,\beta)<0$, which demands $\alpha-\beta\in\Sigma_i\cap\Sigma_j$. This is clearly a contradiction and it 
completes the proof.
\end{pf}

\subsection{}
Suppose $\Sigma$ is a $\pi-$system then $\Sigma\cup -\Sigma$ is closed under multiplication by $-1.$ So it motivates us to define symmetric subsets of real roots. More precisely, a subset $\Sigma_s$ of $\Phi$
is said to be symmetric if $\Sigma_s=-\Sigma_s.$
Let $\lie g(\Sigma_s)$ be the subalgebra of $\lie g$ generated by $\{\mathfrak{g}_{\alpha} :\alpha\in\Sigma_s\}$.
We are now ready to state our equivalent definitions of regular subalgebras of $\lie g.$

\begin{thm}\label{mainthmpisys}
Let $\lie g$ an affine Kac-Moody algebra and let $\lie g'$ be its subalgebra. Then the following definitions are equivalent:
\begin{enumerate}
 \item there exists a closed subroot system $\Psi$ of $\Phi$ such that $\lie g'=\lie g(\Psi)$,
 \item there exists a $\pi-$system $\Sigma$ of $\Phi$ such that $\lie g'=\lie g(\Sigma)$,
 \item there exists a symmetric subset  $\Sigma_s$ of $\Phi$ such that $\lie g'=\lie g(\Sigma_s)$.
\end{enumerate}

\end{thm}
\begin{pf}
 First assume that $\lie g'=\lie g(\Psi)$ for some closed subroot system $\Psi$ of $\Phi.$ Then by Lemma \ref{keylemmapisys}, we have the $\pi-$system $\Sigma$ which is a union of
 simple systems of corresponding irreducible components of $\Psi.$ Since $\Psi$ is reduced, we have $\Psi=W_\Sigma(\Sigma)$. 
 Since $\Delta(\Sigma)=-\Delta(\Sigma)$ and $\lie g(\Sigma)$ is $\lie h$--invariant, we have $\Psi=W_\Sigma(\Sigma)\subseteq \Delta(\Sigma)$ by Lemma \ref{keylemmaequiv}.
 This implies that
 $\lie g_\a\subseteq \lie g(\Sigma)$ for all $\a\in \Psi$, hence we have $\lie g(\Psi)\subseteq \lie g(\Sigma)$.
 Since $\Sigma\subseteq \Psi$, we have $\lie g(\Sigma)\subseteq \lie g(\Psi).$ So, we have the equality $\lie g(\Psi)=\lie g(\Sigma)$. This also implies that $\Delta(\Psi)=\Delta(\Sigma)$
and we have $\Delta(\Sigma)\cap \Phi=\Psi$ from Corollary \ref{1-1}. This proves $(1)$ implies $(2)$. The fact $(2)$ implies $(3)$ follows immediately if we take $\Sigma_s=\Sigma\cup -\Sigma.$ 

\vskip 1.5mm
Now we prove $(3)$ implies $(1)$. Suppose $\lie g'=\lie g(\Sigma_s)$ for some symmetric subset $\Sigma_s$ of $\Phi.$ 
It is easy to see that $\Delta(\Sigma_s)=-\Delta(\Sigma_s)$.
Let $\Psi=\Delta(\Sigma_s)\cap \Phi$. Again by Lemma \ref{keylemmaequiv}, $\Psi$ is a closed subroot system of $\Phi.$
 Clearly $\lie g(\Psi)\subseteq \lie g(\Sigma_s)$ since $\lie g_\a\subseteq \lie g(\Sigma_s)$ for all $\a\in \Psi$. Since $\Sigma_s\subseteq \Delta(\Sigma_s)\cap \Phi=\Psi$, we have
 $\lie g(\Sigma_s)\subseteq \lie g(\Psi).$ So, we have the equality $\lie g(\Psi)=\lie g(\Sigma_s)$. This completes the proof.
\end{pf}

\begin{cor}
 The association $\Sigma \mapsto \Delta(\Sigma)\cap \Phi$ gives a bijective correspondence between the set of $\pi-$systems of $\Phi$ and the closed subroot systems of $\Phi.$
\end{cor}

\begin{rem}
One can easily see that our definition of regular subalgebras 
is little different from the regular subalgebras which appears in \cite[Section 2]{Morita}, see \cite{Naito} for its generalization.
Suppose the closed subroot system has a simple system (i.e., the corresponding $\pi-$system is linearly independent) then our definition of 
 regular subalgebra matches up with the definition of Naito's, see \cite{Naito}, indeed in this case
 our regular subalgebra is the derived subalgebra of Naito's regular subalgebra which is a Kac--Moody algebra by definition.
Note that the closed subroot systems of an affine root system does not need to have simple systems in general.
 For example,  consider the affine root system $\Delta=\tt G_2^{(1)}$ and $\Delta_{\mathrm{re}}=\{\alpha+n \delta: \alpha \in \text{$\tt G_2$}, n \in\mathbb{Z}\}.$
Let $\{\a_1, \a_2\}$ be the simple system of $\tt G_2$, such that $\a_2$ is a short root. 
Then define $$\Psi=\{\pm\a_2+n\delta : n\in\mathbb{Z}\}\cup\{\pm\theta+n\delta: n\in\mathbb{Z}\},$$ where $\theta$ is the long root of $\tt G_2.$
 Clearly,  $\Psi$ is a closed subroot system of type $\tt A_1^{(1)}\oplus A_1^{(1)}$ 
 which has no linearly independent simple system by rank comparison.  So, here in this paper we are dealing with a much bigger class of subalgebras of affine Kac--Moody algebras.

 \end{rem}

\subsection{} We have the following explicit description for the closed subroot systems of untwisted affine root systems.
\begin{prop}\label{closeduntwisted}
 Let $\Phi$ be an untwisted affine root system. We have, $\Psi$ (does not need to be of affine type) is a closed subroot system of $\Phi$ if and only if 
 there exists \begin{itemize}
 \item  mutually orthogonal irreducible closed subroot systems 
 $\Psi_1, \cdots, \Psi_k$ of $\mathring\Phi$ and 
 \item  $n_i\in \mathbb{Z}$ and $\mathbb{Z}$--linear function $p^i : \Psi_i \to \mathbb{Z}, \ \a\mapsto p^i_\a$, satisfying the equation \ref{palpha}, for each $1\le i\le k$
      \end{itemize}
 such that
 \begin{equation}\label{closeduntwistedeq}
 \Psi = \wh{\Psi_1}\oplus \cdots \oplus \wh{\Psi_k}
 \end{equation}
where 
$\wh{\Psi_i}=\{\a+(p^i_\a+rn_i)\delta \in \Psi : \a\in \Psi_i, r\in \mathbb{Z}\}$, $1\le i\le k$.
The subroot system $\wh{\Psi_i}$ is of finite type if and only if the integer $n_i$ associated to  $\wh{\Psi_i}$  is zero. 
\end{prop}

\begin{pf}
Let $\Psi$ be a closed subroot system  of untwisted affine root system $\Phi$. 
Then by Proposition \ref{untwisted02}, we know that $\Gr(\Psi)$ is a closed subroot system of $\mathring\Phi$. Let 
$$\Gr(\Psi)= \Psi_1 \oplus \cdots \oplus \Psi_k$$ be the decomposition of $\Gr(\Psi)$ into irreducible components. Then 
each $\Psi_i$ is an irreducible finite subroot system of $\mathring\Phi$.
Since $\Gr(\Psi)$ is closed in $\Phi$, we see that each $\Psi_i$ is closed in $\mathring{\Phi}$. 
Let $\widehat{\Psi_i}$ denote the lift of $\Psi_i$ in $\Psi$.
Then for each $1\le i\le k$, by Proposition \ref{untwisted01}, there exists 
$n_i\in \mathbb{Z}$ and a $\mathbb{Z}$--linear function $p^i : \Psi_i \to \mathbb{Z}, \ \a\mapsto p^i_\a,$ satisfying the equation \ref{palpha}
such that for each $\alpha\in \Psi_i$, $Z_\alpha(\wh{\Psi_i})=p^i_\alpha+n_i\mathbb{Z}$.
This implies that  $\widehat{\Psi_i}=\{\a+(p^i_\a+rn_i)\delta \in \Psi : \a\in \Psi_i, r\in \mathbb{Z}\}$, $1\le i\le k$.
Notice that if $n_i=0$,
then the lift of $\Psi_i$ must be of finite type and the types of $\wh{\Psi_i}$ and $\Psi_i$ are same.  Converse part is straightforward. 
This completes the proof.

\end{pf}

\subsection{}
Let $\Phi$ be an affine root system and $\Psi$ be a closed subroot system of $\Phi$ as before. Write $\Psi=\Psi_a\oplus \Psi_f$ 
where  $\Psi_a$ (resp. $\Psi_f$) is the affine (resp. finite) part of $\Psi.$
Since $\Psi$ is closed, we have the subroot systems $\Psi_a$ and $\Psi_f$ are closed in $\Phi$ and $\mathring{\Phi}$ respectively.
Since we know the classification of all the closed subroot systems in the finite type (see \cite{Bor, Dynkin}), we only need
to classify all the closed subroot systems of $\Phi$ which are of affine type.
It can be done using the following theorem and the information about maximal closed subroot systems which appears in
sections $3-11$.

\begin{thm}\label{finitechain}
Let $\Phi$ be an affine root system and $\Psi$ be a closed subroot system in $\Phi$ of affine type. 
Then there exists a finite chain of closed subroot systems in $\Phi$, $\Phi=\Phi_0\supseteq \Phi_1 \supseteq \cdots \supseteq\Phi_k=\Psi$ 
such that $\Phi_i$ is maximal closed in $\Phi_{i-1}$ for $1\leq i\leq k$.
\end{thm}

First we fix a notation. For a closed subroot system $\Delta$ of $\Phi$ with decomposition into  indecomposable components $\Gr(\Delta)=\Delta_1\oplus \cdots \oplus \Delta_k$,
we denote by 
$$\mathrm{ht}(\Delta)=\sum\limits_{i=1}^{k}n_s^{\Delta_i}(\Delta)+\sum\limits_{i=1}^{k}n_{\mathrm{im}}^{\Delta_i}(\Delta)+\sum\limits_{i=1}^{k}n_\ell^{\Delta_i}(\Delta)$$
Here it is understood that $n_{\mathrm{im}}^{\Delta_i}(\Delta)=0$ if there is no intermediate roots and so on.
We need the following lemma to prove the theorem \ref{finitechain}. 
\begin{lem}\label{chainlemma}
 Let $\Phi$ be an affine root system and $\Psi\subsetneq \Delta\subseteq \Phi$ be closed subroot systems of $\Phi$ of affine type. 
 Let $\Gr(\Psi)= \Psi_1 \oplus \cdots \oplus \Psi_\ell$ be the decomposition of $\Gr(\Psi)$ into irreducible components.
 Then we have
\begin{itemize}
 \item[(i)] either $\Gr(\Psi)\subsetneq \Gr(\Delta)$ or
 \item[(ii)] $\Gr(\Psi)= \Gr(\Delta)$ and $\mathrm{ht}(\Delta)<\mathrm{ht}(\Psi).$

\end{itemize}
\end{lem}
\begin{proof}
Suppose $\Gr(\Psi)\subsetneq \Gr(\Delta)$, then there is nothing to prove. So, assume that $\Gr(\Psi)= \Gr(\Delta)$. It is easy to see that for each
$1\le i \le \ell$ we have $n_s^{\Psi_i}(\Delta)$ is a divisor of  $n_s^{\Psi_i}(\Psi)$, in particular $n_s^{\Psi_i}(\Delta)\le n_s^{\Psi_i}(\Psi)$.
Similarly,  we have $n_{\mathrm{im}}^{\Psi_i}(\Delta)\le n_{\mathrm{im}}^{\Psi_i}(\Psi)$ and $n_\ell^{\Psi_i}(\Delta)\le n_\ell^{\Psi_i}(\Psi)$ for all $1\le i \le \ell$. This immediately
implies that $\mathrm{ht}(\Delta)\le \mathrm{ht}(\Psi).$

If $\mathrm{ht}(\Delta)=\mathrm{ht}(\Psi)$, then we must have 
$n_s^{\Psi_i}(\Delta)= n_s^{\Psi_i}(\Psi)$, $n_{\mathrm{im}}^{\Psi_i}(\Delta)=n_{\mathrm{im}}^{\Psi_i}(\Psi)$ and $n_\ell^{\Psi_i}(\Delta)=n_\ell^{\Psi_i}(\Psi)$ for all $1\le i \le \ell$.
This implies that  $p_\a^{\Psi}+n_\a^\Psi\mathbb{Z}\subseteq p_\a^{\Delta}+n_\a^\Delta\mathbb{Z}$ for all $\a\in \Gr(\Psi)$. Since $n_\a^\Psi=n_\a^\Delta$, we get 
$p_\a^{\Psi}\equiv p_\a^{\Delta} (\mathrm{mod} \ n_\a^\Psi)$ for all $\a\in \Gr(\Psi).$
This immediately implies that
$\Psi$ must be equal to $\Delta$ which is a contradiction to the assumption.

\end{proof}

Theorem \ref{finitechain} is an immediate corollary of the following proposition.
\begin{prop}\label{finitechainprop}
Let $\Phi$ be an affine root system and $\Psi$ be a closed subroot system in $\Phi$ of affine type.
Then there is no infinite chain of closed subroot systems in $\Phi$, such that 
$$\Psi=\Phi_0\subsetneq \Phi_1 \subsetneq \cdots \subsetneq\Phi_k\subsetneq \Phi_{k+1}\subsetneq \cdots \subseteq\Phi.$$
\end{prop}
\begin{proof}
We prove this result by contradiction. Assume that there is an infinite chain of closed subroot systems in $\Phi$, such that 
$$\Psi=\Phi_0\subsetneq \Phi_1 \subsetneq \cdots \subsetneq\Phi_k\subsetneq \Phi_{k+1}\subsetneq \cdots \subseteq\Phi.$$
Then we have 
$\Gr(\Psi)=\Gr(\Phi_0)\subseteq \Gr(\Phi_1) \subseteq \cdots \subseteq\Gr(\Phi_k)\subseteq \cdots \subseteq\Gr(\Phi).$
Since $\Gr(\Phi)$ is finite, there must exists a $k\in\mathbb{Z}$ such that $\Gr(\Phi_k)=\Gr(\Phi_i)$ for all $i\geq k$. 
Since $\Phi_k\subsetneq\Phi_{j}\subsetneq \Phi_i$, by lemma \ref{chainlemma}, we have $$\mathrm{ht}(\Phi_i)<\mathrm{ht}(\Phi_j)<\mathrm{ht}(\Phi_k)\ \ \text{for all $k<j<i$}$$  which is absurd.
This completes the proof.
\end{proof}



\bibliographystyle{plain}
\bibliography{vk-bib}

\end{document}